\documentclass[12pt, leqno]{amsart}
\usepackage{latexsym}
\usepackage{amssymb}
\usepackage{amsmath}
\usepackage[english]{babel}
\usepackage[colorlinks=true, pdfstartview=FitV, linkcolor=red, citecolor=red, urlcolor=blue]{hyperref} 

\usepackage[top=1.25in, bottom=1.25in, left=1.25in, right=1.25in]{geometry}

\newtheorem{theorem}{Theorem}

\newtheorem{definition}[theorem]{Definition}

\newtheorem{lemma}[theorem]{Lemma}

\newtheorem{proposition}[theorem]{Proposition}
\newtheorem{remark}[theorem]{Remark}

\newtheorem{mytheorem}{Theorem}

\newcommand{\rn}[1]{\mathbb{R}^{#1}}

\newcommand{\re}{ \mathbb{R}}

\newcommand{\beq}{\begin{equation}}
\newcommand{\bea}[1]{\begin{array}{#1} }
\newcommand{\eeq}{ \end{equation}}
\newcommand{\ea}{ \end{array}}
\newcommand{\ep}{\epsilon}
\newcommand{\es}{\emptyset}

\newcommand{\al}{\alpha}
\newcommand{\ga}{\gamma}

\newcommand{\de}{\delta}
\newcommand{\ds}{\displaystyle}
\newcommand{\ts}{\textstyle}
\newcommand{\rar }{\mbox{$\rightarrow$}}

\newcommand{\ran}{\rangle}
\newcommand{\lan}{\langle}
\newcommand{\Ga}{\Gamma}
\newcommand{\la}{\lambda}
\newcommand{\La}{\Lambda}
\newcommand{\ar}{\partial}
\newcommand{\si}{\sigma}

\newcommand{\om}{\omega}
\newcommand{\Om}{\Omega}
\newcommand{\be}{\beta}

\newcommand{\ph}{\phi}
\newcommand{\he}{\theta}
\newcommand{\He}{\Theta}
\newcommand{\Ph}{\Phi}
\newcommand{\hs}[1]{\mbox{$ \hspace{#1}$}}

\newcommand{\sem}{\setminus}
\newcommand{\ze}{\zeta}
\newcommand{\De}{\Delta}
\newcommand{\ti}{\tilde}
\newcommand{\noi}{\noindent}

\begin{document} 

\title[The  Brunn-Minkowski Inequality and A Minkowski problem]{The  Brunn-Minkowski Inequality and A Minkowski problem for $\mathcal{A}$-harmonic Green's function}

\author[M. Akman]{Murat Akman}
\address{{\bf Murat Akman}\\ 
Department of Mathematics \\
University of Connecticut, Storrs, CT 06269-1009} 
\email{murat.akman@uconn.edu}
\urladdr{http://www.math.uconn.edu/~akman/}

\author[J. Lewis]{John Lewis}
\address{{\bf John Lewis} \\ Department of Mathematics \\ University of Kentucky \\ Lexington, Kentucky, 40506}
\email{johnl@uky.edu}
\urladdr{http://www.ms.uky.edu/~johnl/}

\author[O. Saari]{Olli Saari}
\address{{\bf Olli Saari}\\ Mathematisches Institut \\ Endenicher Allee 60, 53115 Bonn, Germany}
\email{saari@math.uni-bonn.de}
\urladdr{https://www.math.uni-bonn.de/people/saari/}

\author[A. Vogel]{Andrew Vogel}
\address{{\bf Andrew Vogel}\\ Department of Mathematics, Syracuse University \\  Syracuse, New York 13244}
\email{alvogel@syracuse.edu}


\begin{abstract}
In this article we study two classical problems in convex geometry associated to $\mathcal{A}$-harmonic PDEs, quasi-linear elliptic PDEs whose structure is  modeled on the  $p$-Laplace equation. Let $p$ be fixed with $2\leq n\leq p<\infty$. For a convex compact set $E$ in $\mathbb{R}^{n}$, we define and then prove the existence and uniqueness of the  so called $\mathcal{A}$-harmonic Green's function for the complement of $E$ with pole at  infinity. We then define a quantity $\mbox{C}_{\mathcal{A}}(E)$ which can be seen as the behavior of  this  function near infinity. 

In the first part of this article,  we prove that $\mbox{C}_{\mathcal{A}}(\cdot)$ satisfies the following Brunn-Minkowski type inequality
\[
\left[\mbox{C}_\mathcal{A}  ( \la  E_1 + (1-\la) E_2 )\right]^{\frac{1}{p-n}}  \geq  \la  \, 
\left[\mbox{C}_\mathcal{A}  (  E_1 )\right]^{\frac{1}{p-n}}  +  (1-\la)  \left[\mbox{C}_\mathcal{A}  (E_2 )\right]^{\frac{1}{p-n}}
\]
when  $n<p<\infty$, $0 \leq \la \leq 1$, and $E_1, E_2$ are nonempty convex compact sets in $\mathbb{R}^{n}$.  While $p=n$ then
\[
\mbox{C}_\mathcal{A}  ( \la  E_1 + (1-\la) E_2 ) \geq  \la  \, 
\mbox{C}_\mathcal{A}  (  E_1 )  +  (1-\la)  \mbox{C}_\mathcal{A}  (E_2 )
\]
where $0\leq \la\leq 1$ and $E_1, E_2$ are convex compact sets in $\mathbb{R}^{n}$ containing at least two points. 

Moreover, if equality holds  in the either of the  above inequalities for some $E_1$ and $E_2, $ then under certain  
regularity and structural assumptions on $\mathcal{A}$  we show that  these two sets are homothetic.

The   classical Minkowski problem asks for necessary and sufficient conditions on a non-negative 
Borel measure on the unit sphere $\mathbb{S}^{n-1}$ to be the surface area 
measure of a convex compact set in $\mathbb{R}^{n}$ under the   Gauss  mapping for the boundary of  this convex set. 
In the second part of this article we study a Minkowski type problem for a measure  
associated to the $\mathcal{A}$-harmonic Green's function for the complement of 
a convex compact set $E$ when $n\leq p<\infty$. If  $   \mu_E $ denotes this measure, 
then we show that  necessary and sufficient conditions for  existence  under this setting are 
exactly the same conditions as in the classical Minkowski problem. Using the Brunn-Minkowski 
inequality result from the first part, we also show that this problem has a unique solution up to translation.\\ \\
\end{abstract}

\maketitle
\setcounter{tocdepth}{3}
\tableofcontents

\part{The Brunn-Minkowski inequality for $\mathcal{A}$-harmonic Green's function}

\section{Introduction}
The classical Brunn-Minkowski inequality states that 
\begin{align}
\label{BMVolume}
\left[\mbox{Vol}(\lambda E_1+(1-\lambda) E_2)\right]^{\frac{1}{n}} 
\geq\lambda\left[\mbox{Vol}(E_1)\right]^{\frac{1}{n}} +(1-\lambda) \left[\mbox{Vol}(E_2)\right]^{\frac{1}{n}}
\end{align}
whenever $E_1, E_2$ are compact convex sets with nonempty interiors in $\mathbb{R}^{n}$ and $\lambda\in [0,1]$. Here $\mbox{Vol}(\cdot)$ denotes the usual volume in $\mathbb{R}^{n}$ and the summation $(\lambda E_1+(1-\lambda) E_2)$ should be understood as a vector sum(called  \textit{Minkowski addition}). Moreover, equality in \eqref{BMVolume} holds precisely when $E_1$ and $E_2$ are equal up to translation and dilatation (i.e. $E_1$ is homothetic to $E_2$). In fact, this inequality holds for bounded measurable sets as well and it is one of the most important and deepest results in the theory of convex bodies. The inequality \eqref{BMVolume} essentially says that $\left[\mbox{Vol}(\cdot)\right]^{1/n}$ is a concave function with respect to  Minkowski addition. The Brunn-Minkowski inequality has connections with many other inequalities including \textit{the isoperimetric inequality}, \textit{Sobolev inequalities}, \textit{Gagliardo-Nirenberg inequalities}.   Inequalities of Brunn-Minkowski type have been obtained for various other homogeneous functionals  including; torsional rigidity, the first eigenvalue of the Laplacian, Newtonian capacity  (for more details see the survey paper of Gardner \cite{Ga} and the book of Schneider \cite{Sc}). Combining results from \cite{CS,B1,CJL}, the following p-capacitary version of the Brunn-Minkowski inequality has been shown for $1<p<n$
\begin{align}
\label{BMcapacityp}
\left[\mbox{Cap}_{p}(\lambda E_1+(1-\lambda) E_2)\right]^{\frac{1}{n-p}}\geq\lambda \left[\mbox{Cap}_{p}(E_{1})\right]^{\frac{1}{n-p}}+(1-\lambda)\left[\mbox{Cap}_{p}(E_2)\right]^{\frac{1}{n-p}}
\end{align}
whenever $E_1, E_2$ are compact convex sets with nonempty interiors  in $\mathbb{R}^{n}$. Here $\mbox{Cap}_{p}(\cdot)$ denotes the $p$-capacity of a set defined by
\[
\mbox{Cap}_{p}(E)=\inf\left\{\, \, \int_{\mathbb{R}^{n}} |\nabla v|^{p} dx:\, \, v\in C^{\infty}_{0}(\rn{n}),\, v(x)\geq  1\, \, \mbox{for}\, \,  x\in E \right\}.
\] 
Moreover, equality in \eqref{BMcapacityp} holds precisely when $E_1$ is homothetic to $E_2$. 
In \cite{AGHLV}, the first, second, and fourth authors of this article along with Jasun Gong and Jay Hineman studied a generalized notion of $p$-capacity for $ 1 < p < n$ associated with ``$\mathcal{A}$-harmonic PDEs''  (see Definition \ref{defn1.1}) and denoted by $\mbox{Cap}_\mathcal{A} (\cdot)$.  It was shown in that article that the  Brunn-Minkowski type inequality holds:
 \[
\left[\mbox{Cap}_\mathcal{A}  ( \la  E_1 + (1-\la) E_2 )\right]^{\frac{1}{(n-p)}}  \geq  \la  \, 
\left[\mbox{Cap}_\mathcal{A}  (  E_1 )\right]^{\frac{1}{(n-p)}}  +  (1-\la)  \left[\mbox{Cap}_\mathcal{A}  (E_2 )\right]^{\frac{1}{(n-p)}}
\]
when  $1<p<n$, $0 \leq \la \leq 1,   $ and $E_1, E_2$ are convex compact sets with positive  $\mathcal{A}$-capacity (so $ E_1$ or $ E_2$ can have an empty interior).
Moreover, if equality holds  in  the  above inequality for some $E_1$ and $E_2, $ then under certain  regularity and structural assumptions on $\mathcal{A},$ it was concluded that these two sets are homothetic.  

We  note that when $p \geq n$, $\mbox{Cap}_{p}(B(x,r))=0$ (see \cite[Example 2.12]{HKM}) for any $r>0$. 
Borell in \cite{B2} and Colesanti and Cuoghi in \cite{CC} considered an analogous problem when $p=n$. To describe their work, consider a convex compact $E$ with non-empty interior. It can be shown that there exists a unique $u$ which is $n$-harmonic in $\mathbb{R}^{n}\setminus E$ with continuous boundary values $0$ on $\partial E$ and $u(x)=\log |x| +a+o(1)$ as $|x|\to\infty$. Now $u$ and $a\in\mathbb{R}$ are uniquely determined by $E$. If we let $\mathcal{C}_n(E):=e^{- a}$ then Borell in \cite{B2} when $n=2$ and Colesanti and Cuoghi in \cite{CC} for  $p=n >  2$ showed that $\mathcal{C}_n(\cdot)$ satisfies the Brunn-Minkowski inequality \eqref{BMcapacityp} with power $1/(n-p)$ replaced by $1$ for $n\geq 2$. In these articles, it was also shown that if equality holds for some convex compact sets $E_1$ and $E_2$ with non-empty interiors then they are homothetic. When $n=2$,  $\mathcal{C}_2(\cdot)$  is  often called  \textit{logarithmic capacity} so the authors of \cite{CC} called $\mathcal{C}_n(\cdot)$ an $n$-dimensional logarithmic capacity. In the first part of this article, using  a  similar  approach,  we study a  Brunn-Minkowski inequality associated with  an $\mathcal{A}$-harmonic Green's function when $n\leq p<\infty$. To our knowledge  this   study  is  the first of its kind  when $ n < p  <  \infty. $

\setcounter{equation}{0} 
\setcounter{theorem}{0}
\section{Notation and statement of results}
\label{NSR}

\noindent Let $n\geq 2$ and denote points in Euclidean $ n
$-space $ \rn{n} $  by $ y = ( y_1,
 \dots,  y_n) $.   Let $ \mathbb S^{n-1} $ denote the unit sphere in $\mathbb R^n$. 
 We write  $ e_m,  1 \leq m \leq n,  $  for the point in  $  \rn{n} $  with  1 in the $m$-th coordinate and  0  elsewhere. Let  $ \bar E,
\ar E$,  and $\mbox{diam}(E)$ be the closure,
 boundary, and diameter of the set $ E \subset
\mathbb R^{n} $ respectively. We define $ d ( y, E ) $   to be the distance from
 $  y \in \mathbb R^{n} $ to $ E$.  Given two sets $E$ and $F$ in  $\rn{n} $,   let  
  \[ 
  d_{\mathcal{H}} ( E, F ) := \max  (  \sup   \{ d ( y, E ) : y \in F \},   \sup \{ d ( y, F ) : y \in E \} ) 
  \]  
  be the Hausdorff distance between the sets $ E$ and $F$. 
 Also 
\[
 E  +  F  := \{ x + y:  x \in E, y \in F   \}
 \]  
 is  the  Minkowski sum of  $ E $ and $F.$ 
 We write  $  E  + x  $  for  $  E  +  \{x\} $ and if  $ \rho $ is a non-negative real number set 
$  \rho E =  \{ \rho y : y \in E\}. $    Let  $   \lan \cdot,  \cdot  \ran $  denote  the standard inner
product on $ \mathbb R^{n} $ and  let  $  | y | = \lan y, y \ran^{1/2} $ be
the  Euclidean norm of $ y. $   Put 
\[
B (z, r ) = \{ y \in \mathbb R^{n} : | z  -  y | < r \}\quad  \mbox{whenever}\, \, z\in \mathbb R^{n} \, \, \mbox{and}\, \, r>0. 
\]
Let  $dy$ denote  the $n$-dimensional Lebesgue measure on    $ \mathbb R^{n} $.  
Let  $\mathcal{H}^{\la},  0 <  \la  \leq n, $  denote the  $\la$-dimensional  \textit{Hausdorff measure} on $ \rn{n}$ defined by 
\[
\mathcal{H}^{\la }(E)=\lim_{\delta\to 0} \inf\left\{\sum_{j} r_j^{\la}; \, \, E\subset\bigcup\limits_{j} B(x_j, r_j), \, \, r_j\leq \delta\right\}
\] where the  infimum is taken over all possible  $  \de$-covering  $\{B(x_j, r_j)\} $ of  $E$.  
If $ O  \subset \mathbb R^{n} $ is open and $ 1  \leq  q  \leq  \infty, $ then by   $
W^{1 ,q} ( O ) $ we denote the space of equivalence classes of functions
$ h $ with distributional gradient $ \nabla h= ( h_{y_1},
 \dots, h_{y_n} ), $ both of which are $q$-th power integrable on $ O. $  Let  
 \[
 \| h \|_{1,q} = \| h \|_q +  \| \, | \nabla h | \, \|_{q}
 \]
be the  norm in $ W^{1,q} ( O ) $ where $ \| \cdot \|_q $ is
the usual  Lebesgue $ q $ norm  of functions in the Lebesgue space $ L^q(O).$  Next let $ C^\infty_0 (O )$ be
 the set of infinitely differentiable functions with compact support in $
O $ and let  $ W^{1,q}_0 ( O ) $ be the closure of $ C^\infty_0 ( O ) $
in the norm of $ W^{1,q} ( O  ). $  By $ \nabla \cdot $ we denote the divergence operator.
\begin{definition}  
\label{defn1.1}	
	Let $p,  \al \in (1,\infty) $ and 
	\[
	\mathcal{A}=(\mathcal{A}_1, \ldots, \mathcal{A}_n) \, : \, \rn{n}\sem \{0\}  \to \rn{n},
\]
   be such that $  
\mathcal{A}= \mathcal{A}(\eta)$  has  continuous  partial derivatives in 
$ \eta_k$ for $k=1,2,\ldots, n$  on $\rn{n}\setminus\{0\}.$  
We say that the function $ \mathcal{A}$ belongs to the class
  $ M_p(\alpha)$ if the following conditions are satisfied whenever  $\xi\in\mathbb{R}^n$ and
$\eta\in\mathbb R^n\setminus\{0\}$: 	
\begin{align*}
		(i)&\, \, \mbox{Ellipticity:}\, \,  \alpha^{-1}|\eta|^{p-2}|\xi|^2\leq \sum_{i,j=1}^n \frac{\partial  \mathcal{A}
_{i}}{\partial\eta_j}(\eta )\xi_i\xi_j   \quad \mbox{and}\quad   \sum_{i=1}^n  \left| \nabla  
 \mathcal{A}_i  (\eta)  \right|   \leq\alpha \, |\eta|^{p-2},\\
(ii)&\, \, \mbox{Homogeneity:}\,\, \mathcal{A} (\eta)=|\eta|^{p-1}  \mathcal{A}(\eta/|\eta|).
	\end{align*}
	\end{definition}   

We  put  $ \mathcal{A}(0) = 0 $  and note that  Definition \ref{defn1.1}  $(i)$ and $(ii) $ implies   
 \begin{align}  
 \label{eqn1.1}
 \begin{split}
 (|\eta | + |\eta'|)^{p-2} \,   |\eta -\eta'|^2   \approx   \lan  \mathcal{A}(\eta) -& 
\mathcal{A}(\eta'), \eta - \eta' \ran
\end{split}
\end{align}  
whenever $\eta, \eta'   \in  \rn{n} \sem \{0\}$ and the proportionality  constants depend only on $p$,$n$, and $\alpha$. 
\begin{definition}
\label{defn1.2}                       
	Let $p\in (1,\infty)$ and let $ \mathcal{A}\in M_p(\alpha) $ for some $\alpha\in (1,\infty)$. Given an  open set 
 $ O  $ we say that $ u $ is $  \mathcal{A}
$-harmonic in $ O $ provided $ u \in W^ {1,p} ( G ) $ for each open $ G $ with  $ \bar G \subset O $ and
	\begin{align}
	\label{eqn1.2}
		\int \lan    \mathcal{A}
(\nabla u(y)), \nabla \he ( y ) \ran \, dy = 0 \quad \mbox{whenever} \, \,\he \in W^{1, p}_0 ( G ).
			\end{align}
	 We say that $  u  $ is an  $\mathcal{A}$-subsolution ($\mathcal{A}$-supersolution) in $O$ 
	 if   $  u \in W^{1,p} (G) $ whenever $ G $ is as above and  \eqref{eqn1.2} holds with
$=$ replaced by $\leq$ ($\geq$) whenever $  \theta  \in W^{1,p}_{0} (G )$ with $\theta \geq 0$.  
As a short notation for \eqref{eqn1.2} we write $\nabla \cdot \mathcal{A}(\nabla u)=0$ in $O$.  
\end{definition}  

\begin{remark}    
\label{rmk1.3}
We remark  for  $O, \mathcal{A}, p,  u,$  as in  Definition \ref{defn1.2}  that if   $F:\mathbb R^n\to \mathbb R^n$  is  the composition of
a translation and a dilation
 then 
 \[
 \hat u(z)=u(F(z))\, \,  \mbox{is}\, \,  \mathcal{A}\mbox{-harmonic in} \, \,  F^{-1}(O).
 \]
 Moreover,  if   $ \ti  F:\mathbb R^n\to \mathbb R^n$  is  the  composition of
a translation,  a dilation, and a rotation 
   then 
   \[
   \ti  u (z) = u (\ti F(z) )\, \, \mbox{is}\, \,  \ti {\mathcal{A}}\mbox{-harmonic in}\, \,  \ti F^{-1}(O)\, \,  \mbox{for some}
  \, \, \ti{\mathcal{A}}\in M_p(\alpha).
  \]  
\end{remark}  

We note that $\mathcal{A}$-harmonic PDEs have been studied in \cite{HKM}. Also   dimensional properties of the
 Radon measure associated with  a positive $\mathcal{A}$-harmonic function $u,$  vanishing  on a portion of the boundary of  $ O, $   
 have been studied in \cite{A,ALV12, ALV,LN,AGHLV},  see also \cite{LLN, LN4}. 

In this  article, we often assume that $E\subset\mathbb{R}^{n}$ is a  convex compact set usually  containing at 
least two points or equivalently having positive $\mathcal{H}^{1}$ measure.  
Note that when $p=2$ and $\mathcal{A}(\eta)=(\eta_1, \ldots, \eta_n)$, $\eta\in\mathbb{R}^{n}\setminus\{0\},$  
then one gets  Laplace's equation in $  \rn{n}. $  If we additionally let $n=2$ 
and $\Omega=\mathbb{R}^{2}\setminus E, $   then the  logarithmic capacity of $E$ 
can be defined through the behavior of the Green's function near $\infty$. That is,  
the Green's function $u(\cdot,\infty)$ with pole at $\infty$ has continuous boundary values  zero  on  $ \ar E, $  and has a representation 
$u(z,\infty)=\log|z|+H(z,\infty)$ in $ \Om $     where  $H(z,\infty)$ is harmonic, bounded,  and  continuous  at  $ \infty. $   The quantity $H(\infty,\infty)=  a$ is known as the 
\textit{Robin's constant} for $E$.  Thus  $u(z, \infty)=\log|z|+a+o(1)$ as 
$|z|\to \infty$.  Classically  the \textit{logarithmic capacity} of $E$ in the plane is defined by 
$\mbox{Cap}_{2}(E)=e^{- a}$.  More about  logarithmic capacity in the 
plane can be found in \cite[Chapter III]{GM} and in \cite[Page 167]{La}.

When $\mathcal{A}(\eta)=|\eta|^{p-2}\eta$ for $\eta\in\mathbb{R}^{n}\setminus\{0\}$ and
 $1<p<\infty$ then one gets the p-Laplace equation $\nabla\cdot(|\nabla u|^{p-2} \nabla u)=0$ 
and solutions are called $p$-harmonic functions.  In $\mathbb{R}^{n}$, $n\geq 2$, $\log |x|$ 
is  an $n$-harmonic function and $|x|^{\frac{p-n}{p-1}}$ is  a  $p$-harmonic function in 
$\mathbb{R}^{n}\setminus \{0\}$ when $1<p\neq n<\infty$.   As in the planar
case, given  a  convex compact set $E\subset\mathbb{R}^{n}$ with non-empty interior, 
there exists a unique $p$-harmonic function $u$ such that 
$\nabla \cdot(|\nabla u|^{p-2} \nabla u)=0$ in $\mathbb{R}^{n}\setminus E$ 
with continuous boundary values $0$ on $\partial E$ satisfying   
\[
u(x)=\left\{
\begin{array}{ll}
\log |x|+ a+o(1) &\mbox{when}\, \, p=n,\\
|x|^{\frac{p-n}{p-1}}+a+o(1) & \mbox{when} \, \, n<p<\infty,
\end{array}
\right. 
\quad \mbox{as}\, \, |x|\to\infty. 
\]  
\noindent Here $a\in \mathbb{R}$ when $p=n$ and $a <0$ when $n<p<\infty$. 
So     based on the  classical  function theory,  $ p = 2, n = 2, $  case,  a natural definition of  
$ \mathcal{C}_p ( E ) $ is $\mathcal{C}_p(E)=e^{- a} $ when $p=n$ and $\mathcal{C}_p(E)=(-a)^{p-1}$ 
when $n<p<\infty$. It can be shown that $u$ and $a$ are uniquely 
determined by $E$. Moreover, $ t  \mapsto \mathcal{C}_p(t E)$ for $t \in (0, \infty) $  is 
homogeneous of degree $1$ function when $p=n$ while this function  is homogeneous of degree $p-n$ 
when $n<p<\infty$. 

Using the above approach  but  a  more  PDE  inspired normalization,  we will define   $\mathcal{C}_\mathcal{A} (E),  $  whenever  $n\leq p<\infty,$  and  $E$ contains at least two points when $p=n$  or is nonempty when $p>n$.  We will show in \S\ref{section4} that there exists a unique $u, $   satisfying 
\begin{align}
\label{Aharmonic}
\begin{split}
\left\{
\begin{array}{l}
\nabla \cdot \mathcal{A}(\nabla u)=0 \, \, \mbox{in} \, \, \mathbb{R}^{n}\setminus E,\\
u=0 \, \, \mbox{on}\, \, \partial E,\\
u(x)=F(x) + a+o(1) \, \, \mbox{as}\, \, |x|\to\infty
\end{array}
\right.
\end{split}
\end{align}
where $F(x)$ is the so called fundamental solution to an $\mathcal{A}$-harmonic PDE 
with pole at infinity (see Lemma \ref{lem1.2} and Lemma \ref{lem1.4}).  Here   $ F (x) \approx \log|x|$,    
when $p=n$ and $F(x)\approx   |x|^{\frac{p-n}{p-1}}$ when $n<p<\infty$ as $ x \to \infty$. Given $E$ 
 as above we call $u$  the $\mathcal{A}$-harmonic Green's function for $\mathbb{R}^{n}\setminus E$ with pole at  $\infty$. We will also see in \S\ref{section5} that $u$ and $a$ 
are uniquely determined by $E$. Thereupon, we define
\[
\mathcal{C}_{\mathcal{A}}(E)=
\left\{
\begin{array}{ll}
e^{- a/\ga}  & \mbox{when}\, \, p=n, \\
(-a)^{p-1} &\mbox{when} \, \, n<p<\infty.
\end{array}
\right.
\]
Here $\gamma$ is a non-negative constant (see Lemma \ref{lem1.4}). With this notation we prove that $\mathcal{C}_{\mathcal{A}}(\cdot)$ satisfies  the  following 
Brunn-Minkowski type inequality when $n\leq p<\infty$. In particular, our first main result is
\begin{mytheorem}
\label{theoremA}
Let $E_1$ and $E_2$ be compact  convex sets in $ \rn{n}, n \geq 2.$  
Assume that both sets contain at least two points when  $ p = n $  and that both sets are nonempty when $p > n.$  If  $\lambda\in[0,1]$ and if $p=n$ then 
\begin{align}
\label{BMn}
\mathcal{C}_{\mathcal{A}}(\lambda E_1+(1-\lambda)E_2)\geq \lambda \mathcal{C}_{\mathcal{A}}(E_1) +(1-\lambda)\mathcal{C}_{\mathcal{A}}(E_2).
\end{align}
While if   $ n<p<\infty$ then 
\begin{align}
\label{BMpn}
[\mathcal{C}_{\mathcal{A}}(\lambda E_1+(1-\lambda)E_2)]^{\frac{1}{p-n}}\geq 
\lambda \mathcal{C}_{\mathcal{A}}(E_1)^{\frac{1}{p-n}} +(1-\lambda)\mathcal{C}_{\mathcal{A}}(E_2)^{\frac{1}{p-n}}.
\end{align}
If equality holds in \eqref{BMn}  or in  \eqref{BMpn}  and $\mathcal{A}$ satisfies 
\begin{align} 
\label{eqn1.8}
\begin{split}
(i) &\, \, \mbox{There exists} \, \, 1\leq  \La <\infty\ \, \mbox{such that}\, \, 
\left| \frac{ \ar \mathcal{A} _i }{\ar \eta_j}   ( \eta )  -  \frac{ \ar \mathcal{A} _{i} }{ \ar\eta'_j} ( \eta') \right|\leq  \La  \, | \eta  -   \eta' |    |\eta | ^{p-3}
 \\ 
 &\hs{.4in}   
 \mbox{whenever}\, \, 0 < {\ts \frac{1}{2}}\,  |\eta |   \leq   |\eta'|  \leq 2     |\eta |\, \, \mbox{and}\, \,   1 \leq  i \leq n, \\
(ii)&\, \,  \mathcal{A}_i   (\eta)  =   \frac{\ar f}{\ar \eta_i} \, \,
\mbox{for}\, \,  1 \leq i  \leq n  \, \, \mbox{where}\, \,   f (t \eta )  = t^p  f  (\eta ) \, \, \mbox{when}\, \,  t > 0\, \, \mbox{and}\, \, \eta \in \rn{n}  \sem \{0\},
\end{split}
\end{align}
  then $  E_2 $ is a translation and dilation of  $ E_1$ provided that both sets contain at least two points.   
\end{mytheorem}
We first remark  that the work in \cite{CC,B2} corresponds to the case $p=n$ 
and $\mathcal{A}(\eta)=|\eta|^{n-2}\eta$,  whereas  in  Theorem  \ref{theoremA}   we consider 
 all possible  values of $p$, $n\leq p<\infty$ and general $\mathcal{A}\in M_p(\alpha)$ for $\alpha\in (1,\infty)$. Second, 
   in \cite{CC,B2}, the convex compact sets $ E_1$ and $E_2$  are required to have interiors, unlike our  ``bare bones''  assumption on $ E_1$ and $E_2$.    
   We   consider it  a  very interesting question as to  whether     \eqref{eqn1.8} is needed in Theorem \ref{theoremA}.  
   Finally,  Theorem \ref{theoremA}   will be used to prove uniqueness in  the  Minkowski type problem associated with  
   an  $ \mathcal{A}$-harmonic Green's function for the complement of  a  compact convex set   when   $ p \geq n $ (see   \S\ref{section7}).

\setcounter{equation}{0} 
 \setcounter{theorem}{0}
\section{Basic estimates for $\mathcal{A}$-harmonic functions}
\label{section2}
\noindent In this section we  state  some fundamental estimates for 
   $\mathcal{A}$-harmonic functions. 
Concerning constants, unless otherwise stated, in this section, and throughout the paper,
$ c $ will denote a  positive constant  $ \geq 1$, not
necessarily the same at each occurrence, depending at most on
 $ p, n, \alpha,  \La $ 
   which sometimes we refer to as depending on the data.
In general,
$ c ( a_1, \dots, a_m ) $  denotes  a positive constant
$ \geq 1, $  which may depend at most on the data  and   $ a_1, \dots, a_m, $
not necessarily the same at each occurrence. If $B\approx C$ then $B/C$ is bounded from above and below by
constants which, unless otherwise stated, depend at most on  the data.
 Moreover, we let
$ { \ds \max_{F}   \ti u, \,  \min_{F} \ti  u } $ be the
  essential supremum and infimum  of $  \ti u $ on $ F$
whenever $ F \subset  \mathbb R^{n} $ and  $ \ti u$ is defined on $ F$.
Finally in this section, unless otherwise stated, we assume that $ 1 < p < \infty. $ 
We shall start with a  definition of thickness and fatness of a closed set
\begin{definition}[{\bf Thickness}]
\label{def3.1}
Given  $ \la>0$ and $r_0  > 0$,  a closed set $E$ is called $( r_0, \lambda)$-thick if there exists $c>0$ such that  
\[
\mathcal{H}^{\lambda}_{\infty}(E\cap \bar{B}(w,r))\geq cr^{\lambda} \quad \mbox{for all}  \quad  0 <  r \leq r_0   \quad \mbox{and} \quad w\in E.
\]
\end{definition}
\noindent Here $\mathcal{H}^{\lambda}_{\infty}$ denotes the Hausdorff content of a 
set (take the infimum in the definition of  $  \mathcal{H}^{\la} $  over all possible covering of the set).  The  classical definition of  the  $p$-capacity of a compact
 set $E$ inside a bounded domain $D$ is defined by
\[
\mbox{Cap}_{p}(E, D)=\inf\left\{\, \, \int_{D} |\nabla v|^{p} dx:\, \, v\in C^{\infty}_{0}(D),\, v(x)\geq  1\, \, \mbox{for}\, \,  x\in E \right\}.
\]
Note that when $D=\mathbb{R}^{n}$ one has $\mbox{Cap}_{p}(E, \mathbb{R}^{n})=\mbox{Cap}_{p}(E)$ as defined earlier.  
\begin{definition}  [{\bf Fatness}]
\label{def3.2} 
A closed set $E\subset\mathbb{R}^{n}$ is called uniformly $(r_0, p)$-fat if
\[
 \frac{\mbox{Cap}_{p}(E\cap\bar{B}(w,r), B(w,2r))}{\mbox{Cap}_{p}(\bar{B}(w,r), B(w,2r))}\geq c 
\]
for $ 0 < r \leq r_0 $  and   all $w\in E.$  
\end{definition} 
We note from \cite[Example 2.12]{HKM} that  for $ n <  p  < \infty, $  
\[ 
\mbox{Cap}_{p}(\{x_0\}, B(x_0, 2r))  \approx  \mbox{Cap}_{p}( \bar B ( x_0, r), B(x_0, 2r))   \approx  r^{n-p}.
\]   
Thus if $ p > n$ and $ 0 < r_0 <  \infty$  then any nonempty closed set is uniformly 
$ (r_0, p )$-fat. On the other hand,  if $ 1  < p  \leq n, $  it  follows essentially from  Frostman's  lemma  
that if a closed set $E$ is $ (r_0, \lambda)$-thick for some $r_0$ and $ \lambda>n-p$ then $E$ is uniformly
$( r_0, p )$-fat    (see  \cite[Corollary 5.1.14]{AH}).  Thus  a  compact convex set  
containing at least two points  is uniformly $ (r_0, n) $-fat with capacitary ratio (i.e, $c$ in Definition \ref{def3.2})  depending only on $n$  when  
$ r_0   =\mbox{diam}(E)$.     Finally, uniform $(r_0, p)$-fatness for some $ r_0 > 0, $  is a sufficient condition for solvability  
 of the Dirichlet problem for $\mathcal{A}$-harmonic PDEs in a  bounded domain $\Omega$   
 in 
 the sense that if  every point   $w\in\mathbb{R}^{n}\setminus \Omega$ is uniformly $(r_0, p)$-fat,
  then   
  \[
\int_{0}^{r_0}\left[\frac{\mbox{Cap}_p((\mathbb{R}^{n}\setminus  {\Omega})\cap \bar B(w,r), B(w,2r))}{\mbox{Cap}_p(\bar{B}(w,r), B(w,2r))}\right]^{\frac{1}{(p-1)}}\frac{dr}{r}=\infty. 
\]  
That is, uniform $(r_0, p)$-fatness implies Wiener regularity  (see \cite[Theorem 6.33]{HKM}). 
Next we  state some basic estimates for $\mathcal{A}$-harmonic functions.
\begin{lemma}
\label{lemma2.1}
Given $ p$, $  1 < p < \infty, $ assume that $\ti {\mathcal{A}} \in M_p(\alpha)$ for some
$\alpha\in (1,\infty)$ in $\mathbb{R}^{n}$, $n\geq 2$.    Let $ \ti u $  be  a
 $\ti {\mathcal{A}}$-harmonic function in $B (w,4r)$ for some $r>0$. Then 	
 \begin{align}
	\label{eqn2.1}
	\begin{split}
		(i)&\, \, r^{p - n } \,\int_{B ( w, r/2)} \, | \nabla \ti  u |^{  p } \, dy \, \leq \, c \, (\max_{B ( w,r)} |  \ti  u| )^p  \leq c^2  r^{-n}  \int_{B ( w,  2r) }  | \ti u|^p \, dx,   \\
		(ii)&\, \, \mbox{ If $ \ti u \geq  0 $ in  $ B ( w, 2r )$   
 then   $  {\ds \max_{B ( w, r ) }} \,\ti   u \, \leq c {\ds \min_{ B ( w, r )} }\ti u.$} 
		\end{split}
\end{align} 	
	Furthermore, there exists $\ti \sigma=\ti \sigma(p,n,\alpha)\in(0,1)$ such that
	\begin{align*}
		(iii)& \ \ | \ti  u ( x ) - \ti  u ( y ) | \leq c \left( \frac{ | x - y |}{r} \right)^{\ti \sigma} \, \max_{B (w, 2 r )}  \,| \ti  u |
	\end{align*}
	whenever $x, y \in B ( w, r )$.
   \end{lemma}

\begin{proof}
A proof of this lemma can be found in  \cite{S}. 
\end{proof} 
\begin{lemma}  
\label{lemma2.2}	
Let $p,n, \ti {\mathcal{A}}, \alpha, w,r, \ti u$ be as in Lemma \ref{lemma2.1}.
Then	 $ \ti u$ has a  representative locally in  $ W^{1, p} (B(w, 4r)),$    with H\"{o}lder
continuous partial derivatives in $B(w,4r) $  (also denoted $\ti u$). Moreover, there exist $ \ti \be \in (0,1]$ and  $c \geq 1 $,
depending only on $ p$, $n$, and $\alpha$, such that if $ x, y \in B (  w,  r ), $   then  
	\begin{align}
	\label{eqn2.2}
	\begin{split}
&	(\hat a) \, \,  c^{ - 1} \, | \nabla \ti u ( x ) - \nabla \ti u ( y ) | \, \leq \,
 ( | x - y |/ r)^{\ti \be }\, \max_{B (  w, r )} \, | \nabla \ti u | \leq \, c \,  r^{ - 1} \, ( | x - y |/ r )^{\ti \be }\, \ti u (w), \\
 & (\hat b)  \, \,  \int_{B(w, r) }  \, \sum_{i,j = 1}^n  \,  |\nabla \ti u |^{p-2} \,  |\ti  u_{x_i x_j}|^2  dy 
\leq   c r^{(n-p-2)} \ti{u}(w).
  \end{split}
\end{align}
 	If  	
  	\[ 
  	\ga\, r^{-1} \ti u   \leq |\nabla  \ti u | \leq  \ga^{-1} r^{-1}    \ti u  \quad \mbox{on}\quad  B ( w, 2r)  
\] 
for some $ \ga \in (0,1) $  and 
 \eqref{eqn1.8} $ (i)$ holds then 
 $\ti u $ has    H\"{o}lder
continuous second  partial derivatives in $B(w,r) $  and
 there exist $\ti \he \in (0,1)$ and $\bar c \geq 1$, depending only on  the data  and  $ \ga $,  such that
  \begin{align}  
  \label{eqn2.3}
  \begin{split}
  \left[ \sum_{i,j = 1}^n  \,   ( \ti  u_{x_i x_j } (x)      -  \ti  u_{y_i y_j } (y) ) ^2  \, \right]^{1/2} \, 
  &\leq  \, \bar c  ( | x - y |/ r )^{\ti \he } \,  \max_{ B(w, r)}\,  \left(\sum_{i,j = 1}^n  \,   |\ti  u_{x_i x_j}| \right)     \\ 
  &\leq  \, \bar c^2  r^{-n/2 }\,  ( | x - y |/ r )^{\ti \he } \,   \left(\sum_{i,j = 1}^n  \,   \int_{ B ( w, 2r) }   \ti  u_{x_i x_j}^2   dx \right) ^{1/2} \\
  &  \leq  \bar c^3 \,  r^{ - 2} \, ( | x - y |/ r )^{\ti \he }\, \ti u (w)
\end{split}  
  \end{align}
  whenever   $ x, y \in B (w,  r/2) $. 
   \end{lemma}    
   \begin{proof}
   A proof of \eqref{eqn2.2} can be found  in \cite{T}. Also, 
\eqref{eqn2.3}  follows from  \eqref{eqn2.2}, the added assumptions,  and Schauder type
estimates (see  \cite{GT}).
\end{proof}
\begin{lemma} 
\label{lemma2.3}
Fix  $ p$ with $ 1 <  p<\infty$ and  assume  that $\ti {\mathcal{A}}\in M_p(\alpha)$ 
for some $\alpha\in (1,\infty)$. Let  $ \ti E \subset B (0, R) $, for some $R>0$,  
be a uniformly  $  (r_0, p )$-fat compact  convex set  where $ r_0 = \mbox{diam}(\ti E)$.  
Let   $\zeta \in C_0^\infty ( B (0, 2 R)   ) $ with  $ \ze  \equiv 1 $  on  
$  B (0, R). $  If $  0  \leq  \ti u $ is   $ \ti {\mathcal{A}}$-harmonic in
  $ B (0, 4R) \sem \ti E,$    and  $  \ti u \ze \in  W_0^{1,p} (B (0,4R) \sem \ti E ),$    
  then  $  \ti u $  has a continuous  extension to  $ B (0, 4R) $  obtained 
  by putting  $ \ti  u  \equiv 0 $ on  $   \ti  E. $  Moreover, if   $  0  < r  <    R$ 
  and   $w \in   \ti E$ then    
	 \begin{align}
	 \label{eqn2.4} 		
	 (i) \quad  r^{ p - n} \, \int\limits_{B ( w, r)} \, | \nabla \ti  u |^{ p } \, dy \, \leq \, c \, \left( \max_{B ( w, 2 r)}  \ti u \right)^p.   
	\end{align}  	Furthermore,  there exists  
$\hat \sigma \in (0,1),$   depending on  $ p,n,\alpha,$ and the uniform 
$ (r_0, p)$-fatness   constant for $ \ti E $,  such that 
	\[ 
		(ii)\quad  | \ti  u ( x ) - \ti  u ( y ) | \leq c \left(  \frac{ | x - y |}{r} \right)^{ \hat \sigma} \,  \max_{ B (w, 2 r )}\, \ti u
		\] 
		whenever $ x, y \in B ( w, r)$ and $ 0 < r <\mbox{diam}(\ti  E)$. 
		\end{lemma}  
\begin{proof} 
Here $(i) $ is a standard  Caccioppoli  inequality and
$(ii) $  for  $ y \in  \ti E $    follows from uniform $ (r_0, p)$-fatness of  $ \ti E $ 
 and  essentially Theorem 6.18  in  \cite{HKM}.  Combining this fact with  
  \eqref{eqn2.1}  $ (iii)$ we  obtain  $ (ii). $   	
  \end{proof}

\begin{remark} 
\label{rmk2.4} 		
As an application of    Lemma  \ref{lemma2.3}  observe from the remarks  following  Definitions   \ref{def3.1}  and \ref{def3.2}  that if 
		$ p >n$  then   Lemma  \ref{lemma2.3} is  valid for any compact set 
		$ \ti E  \subset  B (0, R) $   with   $\mbox{diam}(\ti E)$ in \eqref{eqn2.4} $(ii)$  replaced by $ R$.  
		If   $p =  n$  and $  \ti E $  is  convex, containing  at least two points,  
		then  $ \ti E $  contains a  line segment  so  is  $ (r_0, 1)$-thick and thus  
		uniformly  $( r_0, n )$-fat.  Hence   Lemma \ref{lemma2.3} applies in this situation as well. 
		\end{remark}  
\begin{lemma} 
\label{lemma2.4}
Let  $   \ti {\mathcal{A}}, p, n, \alpha, \ti E, r_0,  R, \ti u   $ be  as in  Lemma \ref{lemma2.3}. 
	 Then there exists a unique finite positive Borel measure $ \ti \mu $ with support contained in $ \ti E $ such that
 \begin{align}
 \label{eqn2.6}
		(i)\, \, \int  \langle \ti {\mathcal{A}} (\nabla \ti u(y)),\nabla\ph(y)\rangle \, dy \, = \, -  \int \, \ph \, d\ti{\mu} \quad \mbox{whenever}\, \, \ph  \in C_0^\infty ( B(0,2R) ).
	     \end{align} 	 
Moreover,  there exists $ c \geq 1, $ with the same dependence as in  Lemma \ref{lemma2.3}    such that   
	\begin{align*}
		(ii)\, \,  c^{-1}\, r^{ p - n} \ti \mu (B( w,  r ))\leq   \max_{B(w, 2r)} \ti u^{ p - 1}\leq c r^{p-n}\ti \mu(B(w,4r))
	\end{align*}  
	whenever $ 0 < r \leq r_0 $ and $ w  \in \ti E$. 
		 \end{lemma} 
\begin{proof}
 For the proof  of  $(i)$ see  \cite[Theorem 21.2]{HKM} or \cite{KZ}.  
 The left-hand inequality in $ (ii),   $  follows from   $ (i) $   in  \eqref{eqn1.1} with  $ \eta'  = (0, \dots,0), $  and  H\"{o}lder's inequality,   
 using a  test function, $ \ph, $ with  $ \ph \equiv 1 $ on $ \bar B ( w, r ).  $   
  The  proof of the right-hand inequality in  \eqref{eqn2.6} $ (ii)$ follows from \cite{EL} or \cite{KZ}.  
\end{proof}

\setcounter{equation}{0} 
\setcounter{theorem}{0}
\section{Existence and uniqueness of a fundamental solution}  
\label{section4}
Let $p$ and $\alpha$ be fixed with $2\leq n\leq p<\infty$ and $\alpha\in (1,\infty)$. 
In this section, for $\mathcal{A}\in M_p(\alpha)$, we will show existence and uniqueness of 
a  fundamental solution to $\mathcal{A}$-harmonic PDE. The cases $ p  >  n $  and  $ p  = n $  
require different  proofs and slightly different definitions. 
We  begin with the  more difficult case   $ p =  n.  $

\subsection{Existence and uniqueness for $p = n$} 
\begin{definition}
\label{def4.1}
If  $ p  = n $  we say that $ F$  is a fundamental solution to  
$\nabla\cdot \mathcal{A}(\nabla F)=0$ in $\mathbb R^{n}$ 
with pole at $0$ if  
\begin{align}
\label{1.2} 
\begin{split}
 (i)  \hs{.2in}  & \quad  \mbox{$F $ is  $  \mathcal{A}$-harmonic in  $ \mathbb{R}^n  \sem \{0\} $,  $ F\in W_{\mbox{\tiny loc}}^{1,l}(\mathbb R^{n}) \, \, \mbox{for}\, \,  1 < l < n$},\\ 
\hs{.45in} &\mbox{  and  $ |F ( x )|  =  O ( \log |x| )  $  in a neighborhood  of  $ \infty. $ } \\ 
(ii) \hs{.2in} &\quad \int \lan \mathcal{A} (\nabla  F(z)), \nabla \he ( z ) \ran\, dz = -\theta(0) \quad \mbox{whenever }\,\, \theta\in C_0^\infty(\mathbb R^{n}).
\end{split}
\end{align}
\end{definition}   
In order to show existence  and uniqueness of $ F $ we  require several  preliminary  lemmas.   Fix  $ R$   with $2   <  R  <  \infty,  $  and let $ v = v ( \cdot, R )  $  
be the  $  \mathcal{A}$-harmonic function in  $B (0, R) \sem \bar B (0, 1) $  
with  continuous boundary values $0$ on   $ \ar B (0, 1 ) $  and  $ \log R $   
on  $  \ar B (0, R ). $ Extend $ v $  to $  B (0, 1) $ by putting $ v \equiv 0 $ on 
$  B (0, 1 ). $     Let  $ \nu $  be the positive Borel measure  as in Lemma \ref{lemma2.4} with support contained
on  $ \ar B (0, 1 ) $ associated to $ v. $     We claim that  
\begin{lemma} 
\label{lem1.2}        
For $ v$ and $\nu$  as above  we have      
\begin{align}  
\label{1.3} 
\begin{split}
 (a)  &\quad   \nu ( \overline{B}(0, 1))  \approx  1, \\  
 (b) &\quad      c^{-1} \leq    \lan \nabla v (x), x  \ran  \leq   |x|  | \nabla v |   \leq c \quad \mbox{whenever} \,\,  x  \in    B(0, R) \sem B (0, 1),  \\  
(c) &\quad    v (x)   \approx  \log |x|\quad \mbox{whenever}\, \, x \in B(0,R) \sem B (0, 1 )
\end{split}
\end{align}    
where  ratio  constants depend  only on $ n $ and $\alpha$.       
\end{lemma}

\begin{proof}   
To  prove  this inequality   let   $  w (x)  =   \log | x |$, when $x  \in  B(0, R), $ and 
note that  $ w $  is  $n$-harmonic in  $ B(0, R ) \sem \bar B  (0, 1)$  with  continuous
 boundary values and $ v \equiv w $  on  $  \ar  [B ( 0, R ) \sem B ( 0, 1 )].  $   Then  
\begin{align}
\label{1.4}  \nu ( \bar B (0, 1 )) \, \log R  \, =  \int_{ B (0, R ) \sem \bar B (0, 1 ) }  \lan \mathcal{A} ( \nabla v ), \nabla v \ran dx 
\approx  \int_{ B (0, R ) \sem \bar B (0, 1 ) }  |  \nabla w |^n   dx  \approx  \log R.  
\end{align}
where all ratio constants depend only on $ n $ and $\alpha$.     The left-hand inequality
 in  \eqref{1.4}  follows from the definition of  $ \nu, $  using  $  \max(\log R - v,0) $  as  a  test function.
That is, as in Lemma \ref{lemma2.4}, we first have
\begin{align*}
-\int \langle \mathcal{A}(\nabla v), \nabla v)\rangle dx&=\int \langle \mathcal{A}(\nabla v), \nabla (\log R-v)\rangle dx \\
&= -\int_{\partial B(0,1)} \log R \, d\nu = -\log R \, \nu(\bar{B}(0,1)). 
\end{align*}
The  middle  inequality in \eqref{1.4} follows from  H{\"o}lder's inequality after using  
$ v  -  w $  as  a  test function in the definition of  $ \mathcal{A}$ and $n$-harmonic functions
  using  $ v, w. $ Indeed, since $v$ is $\mathcal{A}$-harmonic in $B(0,R)\setminus \bar{B}(0,1)$ 
\begin{align*}
0=\int \langle \mathcal{A}(\nabla v), \nabla (v-w)\rangle dx =
\int \langle \mathcal{A}(\nabla v), \nabla v\rangle dx - \int \langle \mathcal{A}(\nabla v), \nabla w\rangle dx.
\end{align*}
Using the structural assumptions on $ \mathcal{A}$ and  H{\"o}lder's inequality  with $ \ep$'s to estimate the second  integral on the right in the above equality,  it follows that 
\[  \int |\nabla v|^n dx \leq \alpha \int \langle \mathcal{A}(\nabla v), \nabla v\rangle dx 
 =\alpha \int \langle \mathcal{A}(\nabla v), \nabla w\rangle dx  \leq  c 
\int |\nabla w|^n dx
\]
where $c$ depends only on $n$ and $\alpha$. A similar argument using $n$-harmonicity of   $w$  and $ v - w $  as  a  test function gives the reverse inequality. 
  Thus \eqref{1.3} $(a)$ is valid.  To  prove \eqref{1.3} $(b)$  we note that  $  \ar B (0, 1 )$  is uniformly  $(1,n)$-fat   
in the sense of  $n$-capacity defined relative to  $ B (0, 4 ) $.  
It follows  from  Lemma \ref{lemma2.4} $(ii)$ and  Harnack's inequality for $ \mathcal{A}$-harmonic functions   that 
\begin{align}
\label{1.5}   
1  \approx  \nu ( \bar  B (0, 1) )^{1/(n-1)}  \approx   \max_{ \ar B (0, 2 ) } v.   
\end{align}  

Similarly, if  $ \tau $ is the measure with support contained on $\ar B ( 0, R ) $ 
associated with  $\log R  - v $ as in Lemma \ref{lemma2.4}  and $   \ti{\mathcal{A}} ( \eta )  =- \mathcal{A} (- \eta )$  then  
\begin{align} 
\label{1.6}  
1 \approx   \nu ( \bar B (0, 1) )^{1/(n-1)} =  \tau ( \bar B (0, R) )^{1/(n-1)}          \approx  \max_{ \ar B (0, R/2  ) }( \log R  -   v ).
\end{align} 

We  now  argue as in the proof of Lemma 4.3 in \cite{AGHLV}. For  fixed  $   1 <  \la  <   101/100,   $  we claim that   
\begin{align}
\label{1.7}    
k ( x )   :=  \frac{ v ( \la x ) -  v ( x )}{ \la  -  1} \geq  c^{-1} \quad  \mbox{whenever}\, \, x \in  B (0,  R/\la ) \sem B (0, 1 )
\end{align}  
where  $ c $  depends only on $ n $  and the structure constants for  $  \mathcal{A}. $   To prove \eqref{1.7} let           
\begin{align}  
\label{1.8}   
\ph_1 ( x  )  =      \frac{ e^{ N |x|^2 } -   e^{ N  }}{e^{64  N    } -      e^{N }  } \quad  \mbox{and} \quad  \ph_2 ( x  )  =     
   \frac{ e^{-  N |x|^2 } -  e^{- 64  N   }}{e^{-  N    } -  
   e^{- 64 N  }}.
   \end{align}
   Notice that $\phi_1 = 0$ on $\partial B(0,1)$ and $\phi_1 = 1$ on
    $\partial B(0,8)$. Similarly, $\phi_2 = 1$ on $\partial B(0,1)$ and
     $\phi_2 = 0$ on $\partial B(0,8)$.  Note also that   at   points  in  $ B (0, 8) \sem \bar B (0, 1)   $  where  $  \nabla v  \neq  0 $,  these functions are subsolutions    to the  second order  non-divergence form PDE  for $v$ 
     corresponding to  $ \mathcal{A}$-harmonicity when $ N  \geq 1 $ is sufficiently   
     large depending only on the data          (see \cite[Section 7.1]{AGHLV}).  Using these  notes and arguing as in \cite[Section 7.1]{AGHLV}  
 we see   there exists  $ N, c_1 \geq 1,  $ 
     depending only  on   the structure constants for $\mathcal{A}$, $\alpha$, and $n$ such that   
\begin{align}  
\label{1.9}   
c_1^3  (|y| - 1)  \geq   c_1^2 \,  [   1 - \ph_2 (y)  ]  \geq  c_1  v ( y  ) \geq   \ph_1 (y )   \geq 
    c_1^{-1}    (  |y| - 1 )
    \end{align}
    when     $ y  \in B (0, 8 ) \sem \bar B (0, 1)$.  From the  lower bound for  
    $ v $  in \eqref{1.9} we find that   \eqref{1.7}  is valid  when  $ x   \in   \ar B ( 0,  1  ) $ 
    for $ c $ suitably large.
  Similarly    
\begin{align}
     \label{1.10}   
     c_3 (  1 - |y|/R)  \geq   c_1^2   [   1   -  \ph_1 ( 8 y/R ) ]  \geq c_1 (\log R  - v (y ))  \geq 
      \ph_2 (  8y/R )   \geq   c_1^{-1}  (1  - |y|/R ) 
      \end{align}       
   when  $ y   \in  B ( 0, R )\sem B(0, R/8). $   From the lower bound for  
   $ \log R - v $  in  \eqref{1.10}  we deduce that  \eqref{1.7}  holds for $ x 
   \in \ar B (0, R/\la). $     From the  boundary  maximum principle for 
   $ \mathcal{A}$-harmonic functions it follows that \eqref{1.7} is true.  
   Letting $ \la  \to 1 $  and using the chain rule we obtain   the left-hand 
   inequality in  \eqref{1.3} $ (b).$  The  right-hand  inequality in \eqref{1.3} $ (b)$ 
   for  $  y  \in  B ( 0, 8 )  \sem B  (0, 1 ) $      follows from        
   \eqref{eqn2.6} $(ii), $ \eqref{1.6},   Lemma  \ref{lemma2.2} $(\hat a)$  
   with $ v = u $ and  the     estimate for $ v $  from above in \eqref{1.9}.   
   Likewise using    $ \log R - v  =  u  $   in these displays  and the estimate 
   from above for $ \log R - v $ in \eqref{1.10}  we   arrive at   the right-hand 
   inequality in \eqref{1.3} $ (b) $  for $ y \in B ( 0, R ) \sem B( 0, R/8)$. 
   To handle other values of  $ y $ we need some notation.  
   Given  a  real  valued   continuous function $ q $  defined on $ \ar B (0, s ), $  set 
        \[  
        m( s, q ) =   \min_{ x  \in \ar B (0, s )}  q (x)    \quad \mbox{and}\quad   M( s, q ) =   \max_{ x  \in \ar B (0, s )}  q (x).
        \]    
     If  $   2  <  t   < R/4 $  we note that  
    \[
     E =   \{ x ;\, \,  v ( x)  \leq  m ( t, v )  \} \subset  \bar B (0, t)  
     \]
      is connected and has  diameter at  least $ t$.  Also   
$\max ( v(x)  - m(t,v), 0  ) \geq 0   $   is  $ \mathcal{A}$-harmonic  in 
$  B ( 0, R ) \sem  E $ and $\equiv 0 $  on  $ E. $   Let   $ \be = \be_ t  $  
be the positive Borel measure with support  contained in $E$ corresponding to 
$\max ( v(x)  - m(t,v), 0  ) $ as in Lemma \ref{lemma2.4}. As in Lemma 4.2 of \cite{AGHLV}  we see that   
    $\be  (E)  = \nu ( \bar B ( 0, 1 )). $  So once again from Lemma \ref{lemma2.4}   we have 
\begin{align}   
\label{1.11}    
\begin{split}
1&\approx \nu ( \bar B( 0, 1 ))^{1/(n-1)}= \beta(E)^{1/(n-1)}\approx \max_{ \ar B (0, 2t) } [\max ( v(x)  - m(t,v), 0  )]  \\
&=M (  2t, v )  - m ( t, v ) 
\end{split}
\end{align}
for  $2 < t < R/4$.  Now  from  Lemmas \ref{lemma2.1}-\ref{lemma2.3} with  $ u  $  replaced by 
        $ v  -  m ( t, v ) $  and  \eqref{1.11}  we see that  
\begin{align} 
\label{1.12}    
M ( 3t/2,  |\nabla v | )  \leq    c  t^{-1}.
\end{align}
All constants in \eqref{1.11}-\eqref{1.12}  depend only on the data. 
Combining  \eqref{1.11}-\eqref{1.12}  we get the  right-hand inequality
 in \eqref{1.3} (b)  when $  4  \leq |y| \leq  R/4.$ Thus \eqref{1.3}(b)  is  valid.  
Now  \eqref{1.3}(c) follows from  \eqref{1.3}(b)  and integration.  
\end{proof} 
To avoid confusion we temporarily write  $ v ( \cdot, R )$ and $\nu ( \cdot, R ),  $  
for  $ v $ and $\nu$ in   Lemma \ref{lem1.2}. Using   Ascoli's theorem and Lemmas \ref{lemma2.1}-\ref{lemma2.3} 
for $  \mathcal{A}$-harmonic functions  we see as $ m \to \infty $   that  a
 sub-sequence  of  $  \{ v ( \cdot, m  )\}_{m=8}^\infty  $  
 converges uniformly on compact subsets of  
  $ \rn{n} $ to a  locally H{\"o}lder continuous function $  \hat V \geq 0$ which is
    $ \mathcal{A}$-harmonic 
  in   $ \rn{n} \sem  \bar B (0, 1 ) $  and $   \hat V \equiv 0 $ on $ \bar B (0, 1 ). $ 
   Let  $ \hat \Theta   $ be the measure corresponding to $ \hat V$ as in Lemma \ref{lemma2.4} and  put  
\begin{align*}
 \Theta(\cdot)=   \frac{ \hat{\Theta}(\cdot)}{\hat \Theta ( \bar B (0, 1 ) ) } \quad 
 \mbox{and} \quad V(\cdot) = \frac{\hat{V}(\cdot)}{[\hat \Theta ( \bar B (0, 1 ) )]^{1/(n-1)}}.
\end{align*}  
Then Lemma  \ref{lem1.2} is valid with $ v$ and $\nu $  replaced by   $ V$ 
and $\Theta$ in  $ \rn{n} \sem   B (0, 1 ), $  and  $  \Theta ( \bar B (0, 1 ) ) = 1. $     
Next we prove   
\begin{lemma}
 \label{lem1.3}
For $ V$ and $\Theta $ as above there exist  $ R_0 \geq 4$,  $b \in  C^{1, \si} ( B (0, 2 )   \sem \bar B (0, 1/2 ) )$,  
 $\ti \al  \in (0, 1)$,  $c_+  \geq 1$, and $ 0 <  \ga  \leq c_+$   where  $  \ti \al, \si,  c_+, $   depend only on the data,  satisfying         
\begin{align} 
\label{1.13}
 |  V ( x  )   -   \ga  \log |x|   -   b ( x/|x| ) |  \leq c_+  (R_0/|x|)^{   \ti  \al  } 
 \end{align} 
whenever $| x |  =  R   \geq  R_0  \geq   4$. 
\end{lemma}     

\begin{proof}  
  We   first note  that  $  h (x) :=  \lan x,  \nabla V (x) \ran$ is  a  weak  solution to  
\begin{align} 
\label{1.15}   
L h  =  \sum_{i, j =1}^n  \frac{\ar}{\ar x_i} ( \mathfrak{b}_{ij } (x)  h_{x_j} )   = 0 
\end{align}   
whenever    $x \in \rn{n} \sem \bar B (0, 1 ) $ where  
\begin{align}
\label{1.16}   
\mathfrak{b}_{ij} ( x )  =  \frac{\ar \mathcal{A}_i (\nabla V (x)) }{ \ar \eta_j }
\end{align}
 for almost  every $x \in  \rn{n}  \sem  \bar B (0, 1)$.  From  \eqref{1.16} and 
the structural assumptions on $ \mathcal{A} $ in  Definition  \ref{defn1.1}  it follows that   
 \begin{align} 
 \label{1.17}  
  \sum_{i,j = 1}^n \mathfrak{b}_{ij} (x)  \xi_i \xi_j       \approx   | \nabla V (x) |^{n-2} |\xi |^2  \quad \mbox{whenever}\, \,  \xi \in \rn{n} \sem \{0\}
\end{align} 
where  ratio constants depend only on the data.  From \eqref{1.15}-\eqref{1.17}  
and  \eqref{1.3} $ (b) $  we see for  $ r \geq 8  $    that $ h  $  is  a  bounded  
positive weak solution to  a uniformly elliptic PDE in     $ B ( 0, 2 r)  \sem  B ( 0,  r/2)  $      
            with bounded measurable coefficients. 
             
 We observe  that  $ m ( r, h ) $ as a function of $r$ is either non-decreasing 
  on  $ (1, \infty ) $  or  ultimately non-increasing  in the sense that there 
  exists $ t_1  >  1    $   with     $ m(\cdot, h) $ non-increasing on   $ [ t_1,  \infty).$   
  Also either  $ M (r, h ) $ is non-increasing on  $(1, \infty)$       
  or ultimately non-decreasing    on  $  (1, \infty) $ in the   
  sense that there exists $ t_1 > 1   $   with   $ M(\cdot, h) $ non-decreasing
   on   $ [ t_1,  \infty).$   Both statements follow from the 
   maximum-minimum  principles for $ \mathcal{A}$-harmonic  functions.   
  Thus   there exist constants $\be$ and $\ga$ such that
  \[  
  0 < \be  := \lim_{r \to \infty}  m ( r, h )  \leq \lim_{r \to \infty}  M( r, h ) =:\ga< \infty
  \] 
  thanks to  Lemma \ref{lem1.2}. 
 Then $ h  -  \min ( m ( r/2, h ), m (2r, h ) )  $ is non-negative
  in  $  B  ( 0, 2r  ) \sem B ( 0, r/2)  $  so by  Harnack's inequality
   for uniformly elliptic PDE in \eqref{1.15} we have, 
\begin{align} 
\label{1.17a}  
M ( r, h )  -  \min ( m ( r/2, h ), m (2r, h ) )   \leq  c  [ m ( r, h )  -  \min ( m ( r/2, h ), m (2r, h ) ) ] ,
 \end{align}   
 $ c $ depending only on the data.  Letting   $  r   \to  \infty $ in \eqref{1.17a} we obtain that 
 \begin{align}  
 \label{1.18}     
 \ga  =  \be <  \infty. 
 \end{align}
To get the estimate in  \eqref{1.13}  first assume that  $ M ( \cdot, h ), $   is  
non-decreasing on  $ [ t_1, \infty) $ for  some $ t_1 \geq 1. $    Then from \eqref{1.18} 
it follows that  $ m ( \cdot, h ) $  is  also non-decreasing on 
 $ [t_1,  \infty )$.      In this case   we  show the existence of $ t_2  \geq t_1 $ with   
 \begin{align}  
 \label{1.19}  
 h  \equiv  \ga  \quad \mbox{on}\, \, \mathbb{R}^n \sem B (0, t_2 ). 
 \end{align}
   From \eqref{1.18} and   Harnack's  inequality for $  L  $  in   
   \eqref{1.15} applied  to   
 $ h - m (  t,  h  )$  it follows that either  $ m ( t,  h )  = \ga  $  for some $ t  > t_1 $  in which case we put 
 $  t_2 =  t $  and observe  that   \eqref{1.19}  is valid  or   there exists $ t_2  > t_1 $  with  
  \begin{align}  
 \label{1.20} 
    \ga > m ( t, h ) > M ( t_1, h )  \quad \mbox{for}\, \,  t \geq t_2.   
\end{align}   Assuming \eqref{1.20} we claim   that  for  $ \mathcal{H}^1$-almost every 
$ \tau  \in [m(t_2, h), \ga], $   
\begin{align} 
\label{1.21}  
  I (  \tau )   :=  \int_{ \{ x \in  \mathbb{R}^n \sem B ( 0, t_1) :  \, \,   h(x) =  \tau, \, \,  \nabla h (x) \neq 0  \}}  \,  
\sum_{i, j = 1 }^n  |\nabla h |^{-1} \mathfrak{b}_{ij} ( x )     h_{ x_i}  h_{x_j}   d \mathcal{H}^{n-1} = \breve{a} 
\end{align} 
where  $ \breve{a} $ is  a  constant independent of  $ \tau \in  [m(t_2, h), \ga ]. $      
To prove this claim, let $\tau_1, \tau_2, \epsilon$ be so that      $m(t_2, h)< \tau_1 <    \tau_2  <   \ga $  and 
\[
0 <  \ep  <  \frac{1}{8} \min \{ \tau_1 -    m (t_2, h ), \ga - \tau_2, \tau_2 - \tau_1 \}.
\]     
Let   $ l     $  be infinitely differentiable on  
 $ \re $   with  $ l  \equiv 1 $ on  $ [ \tau_1, \tau_2],  $   $ l  \equiv 0 $  on 
 $ \re  \sem  [ \tau_1 - \ep,  \tau_2 +  \ep], $  and $ | \nabla l |   \leq c/\ep $ where $ c $ is an absolute constant.    Using   $ l \circ h $ as  a  test function in 
the  weak formulation of     \eqref{1.15} and using \eqref{1.16} and \eqref{1.17}  we see that     
\begin{align} 
\label{1.21a}    
0 =  \int_{\rn{n}} \sum_{i, j = 1}^n  \mathfrak{b}_{ij} ( x )     h_{ x_i} h_{ x_j} \,  (l' \circ h) dy =   \int_{\tau_1 -  \ep}^{\tau_1}
         I(s)  l'(s) ds   +     \int_{\tau_2 }^{\tau_2 + \ep}  I (s)   l' (s)  ds. 
\end{align}
 where we have used the version of the   coarea theorem in \cite{MSZ}  which is 
 permissible since   $  h  \in  W^{1,2}_{{\rm loc}} ( \rn{n} \sem \bar B(0, 1) )$.   
 From  Lemma \ref{lem1.2} applied to $ V, $  \eqref{1.16}, \eqref{1.17}, and once 
 again the coarea theorem,   we  deduce that   $ I $ is integrable on any compact 
 subset of  $ (m ( t_2, h ), \ga ). $  This observation,  \eqref{1.21a}, and the Lebesgue differentiation theorem  imply 
\begin{align}
	\label{1.21b}
	\begin{split}
	I ( \tau_2) - I ( \tau_1) =    \int_{\tau_1-\ep}^{\tau_1}   [I(s) - I(\tau_1)] l' (s) ds +  \int_{\tau_2}^{\tau_2 + \ep}   [I(s) - I(\tau_2)] l' (s) ds  \to 0 \quad \mbox{as}\, \, \ep \to 0
	\end{split}
\end{align} 
for $\mathcal{H}^{1}$-almost every $ \tau_1, \tau_2 \in [m(t_2, h ), \ga]$.  This proves claim \eqref{1.21}.      
	
Assuming $ \breve a   \neq 0$   and  that    $  t_0  > t_2 $   we  let  
\[
 \ti h (x) :=\left\{
 \begin{array}{ll}
   \max ( h (x)   - m( t_0, h  ),  0 ) &\mbox{whenever}\, \, x  \in  \rn{n} \sem \bar{B}(0, t_1),\\ 
0  & \mbox{whenever} \, \, x\in \bar{B}(0, t_1). 
 \end{array}
\right.
   \]
 Then from   \eqref{1.20}  we observe  that  
 \[
\lim_{x \to y} \ti h (x)  =  0  \quad     \mbox{whenever}\quad   y \in \ar B (0, t_1) .
  \]
Let  $ 0 \leq   \ph   \in  C_0^\infty ( B ( 0, 4t_0)) $ with $ \ph \equiv 1 $ on
  $ \bar B ( 0, 2t_0) $ and  $ |\nabla \ph |  \leq  c t_0 ^{-1}$ where  $c = c(n)$.  
  Then from  the above observation,   \eqref{1.20} and  Lemma  \ref{lem1.2}  
  we see that  $  \ti h \, \ph^2 $  can be  used  as  a  test  function in  the weak
   formulation  of   \eqref{1.15}. Doing this and using  H{\"o}lder's  inequality
     we  get the  Caccioppoli inequality,   
\begin{align}  
\label{1.22}   
 I_1 := \int   \sum_{i, j = 1}^n  \mathfrak{b}_{ij}  (x)  \ti h_{x_i} \ti h_{x_j}  \ph^2  dx   \,   \leq \,  c  \int  | \nabla V |^{n-2}  \ti h^2 \,  | \nabla \ph |^2  dx  =:  I_2.
 \end{align}      
 Using   \eqref{1.22},  the coarea theorem from \cite{MSZ}, and  \eqref{1.21}  once again we arrive at  
 \begin{align}
 \label{1.23} 
\begin{split} 
  I_1  &\geq  \int_{\{  x  :\, \,  0  < \ti h ( x ) < 
  m ( 2t_0, h ) - m (t_0, h )  \}}  
   \sum_{i, j = 1}^n  \mathfrak{b}_{ij}  (x)  \ti h_{x_i} \ti h_{x_j}  \ph^2  dx  \\
&= \int_{ 0}^{m ( 2 t_0, h )  -  m ( t_0, h ) }   \left( \int_{\{ \ti  h = t \}}  |\nabla \ti h |^{-1} \sum_{i, j = 1}^n  \mathfrak{b}_{ij}  (x)  \ti h_{x_i} \ti h_{x_j}   d\mathcal{H}^{n-1} \right) dt \\
  &=     \breve{a} \,  [  m( 2t_0, h ) -  m ( t_0,  h) ] 
\end{split}
\end{align}   
Also  from Lemma \ref{lem1.2} (b)   for $ V $ and the definition of $\ti h$ we find  
  \begin{align}
    \label{1.24}  
     I_2   \leq   c  [ M (4 t_0, h)  -  m(t_0, h ) ]^2  \leq c^2  [ m (2 t_0, h)  -  m(t_0, h ) ]^2 \neq 0
    \end{align}
 where to get the  last two inequalities we have used  \eqref{1.20} and Harnack's  
 inequality for  $ h - m( t_0, h )   $ in $ \mathbb{R}^n \sem \bar B (0, t_0). $     
 Using \eqref{1.23} and \eqref{1.24} we  conclude for some $ \ti c \geq 1 $  that  
   \begin{align}  
\label{1.25}   
\breve{a}  \leq  \ti c  [ m (2 t_0, h)  -  m(t_0, h ) ]. 
\end{align} 
 We  now  repeat this  argument with  $ t_0 $  replaced  by     $ 2^k  t_0$ for $ k = 1, 2, \dots $  
 in    \eqref{1.25}. Summing the  resulting  inequalities  we eventually
  get a  contradiction to   \eqref{1.18} unless  $\breve a = 0. $   If $ \breve{a} = 0 $  
we observe  that    
\[   
B ( 0, 4 t_0 )   \sem  \bar B ( 0, t_0 )  \subset  \{ x :   m ( t_0,  h ) <   h ( x ) < M ( 4 t_0,  h ) \}.  
\]   
 Using this observation, the coarea theorem from \cite{MSZ}, \eqref{1.21}, 
 and arguing as in \eqref{1.23} we get  $ I_1 = 0. $  It then follows 
 from  \eqref{1.17}  that $ \nabla  \ti h  \equiv 0  $ in  $ B ( 0, 2t_0) \sem \bar B ( 0, t_0) $ 
 and thereupon  from Harnack's inequality that  $ h \equiv \ga $  
 in   $ \mathbb{R}^n \sem  B (0, t_0 ),  $ a contradiction to   \eqref{1.20}.     
  Since $ t_0 > t_2 $ is arbitrary we conclude  from all these 
  contradictions  that \eqref{1.19} holds  when  $ M   ( \cdot, h )$ 
  and $m  (\cdot, h ) $   are both non-decreasing on  $ [t_1, \infty )$.   
    
If   $ M   ( \cdot, h )$ and $m  (\cdot, h ) $  are both non-increasing  on  
$ [t_1, \infty ) $  then either  $  M ( t,  h )  =  \ga $  for some $ t  \geq t_1 $  in which  case we put 
$ t_2  =  t $   and observe that \eqref{1.19} holds  or there exists 
 $ t_2  > t_1 $  so that    
\begin{align}
\label{1.25a}  
\ga < M ( t, h )   <  m ( t_1, h )  \quad \mbox{for}\, \,  t  \geq t_2. 
\end{align}  
In  this case   let    $ t_0 > t_2 $   and define 
\[
\hat  h(x)   =
\left\{
\begin{array}{ll}
\max(M( t_0, h ) - h(x), 0) & \mbox{whenever} \, \, x\in \mathbb{R}^n \sem \bar{B} ( 0, t_1), \\
0 & \mbox{whenever} \, \, x\in B ( 0, t_1 ).
\end{array}
\right.
\]
Repeating the argument from \eqref{1.22}-\eqref{1.25} with  $ \ti h $ replaced by 
$ \hat h $  and using  arbitrariness of  $ t_0 $   we  eventually  arrive at  a  contradiction
  so  \eqref{1.19} is valid. The only other possible case in view of  \eqref{1.18} is  that  
   there exist $ t_1 > 1,  $  such that $  M ( \cdot, h )  $ 
is  non-increasing  on $ [t_1, \infty) $   and $ m ( \cdot, h ) $ is non-decreasing on  
$ [t_1,   \infty ).  $   In this case we see that either \eqref{1.19}  is valid for some 
$ t_2  \geq  t_1 $  or     
\begin{align} 
\label{1.25b}     
m ( t_0, h )  <  \ga  <  M ( t_0, h )  \quad \mbox{for all}\, \,  t_0  \in  [ t_1, \infty).  
\end{align}
 In this case  we can  apply Harnack's  inequality to   
 $ M ( t_0, h )  -  h $ and  $  h  -  m (t_0, h ) $  in  $ \rn{n} \sem \bar B ( 0, t_0) $  
 whenever $  t_0 >  t_1$ to get
\[
M ( t_0, h )  -  m(2t_0, h) \leq  c_1 [M ( t_0, h )  -  M(2t_0, h)]
\]
and
\[
M(2t_0, h)  -  m (t_0, h )\leq c_1 [m(2t_0,h)  -  m (t_0, h )].
\]   
Adding  these two inequalities  and  doing  some arithmetic  we obtain 
\begin{align} 
\label{1.26a}  
M ( 2 t_0, h )  -  m ( 2t_0, h )  \leq  \kappa ( M (  t_0, h )  -  m ( t_0, h ) )
\end{align} 
 where $\kappa=\frac{c_1-1}{c_1 + 1 }\in (0,1)$.  
Iterating this inequality with  $ t_0 $  replaced by  $  2^k t_0,  k  = 1, 2, \dots,  $  
 it follows  that if $ s \geq 2t_0, $ then  
\begin{align}
 \label{1.26}  
 \begin{split}
 M ( s, h )  -   m(s, h ) \leq  c (t_0/s)^{\ti\al} ( M (  t_0, h )  -  m ( t_0, h ) ) \leq c^2  (t_1/s)^{\ti \al}
\end{split}
\end{align}
 for some $ c  \geq 1, $  depending only on the data.  
 
 Finally, we are in a position to 
 prove \eqref{1.13}  in Lemma  \ref{lem1.3}.   
 First if   \eqref{1.19}  is valid, then  using  $  r^{-1}  h ( r \om  )  =    \frac{\ar V ( r  \om) }{\ar r} $ 
 when $ |x| =r$ and $\om = x/ |x| \in \mathbb{S}^{n-1} $  
 we find upon integrating  \eqref{1.19} from $ r = R  \geq t_2 $  to  $ \tau  > 2R $  
  that  
\begin{align}
\label{1.26b}  
V (\tau \om  )  - V ( R \om) = \ga  \log (\tau / R).  
\end{align}
   Otherwise,   (i.e. \eqref{1.25b} holds and therefore \eqref{1.26} holds)    
   we find  from \eqref{1.26}    that  if $  \om \in \mathbb{S}^{n-1}$   then  
   \[    
   | h ( s\,  \om )  -  \ga  |  \leq  c_{++}  (t_1/s)^{  \ti \al } \quad \mbox{for}\,\, s \geq    2R  \geq 4 t_1 
   \] 
 where  $ c_{++} $  depends only on the data.   
 Integrating this inequality with respect to $ s $  
 from $ R $ to $ \tau > R $    we get  as in \eqref{1.26b}  that 
\begin{align}
 \label{1.27}  
  | V ( \tau \om )  - V ( R  \om )  -  \ga \log (\tau /R) |  \leq  c_+ ( t_1/R)^{\ti \al}  
\end{align}
  whenever $ \om \in \mathbb{S}^{n-1}$ and $ \tau  \geq 2R \geq 4 t_1. $ If  either 
  \eqref{1.26b} or \eqref{1.27} holds we let   
    \[   
    \psi_m ( y ) =  V ( m y/|y| ) - \ga \log m \quad \mbox{whenever} \, \, y  \in  
    \mathbb{R}^n \sem \{0\}  \quad \mbox{for}\, \, m =  2, 3, \dots. 
    \]  
Note that $\psi_m$ is a homogeneous of degree zero function for $m=2,3,\ldots$.     
From  Lemma \ref{lem1.2} $(b) $ for  $ V $ in $ \rn{n} \sem \{0\} $ 
and Lemmas \ref{lemma2.1}-\ref{lemma2.3}, we  see that   
$  \nabla \psi_m  \in  C^{0, \si } ( B (0, 4) \sem B (0, 1/4) ) $  for some 
    $ \si \in (0, 1) $  with  norm $ \leq c, $  while from  \eqref{1.26b} or \eqref{1.27}   
    we deduce  for  $ m  \geq 2R $   that  
    $ \{\psi_m\}$   is uniformly bounded by a  constant  depending on $ R. $   
    Using  Ascoli's theorem we conclude that 
    a sub-sequence of  $ \{\psi_m\} $  say  $  \{\psi_{m_l}\} $  converges uniformly
     in  the $ C^{1,\si} (B (0,2) \sem B(0,1/2) ) $ norm to   $ b   $, 
   a homogeneous  degree 0 function on $ \mathbb{R}^n \sem \{0\}. $    
   Letting $ \tau = m_l$,  $\om = x/|x|$  in  \eqref{1.26b} or  \eqref{1.27}, and $  l \to \infty $ we conclude 
  \eqref{1.13}    for a  fixed $ x $  with $ |x| = R \geq  R_0 $ where $ R_0 = 2 t_2 $ if 
  \eqref{1.26b} holds and  $ R_0 = 2 t_1 $ if \eqref{1.27}  holds.    
  Now letting  $ R $ vary   we get Lemma \ref{lem1.3}. 
  \end{proof}  
  Finally, we  prove    
  \begin{lemma}
 \label{lem1.4}  
Let $ \mathcal{A}\in M_n(\alpha)$. Then there exists a  unique (up to an additive  constant) fundamental solution $ F$ 
to the equation $\nabla\cdot  \mathcal{A}(\nabla F )=0$
in $\mathbb R^{n}$, with pole at $0$,  in the sense of \eqref{1.2}.  Furthermore, there exist
$b \in C^{1,\si}(B (0, 2 ) \sem B (0, 1/2))$, $\ga  > 0, $ and  $ \ti c   \geq 1 $ such that
\begin{align}
\label{1.27a}
\begin{split}
& (+) \hs{.24in}  F(z)= \ga \log |z| + b(z/|z|)  \quad \mbox{whenever} \, \, z\in\mathbb R^{n}\setminus\{0\},  \\
&(++) \hs{.2in} \ti c^{-1}  \leq  \lan z,  \nabla F (z) \ran  \leq   |z | \nabla F (z) |  \leq  \ti c    \quad \mbox{whenever}\, \, z  \in \mathbb{R}^n \sem \{0\}.  
\end{split}
\end{align}
\end{lemma}  
\begin{proof} 
Let  $ x  =  m y,  |y| \geq 1/m, $ for $ m  \geq 4  R_0^2$ where $R_0$ is as above, and put  
\[  
F_m (y)  =   V  (x)  -  \ga \log m \quad \mbox{for}\,\,   y \in\rn{n} \sem \bar B ( 0, 1/m).  
\]  
Then from Lemma  \ref{lem1.3} with  $ R = |x| $  we  get for $ |y| \geq m^{-1/2} $ that    
\begin{align}
\label{1.28}  
\begin{split}
|F_m ( y ) -  \ga \log | y |  -  b(y/|y| ) |   &= |V(x)-\ga \log |y| - \ga \log m - b(y/|y|)| \\
&=|V(my)-\gamma\log m|y|- b(my/(m|y|)| \\
&\leq c (R_0/m|y|)^{\al} \leq c  ( R_0^2 /m)^{\al/2} .
\end{split} 
\end{align}  Also  from the chain  rule and  Lemma  \ref{lem1.3} for $ V $ we have if  $ |y| > 1/m $ that 
\begin{align} 
\label{1.28a} 
c^{-1}  \, \leq  \,  \lan \nabla F_m ( y ), y  \ran  \leq  |y| \, | \nabla F_m (y) |   \leq c  
 \end{align}  
 where $ c \geq  1 $ depends only on the data.  From 
  Lemma  \ref{lem1.2},  \eqref{1.28},  \eqref{1.28a},  as well as Lemmas \ref{lemma2.1}-\ref{lemma2.3}, 
  we observe  that the $ C^{1,\si} $ norms of  $ \{F_m\}_{\{m\geq 16\}} $    
    are uniformly bounded on a given compact subset of  
   $ \rn{n} \sem  \{0\} $ for $m$ large enough.  From this observation and  \eqref{1.28} 
   we conclude that 
   \[  
   \lim_{m\to \infty} F_m ( y )  =    \ga \log |y|  +  b(y/|y|) = F(y)  \quad \mbox{whenever} \, \, y \in \rn{n} \sem \{0\}.
   \] 
Thus, \eqref{1.27a}  $ (+) $  is proved.  Also \eqref{1.27a} $(++)$ 
   follows from  \eqref{1.28a}  and uniform convergence of    $ \nabla  F_m  \to \nabla F $  on 
compact subsets of  $ \rn{n} \sem \{0\}. $ Moreover, $ F $ is  $ \mathcal{A}$-harmonic in $ \rn{n} \sem  \{0\}.$        
   Extend $ F_m + \ga  \log m  $ to  a  continuous function on $ \rn{n} $ 
   by putting this function $ \equiv 0 $ on $ \bar B (0, 1/m). $   Let $ \mu_m $  
   denote the measure as in Lemma \ref{lemma2.4} $(i)$ corresponding to  
   $ F_m +  \ga \log m$, with support contained in 
    $ \ar B (0, 1/m). $  Once again using invariance of  $ \mathcal{A}$-harmonic 
    functions under dilation and  translation  we  deduce that
    \[  
    \mu_m ( \bar B (0, 1/m) ) = \Theta (\bar B(0,1) ) = 1. 
    \]  
       Using uniform convergence of    $ \nabla  F_m  \to \nabla F $  on 
compact subsets of  $ \rn{n} \sem \{0\} $  and  \eqref{1.28a}  
we conclude for $ \he \in C_0^\infty( \rn{n} ) $    that  
\begin{align} 
\label{1.29} 
\int  \lan \mathcal{A} ( \nabla F ),  \nabla  \he \ran  dx =      
 \lim_{m \to \infty}  \int  \lan \mathcal{A} ( \nabla F_m ),  \nabla  \he \ran  dx 
 =  -  \lim_{m\to \infty}  \int \he d  \mu_m=  - \he (0). 
\end{align}  
    Thus \eqref{1.2} $ (i)$ and $(ii)$ are  valid   so $ F $ is  a  fundamental solution in the sense of Definition \ref{def4.1}.    
    
To prove uniqueness of  $ F $ up to a constant,  suppose $ F_1 $ also 
satisfies \eqref{1.2} $(i)$ and  $(ii)$. We assume, as we may,  that $  F_1 (  e_1 ) = 1 $ where  $  e_1 = 
 (1,0, \dots, 0).$   Then
 \[
 A =  \lim_{t \to 0 }  m ( t,  F_1) \quad \mbox{and} \quad B  =   \lim_{t \to \infty }  m ( t,  F_1)
  \]
  in the extended sense thanks to the minimum principle for 
  $ \mathcal{A}$-harmonic functions. We note that  if  $A$ is finite  
  then   $ F_1 \to A $ as $ x \to 0 $  while if  $ B $ is finite then  $ F_1 \to B $ as 
  $ x \to \infty.  $   This note is proved using Harnack's inequality as in   \eqref{1.17a}.
   We  first show that  
\begin{align}
\label{1.30}  
A = - \infty.   
\end{align}  
Indeed, if   $ A \neq - \infty, $ then  from the above note and the 
maximum-minimum principles for   $  \mathcal{A}$-harmonic  functions  we see that  
 $  A   = \lim_{t\to 0 }  M ( t, F_1)$  and $  B  \neq A. $     
  In  this case if  $   A  >  B,   $  let   $s, t $ satisfy 
    $   A  > s > t > B $  and set  
    \[   
    \hat \he ( x )  =  \max(  \min ( F_1, s ),  t ) - t. 
    \]   
    Approximating $ \hat \he $ by $ C_0^\infty ( \rn{n} )  $  functions  
    we see that  $ \hat  \he $  can be used as  a  test function in   \eqref{1.2} $ (ii). $  
    Doing this  it follows  that    
\begin{align} 
\label{1.31}    
 -\hat \he ( 0)= -s+t =  \int_{ \{ x :\,  t <  F_1 (x)  <  s\} }   \lan \mathcal{A}(\nabla F_1), \nabla F_1 \ran  dx .
\end{align}    
which is  a contradiction since the  right-hand integral is always non-negative 
as we see from     the structural assumptions on  $ \mathcal{A}.$  If $  A  < B  $ we  suppose  
$ A < s < t  <  B $  and  use    
\[  
\hat \he ( x )  =  \min(  \max ( F_1, s ),  t ) - t    
\]   
as a test function  in   \eqref{1.2} $ (ii) $  in order to obtain   \eqref{1.31}.  
          Since $ A \neq - \infty $ we can let  $ s \to  A$ and   
          take  $ t  \geq   M ( 1, F_1 ) $   in order to   conclude from  \eqref{1.31} that  $ F_1 \in W^{1,n} 
    (B (0, 1 ) \sem \{0\}).  $      
    Using this  conclusion and the fact that a point has  zero    $n$-capacity
      one can now show  that   $ F_1 $  extends to a $ \mathcal{A}$-harmonic
       function in  $  B (0, 1) $ which  easily leads to a contradiction. 
       To see this contradiction, given $ \he \in C_0^{\infty } ( B ( 0, 1 ) ) $  with  $ \he (0 ) \neq 0$ 
       and $\epsilon>0$, we let  $ \he_\ep =  \eta_\ep  \, \he  $  where
\begin{align}  
\label{1.32}   
 \eta_\ep ( x )   =   \frac{ \max (\log ( |x|/\ep ), 0 ) }{ \log(1/\ep) } \quad \mbox{whenever} \, \, x  \in B ( 0, 1 ). 
 \end{align}  
 It is easily shown  that  
\begin{align} 
\label{1.33}    
0 \leq \eta_\ep  \leq 1,  \quad  \int_{B(0, 1 )} | \nabla \eta_\ep|^n  dx \to 0, 
\quad  \mbox{and}\, \,\,     \eta_\ep \to 1 \, \,   \mbox{a.e in $B(0,1)$}\quad  \mbox{as $ \ep \to 0$}.
\end{align}
        Using  $  \he_\ep $ as  a test function in  \eqref{1.2} $(ii)$  
        and letting  $ \ep \to 0 $  it  follows from \eqref{eqn1.1},   $ F_1 \in W^{1,n} 
    (B (0, 1 ) \sem \{0\}),  $  \eqref{1.33}, H{\"o}lder's inequality, 
     and  \eqref{1.2} $(ii)$ for $ \he $   that    
  \begin{align} 
  \label{1.34}
  \begin{split}
0  &=  \int_{B(0,1)}  \lan  \mathcal{A} ( \nabla F_1 ),  \nabla \he_\ep \ran  dx \\
&= \int_{B(0,1)}  \he \lan  \mathcal{A} ( \nabla F_1 ),  \nabla \eta_\ep \ran  dx +  \int_{B(0,1)} \eta_\ep \lan  \mathcal{A} ( \nabla F_1 ),  \nabla \he \ran  dx\\
&\leq \int_{B(0,1)}  \he  |\nabla F_1|^{n-1}  |\nabla \eta_\ep|  dx + \int_{B(0,1)}  \eta_\ep \lan  \mathcal{A} ( \nabla F_1 ),  \nabla \he \ran  dx \\ 
& \leq c \left(\int_{B(0,1)}   |\nabla F_1|^{n} dx\right)^{(n-1)/n} \left(\int_{B(0,1)} |\nabla \eta_\ep|^{n}  dx\right)^{1/n} + \int_{B(0,1)}  \eta_\ep \lan  \mathcal{A} ( \nabla F_1 ),  \nabla \he \ran  dx \\
&\to 0+\int_{B(0,1)} \lan  \mathcal{A} ( \nabla F_1 ),  \nabla \he \ran  dx  = - \he (0)  \neq 0  \quad \mbox{as $\ep\to 0$.}
\end{split}
\end{align} 
From this contradiction in \eqref{1.34} we conclude that \eqref{1.30} holds.  
Next we use  \eqref{1.30}  and  \eqref{1.2} $(i)$  to show that  
\begin{align}
\label{1.35}  
d =   \lim_{t \to 0}  M  (  t,  F_1 ) = - \infty. 
\end{align}      
The proof     is by contradiction.   If   $ d  \neq  \pm \infty,  $  then 
from Harnack's inequality we  find that      $ d  =    \lim_{t \to 0}  m  (  t,  F_1 ), $ 
 a contradiction to  \eqref{1.30}.
If  $ d = \infty $    then  from the maximum-minimum principles
  for  $  \mathcal{A}$-harmonic functions we deduce the  existence
   of $ t_1 > 0 $ with    $ M ( \cdot, F_1 ) $  non-increasing and  $ m ( \cdot, F_1 ) $
  non-decreasing on $ (0, t_1)$.  So  arguing as  in  \eqref{1.26}  we get  
  for some $ \ti \al \in (0, 1 ) $ and $ 0 < 2t < s  < t_1 $  that   
\begin{align} 
\label{1.36}  
M ( s, F_1 )  -   m(s, F_1 )  \leq  c (t/s)^{\ti \al} ( M ( t,  F_1 )  -  m (t,  F_1 ) ).
\end{align} 
Fix $ s <  t_1 
 $ so that  $ m ( s,  F_1) < 0$ and $M ( s, F_1 ) > 0 $,  and choose $ x $ so that  
 $ |x|  = t$ and $F_1  ( x )  =  \max(M ( t, F_1 ), - m ( t, F_1 )).$  
 Then from  \eqref{eqn2.1}  we have for some  $\bar c  \geq 1, $ 
\begin{align} 
\label{1.36a}     
| F_1 ( x) |^n  \leq  \bar c \,     t^{- n}    \int_{B( x, t/8)}  | F_1 (y ) |^{n} dy. 
\end{align} 
From  \eqref{1.36}, \eqref{1.36a}, and  H{\"o}lder's inequality it follows that if    $ \la = 2/\ti \al, $  then   
\begin{align} 
\label{1.36b}  
[ M ( s, F_1 )  -   m(s, F_1 ) ]^{n\la} \,  (s/t) ^{ 2 n } 
\leq ( \hat c | F_1 ( x ) |)^{n \la}   \leq  \hat c^{2 \la n} t^{-n} \int_{B (x, t/8)}  | F_1 (y)  |^{n \la}  \, dy  
\end{align}
where $ \hat c \geq 1 $  depends only on the data.   Multiplying both sides of   \eqref{1.36b}  by $ t^n $  it  follows that  
\[   
\int_{B ( x, t/8) }  | F_1 ( y )|^{n \la}  dy \to  \infty  \quad \mbox{as}\, \, t\to 0. 
\]
We have reached a  contradiction since by  \eqref{1.2}  $(i)$  and  
Sobolev's inequality we  have   $  |F_1|^{n\la} $  integrable 
on  $ B (0, 1). $ Thus \eqref{1.35} is valid.                
    We put  $ F_1 ( 0 ) = - \infty $ and observe  from \eqref{1.35} 
    that $ F_1 $ is continuous in the extended sense on $ \rn{n}. $ 
     
  The  next  step in our proof  of   Lemma \ref{lem1.4}  is to show  using    \eqref{1.35} 
  that for  $ 0 < t < \infty, $ 
  \begin{align} 
\label{1.36c}   
\begin{split}
&(a)  \hs{.2in}  \{ x :  F_1 (x)   <  m ( t, F_1 )  \}  \,\,   \mbox{is  open, bounded, connected, and contains 0},  \\   
&(b)  \hs{.2in}    m (  t, F_1 ) \, \, \mbox{and}\, \, M ( t,  F_1 )  \,\,   \mbox{  are strictly increasing on}\, \,  (0, \infty).
\end{split}
\end{align} 

To prove     \eqref{1.36c} $(b)$  observe  from   \eqref{1.35}  and the maximum principle  for      $ \mathcal{A}$-harmonic functions  that  $ M ( \cdot,  F_1 ) $  is  at least  non-decreasing on  
$ (0, \infty ). $   Also  if   $ M ( \cdot , F_1) $  is not strictly increasing  then  from Harnack's inequality for      $ \mathcal{A}$-harmonic functions it follows that  $ F_1  \equiv  $ constant in   $ B ( 0, t_2 )  \sem \bar B (0, t_1)  $   for some  $ 0  < t_1 < t_2 < \infty.  $   To  get a  contradiction choose  $  \he \in C_0^\infty ( B (0, t_2) ) $ with $  \he  \equiv  1  $ in  $ \bar B ( 0, t_1)  .$    Using  $  \he  $  as  a  test function  in  \eqref{1.2} $(ii) $  we obtain   a  contradiction,  since then the  integral in this display is  zero.    

 To prove \eqref{1.36c} $(b)$    for  $ m ( \cdot, F_1 ) $  by way of  contradiction,  observe  by the same reasoning as  above,   that if    $(b)$ is  false for   this function,   then  there  exists  $ t_1 \in ( 0, \infty ) $  with    $  m (\cdot , F_1)  $ strictly increasing on   $ ( 0, t_1 ) $ and strictly decreasing on $ [t_1,  \infty).$      In this  case we can apply  the same argument as  in  \eqref{1.36}  to  find for some $ \ti \al  \in ( 0, \infty) $ that 
 \begin{align} 
\label{1.36d}  
M ( s, F_1 )  -   m(s, F_1 )  \,  \geq  \, c^{-1}  (s/t_1)^{\ti \al} \,  ( M ( t_1,  F_1 )  -  m (t_1,  F_1 ) ).
\end{align} 
whenever   $  s    >  t_1 . $   Letting  $  s   \rar \infty $  in \eqref{1.36d} 
we   obtain  a     contradiction   to our assumption that    
 $| F_1  ( x ) |   =  O  ( \log |x| )  $   in  a  neighborhood of  $  \infty.  $  
 
   From    \eqref{1.36c}  $ (b) $   and the discussion above  $\eqref{1.30}  $  it follows that 
     \[  \lim_{ t  \rar  \infty}  M  ( t,  F_1)  =  \lim_{ t  \rar  \infty}  m  ( t,  F_1)      \] in possibly the extended sense.    \eqref{1.36c}  $ (a) $ follows from this  display, continuity of  $ F_1 $ 
in the extended sense in  $\mathbb{R}^n, $    
and the maximum  principle  for   $ \mathcal{A}$-harmonic functions.

Using \eqref{1.36c}  we   show for  given  $  r  > 0 $  
that   $   k =   k ( \cdot, r ) =  \max(  F_1 -  m ( r, F_1),  0 ) $  is  an  $ \mathcal{A}$-subsolution 
on $  \rn{n} $  with  corresponding  measure  $ \mu = \mu( \cdot, r ) $  satisfying
\begin{align} 
\label{1.37}   
\mbox{supp}(\mu)\subset  \{ x  :\,   F_1(x) =  m( r, F_1 )   \} \quad \mbox{and}\quad       \mu ( \{ x  : \,  F_1(x) =  m( r, F_1 )   \} )  = 1.
     \end{align}
   To prove   \eqref{1.37} we first note  from \eqref{1.36c}  that 
           $ k\in W^{1,n}_{\mbox{\tiny loc}} ( \rn{n} ). $     
           Let   $  \psi  =    ( k +  \ep_1)^{\ep_2}  - \ep_1^{\ep_2}   \,$ where 
   $ \ep_1, \ep_2 > 0 $   and suppose $ 0 \leq \he \in C_0^\infty (\rn{n} ). $  
   Using   $ \psi \,  \he  $   as  a test function in  \eqref{1.2} $(ii)$ for    $  \ep_1 > 0 $  small,  we  obtain  
\begin{align}
 \label{1.38}  
 \begin{split} 
 0  & = \int \lan \mathcal{A} ( \nabla F_1),  \nabla k \ran \, \ep_2  ( k + \ep_1)^{(\ep_2 - 1) } \he    dx  +  \int \lan \mathcal{A} ( \nabla F_1),  \nabla \he  \ran   \psi dx \\
 &\geq  \int \lan \mathcal{A} ( \nabla k),  \nabla \he  \ran   \psi  \,   dx.   
\end{split}
 \end{align}
To get  \eqref{1.38} we have used 
\[
\int_{\{ x:\,   F_1(x) - m( r, F_1 )>0\} }\lan \mathcal{A} ( \nabla F_1),  \nabla F_1 \ran \, \ep_2  ( k + \ep_1)^{(\ep_2 - 1) } \he \,    dx>0
 \]
 and also that 
 \begin{align*}
0&\geq  \int \lan \mathcal{A} ( \nabla F_1),  \nabla \he  \ran   \psi dx \\
&= \int_{ \{ x:\, F_1(x) - m( r, F_1 )>0\} } \lan \mathcal{A} ( \nabla k),  \nabla \he  \ran   \psi  \,   dx \\
&= \int \lan \mathcal{A} ( \nabla k),  \nabla \he  \ran   \psi  \,   dx.
  \end{align*}
Letting  $  \ep_1 \to 0 $  first and  then  $ \ep_2  \to  0 $  in  \eqref{1.38},  
we conclude from the Lebesgue  dominated converge theorem and 
arbitrariness of  $ \he $  that $ k $ is an  $\mathcal{A}$-subsolution 
 on $ \rn{n }, $  so there exists  $  \mu $  as defined above  
  \eqref{1.37}. Let $ K  =  \{ x:\, \, F_1(x)   \leq  m ( r, F_1 ) \}.$  Then from  \eqref{1.36c} we see that  $ K $ is compact  so there exists   
                           $ 0 \leq \he \in C_0^\infty (\rn{n} ) $ with 
                             $  \he  \equiv 1 $ on an open set  containing $ K. $ Then from  \eqref{1.2} $ (ii)$  
                             and the integral inequality  involving $  k, \mu $  we obtain that 
 \begin{align}
 \label{1.39}  
\begin{split} 
  \mu  (   K  )   =  \mu ( \rn{n} )& = \int  \theta d\mu=  - \int  \lan \mathcal{A} ( \nabla k),  \nabla \he \ran  dx   \\
  &=  -    \int   \lan \mathcal{A} (\nabla F_1),  \nabla  \he  \ran dx  =\theta(0)= 1.   
\end{split}
\end{align} 
  Also  $  \mu  ( \{ x : F_1 ( x )  <  m ( r, F_1 )  \} )  = 0 $ since this set is open, bounded,  and     
  $  \nabla k  = 0  $   on it.   Using arbitrariness of $ \he $
   and regularity of $ \mu $  we see that   \eqref{1.37}  is valid.    

 Next  from  \eqref{1.36c} $(a)$  and  facts about $n$-capacity      
 we deduce   that  $ \{ x: \, \, F_1 (x)  \leq   m ( r, F_1) \} $ 
   is uniformly $(4r,n) $-fat.   
 Using this  fact  and  \eqref{1.37}  we  see  as in \eqref{1.3} $(a)$  that  
\begin{align}   
\label{1.40}   
M( 4r, k )  =    M ( 4r, F_1) -  m ( r,  F_1)  \approx   m ( 2r, F_1) - m ( r, F_1)  \approx  1
\end{align}
whenever $0< r < \infty$. From  \eqref{1.40}  and 
Lemmas \ref{lemma2.1}-\ref{lemma2.3},  it follows that 
 \begin{align}   
 \label{1.41}   
 M ( 2r, |\nabla F_1| )  \leq c r^{-1}.
\end{align}   
Now \eqref{1.40}, \eqref{1.41},    and $ F_1 ( e_1 ) = 1 $    also imply that     
\begin{align}   
\label{1.42}  
 -  F_1 ( x  )  \approx   -  \log |x| \, \, \mbox{whenever}\, \,   |x|  \leq 
 1/2  \quad \mbox{ and } \quad F_1 ( x )  \approx \log |x| \,\, \mbox{whenever} \, \,  |x| \geq 2
 \end{align} 
  where all  constants in  \eqref{1.40}-\eqref{1.42} depend only on  the data.   
  Indeed,  if  $ | x | \geq 2,$  choose  $ k $  so that  $ 2^{k} \leq |x| \leq  2^{k+1}.$  
  Putting  $ r =  2^l $ in  \eqref{1.40} and  summing the resulting 
  inequality from  $ 0 $ to  $ k $  we obtain  \eqref{1.42} after simple estimates.  
  A  similar argument gives   \eqref{1.42}  if   $ |  x | \leq 1/2.  $    
     To  get  a  better inequality we note that   if $  F_2  = \zeta F$,  $\zeta  $ a real number,  then  
\begin{align} 
\label{1.43}   
\hat L ( F_1 - F_2)  =  \sum_{i, j =1}^n  \frac{\ar}{\ar x_i} ( \hat{\mathfrak{b}}_{ij } (x)  (F_1 - F_2)_{x_j} )   = 0 
\end{align}
weakly in $ \rn{n} \sem \{0\}$   where 
\begin{align} 
\label{1.44} 
 \hat{\mathfrak{b}}_{ij} ( x )  =   \int_0^1  \frac{\ar \mathcal{A}_i ( t \nabla F_1 (x)  +   (1-t)  \nabla F_2  (x ) )}{ \ar \eta_j }  dt
\end{align}
  for almost  every $ x \in  \rn{n}  \sem  \{0\}$.     From \eqref{1.44} and  
   the structural assumptions on $ \mathcal{A} $ it follows that for 
  fixed $ x \in \rn{n} \sem \{0\}, $    
\begin{align}
\label{1.45}   
\sum_{i,j = 1}^n \hat{\mathfrak{b}}_{ij} (x)  \xi_i \xi_j       \approx  
(| \nabla F_1 (x) | + |\nabla F_2 (x)|)^{n-2} |\xi |^2  \approx  
 [( 1   +  |\zeta| ) | x |]^{2 - n} |\xi |^2  
 \end{align}
whenever $\xi \in \rn{n} \sem \{0\}$. Here  to get the last inequality 
we have  used   \eqref{1.41}   and  \eqref{1.27a} 
 $(++)$.  Ratio constants depend only on  the data.  From 
  \eqref{1.43}-\eqref{1.45}     we see for  $ r \in ( 0,  \infty)   $    
 that $ F_1 - F_2   $  is  a  bounded  weak solution to  a 
 uniformly elliptic PDE  with  bounded measurable  coefficients in     $ B ( 0, 2 r)  \sem  B ( 0,  r/2)  $. 
 Moreover, ellipticity constants and  $ L^{\infty} $  bound on the  coefficients depend only on the data  and $ \zeta. $    
   Put    
 \[    
 a_1 =   \liminf_{r \to \infty  } m(r,  (F_1 - F)/F ). 
 \]   
 From   \eqref{1.42}  and our knowledge of $ F, $   we see that $  F_1 \approx  F $  
 so that $ - 1  < a_1 \leq  c, $    where  constants depend  only on the data.   
    Using this fact and  Harnack's inequality  it follows that 
\begin{align} 
\label{1.46}      
a_1 =  \lim_{r \to \infty } M(r,  (F_1 - F)/F ) =      \lim_{r \to \infty } m(r,  (F_1 - F)/F ).  
\end{align}     
To outline the proof of \eqref{1.46}   observe  from  the definition of  $ a_1 $  
that there exists  an increasing sequence   $ \{r_j\} $  with $ r_j \to \infty $  as $ j \to  \infty$  such that   
$ m (r_j, F_1 - F )/F ) \to a_1 $ as $j \to \infty$.  Then  for $ j $ 
large enough, say  $j \geq j_0, $  and fixed   $ \ep>0,  $  small we have    
   $ F_1 - F  - ( a_1   - \ep) F  \geq  0  $ on $  \ar B ( 0, r_j).  $  
    Using the  maximum principle  for  $  \mathcal{A}$-harmonic 
    functions in   $  B (0, r_{j}) \sem B ( 0, r_{j_0} ) $  and letting $ j  \to  \infty $
     it   follows that    $ F_1 - F  - ( a_1   - \ep) F  \geq   0  $ 
     on $  \mathbb{R}^n \sem   B (0, r_{j_0} ). $   
     Next we choose $ j_1 $  so that   $ r_{j_1} \geq 2 r_{j_0} $    
     and    $  F_1   -  F  \leq  (a_1 + \ep ) F $  at  some point 
     on  $  \ar  B (0, r_j )$  for $ j \geq j_1. $       Using this  
     choice and the previous facts,   as well as \eqref{1.43}-\eqref{1.45} 
     with  $ F_2 =  (1 + a_1  -    \ep)  F, $   we see that  
     Harnack's  inequality  can be  applied  on  
   $  \ar B (0, r_j )$ for   $ j \geq  j_1. $     Thus   
\begin{align}
\label{1.46a}  
\max_{\partial B(0,r_j)} [F_1  - F -  (a_1 - \ep ) F ] 
\leq c \min_{\partial B(0,r_j)} [F_1  - F -  (a_1 - \ep ) F ]   
\leq   c^2  \ep \min_{\partial B(0,r_j)}  F
\end{align}  
Using once again the maximum principle for $ \mathcal{A}$-harmonic 
functions in $ B ( 0, r_{j} )  \sem B ( 0, r_{j_1} ) $ and letting $j  \to \infty $  
we conclude  for some $ c'  \geq 0, $ depending only on the data and $ F_1 (e_1) $  that 
 \[  
F_1 -  F  -  a_1 F   \leq  c'  \ep  F   \quad \mbox{in}\, \,  \mathbb{R}^n \sem B (0, r_{j_1} ).  
 \]      
Dividing both sides by  $ F $ and letting $ \ep \to 0 $  we get    \eqref{1.46}.  
    
    Next  we put   $ F_2 =    (1 +a_1 ) F $  and  observe as in the display above \eqref{1.17a} that      
\begin{align}  
\label{1.47a}  
  a_2   =  \lim_{r \to 0 }  m(r, F_1 - F_2 )\quad \mbox{and} \quad a_3 = \lim_{r \to 0} M ( r, F_1 - F_2 )  \quad \mbox{both exist}.        
\end{align} 
   if    $ \ti a  \in \{ a_3, a_2 \}$  is finite, then  from  \eqref{1.43}-\eqref{1.45}  
   we see that  Harnack's inequality can be used in a now well-known way to  deduce     $ \ti a  = a_3 = a_2. $  
  Also,   the case  $ a_2 =  -  \infty,  a_3 = \infty $    
  cannot  occur in view of  \eqref{1.41}-\eqref{1.45}, our knowledge of $F,  
  F_1, $  and the same argument as in  \eqref{1.36}.  
   In order to  rule out other cases  choose     $  t_1  >  0 $  so that   $  F ( x )  \geq 0 $ in  $  \rn{n}  \sem      B ( 0, t_1) . $      Then  given  $  \ep  > 0, t  > t_1,  $   it  follows from 
   \eqref{1.27a}  and our  choice of  $ a_1 $      that      there   exists  $ s_0  =  s_0 ( \ep, t )  > > t$  so that 
     \[  m ( t,  F_1  -  F_2 )   -  \ep  F    \leq    F_1 - F_2     \leq   M ( t,  F_1 - F_2 )  +    \ep  F    
      \mbox{ on  $  \ar [ B ( 0, s ) -  B ( 0, t )  ] $  for   $ s \geq s_0$  } .   \] 
    Thus from the  maximum - minimum principles for  $ \mathcal{ A } $-harmonic functions,  this inequality also holds in  $   B ( 0, s ) -  B ( 0, t ) . $  Letting  $ \ep  \rar 0 $ 
   we conclude that   
   \beq  \label{1.47b} \mbox{ $ m (\cdot, F_1 - F_2  ) $ is non-decreasing, 
   and   $ M (\cdot, F_1 - F_2  ) $ is non-increasing on  $ (t_1,  \infty ) .$ } \eeq 
      From   the above discussion of  cases    we  
   see that if   $  a_2 =  -  \infty, $ then $ a_3 =  - \infty$.  However  from  \eqref{1.47b}  and the maximum principle for  $ \mathcal{A} $-harmonic functions it then follows  that 
    $ F_1 - F_2 \equiv $  constant, an  obvious contradiction.   If  $ a_2   = \infty = a_3 $  we can apply  the same argument to  get that  $ F_1 - F_2 $  has  an  absolute minimum  so  is  constant,  another  contradiction.    Also   if    $  a_2  =  a_3  \not  =  \pm \infty, $  then 
   \eqref{1.47b}  and  the above  maximum - minimum principles imply  that  
   \[   m ( \cdot,  F_1  -  F_2 )  \geq  a_2   \geq  M ( \cdot,  F_1  -  F_2 )  \mbox{ on } (0, \infty ) \]
   so    $ F_1 -  F_2 \equiv $  constant.   Finally  $ a_1  = 0 $  since if $   \he   \in C_0^{\infty} (\rn{n}) $  with  $  \he ( 0 )  =  1, $  then   from \eqref{1.2} $ (ii)$  for $ F, F_1, $  we have       
      \[      0 =  \int_{\rn{n}}\lan   \mathcal{A} (\nabla F_1 )  -   \mathcal{A} (\nabla F_2), \nabla \he  \ran  \,  dx   =  ( 1  + a_1)^{n-1}     -  1.  
    \]
      The proof of  Lemma \ref{lem1.4} is now complete.  
      \end{proof} 

\subsection{Existence and uniqueness for $p > n$} 
\begin{definition}
\label{def4.2}
If  $ p > n \geq 2$   we say that $ F$  is a fundamental solution to 
$\nabla\cdot \mathcal{A}(\nabla F)=0$
in $\mathbb R^{n}$ with pole at $0$ if
\begin{align}
\label{1.53} 
\begin{split}
 (i)\, \,\,  & \mbox{ $F $ is  $  \mathcal{A}$-harmonic in $ \mathbb{R}^n  \sem \{0\}, $}  \\    
(ii)\, \, & F\in W_{{\rm loc}}^{1,p}(\mathbb R^{n}),\, \, \mbox{$F$ is continuous in $\mathbb R^{n}$},\, \,  F (0)=0,\, \,  F >0 \, \, \mbox{in}\, \,  \mathbb R^{n}\setminus\{0\},  \\
(iii)\, & \int \lan \mathcal{A} (\nabla  F(z)), \nabla \he ( z ) \ran\, dz = -\theta(0) \quad \mbox{whenever} \, \, \theta\in C_0^\infty(\mathbb R^{n}).
\end{split}
\end{align}
\end{definition}\
 As in  Lemma \ref{lem1.4} we prove 
\begin{lemma}
\label{lem1.5} 
Let $2 \leq n$ be an integer and let $p$ be fixed with $n<p<\infty$. Let $\xi=(p-n)/(p-1)$ and assume that $ \mathcal{A}\in M_p(\alpha) $ for some $\alpha\in (1,\infty)$. Then
there exists a  unique  fundamental solution $F$ to the equation  $\nabla\cdot  \mathcal{A}(\nabla F )=0$
in $\mathbb R^{n}$ with pole at $0$  in the sense of Definition \ref{def4.2}.  Furthermore, there exist
$\sigma=\sigma(p,n,\alpha)\in(0,1)$, $c  \geq 1$, depending only on the data, and  $\psi\in C^{1,\sigma}(\mathbb S^{n-1})$ such that
\begin{align}
\label{1.54}
\begin{split}
&(i) \hs{.2in}  F(z)=|z|^{\xi}\, \psi(z/|z|)  \quad \mbox{whenever}\, \,  z\in\mathbb R^{n}\setminus\{0\},  \\
&(ii) \hs{.1in}  c^{-1}   F(z) =  \lan z,  \nabla F ( z ) \ran  \leq |z| | \nabla F ( z ) |  \leq c F (z)  \quad \mbox{whenever}\, \,  z \in \rn{n} \sem \{0\}.
\end{split}
\end{align}
\end{lemma}  
\begin{proof} 
A  proof of  Lemma  \ref{lem1.5} is given in  Lemmas 5.1  and 5.2  of \cite{LN4} and  somewhat  simpler  than the case $p = n $  so we only sketch the details. 
   Fix  $ R,   2   <  R  <  \infty$,  and let $ v = v ( \cdot, R )  $  be the  $  \mathcal{A}$-harmonic function in  $ B (0, R) \sem \{0\} $  with  continuous boundary value $0$  at  $0$ and  $R^{\xi}  $   on  $  \ar B (0, R )$. Existence of $ v \in W^{1,p}  ( B ( 0, R ) ), $  follows from the fact that a point set has positive $\mathcal{A}$-capacity when $ p > n $  (see \cite[Chapter 2]{HKM}). Extend $ v $  to  a continuous function in $ B ( 0, R ) $  by setting $ v ( 0 ) = 0$.  
       Let  $ \nu $  denote the  positive Borel measure  with support at 
       $ \{0\} $  associated with $ v $.     As in  Lemma \ref{lem1.2}  we prove        
     \begin{lemma}  
     Let $\nu$, $v$, $R$, and $\xi$ be as above. Then
     \label{lem1.6}  
     \begin{align}  
\label{1.55} 
\begin{split}
 (a)  &\quad   \nu ( \bar B (0, R) ) \approx  1, \\  
 (b) &\quad     c^{-1} v(x)  \leq     \lan \nabla v (x), x  \ran  \leq   |x|  | \nabla v |   \leq c v(x) \quad \mbox{whenever} \,\,  x  \in    B(0, R) \sem  \{0\},  \\  
(c) &\quad    v (x)   \approx   |x|^{\xi}\quad \mbox{whenever}\, \, x \in B(0,R) \sem B (0, 1 )
\end{split}
\end{align}    
where  ratio  constants depend  only on $p$, $ n $, and $\alpha$.       
\end{lemma}  
 \begin{proof}   
   To  prove  \eqref{1.55}    let   $  w (x)  =  | x |^{\xi} $,           when $x  \in  B(0, R), $ and 
note that  $ w $  is  $p$-harmonic in  $ B(0, R ) \sem \{0\}$  with  $ v \equiv  w, $ 
   continuously  on  $  \ar B ( 0, R ) \cup  \{0\}.  $   Then  as   in    \eqref{1.4}  we also have
    \begin{align}
\label{1.56}  
 \nu (\bar B (0, 1 )) \, R^{\xi}=  \int_{ B (0, R ) \sem \bar{B}(0,1)}  \lan \mathcal{A} ( \nabla v ), \nabla v \ran dx 
\approx  \int_{ B (0, R ) }  |  \nabla w |^{\xi}   dx  \approx    R^{\xi}.  
\end{align}  
Thus  $ (a) $ in  \eqref{1.55}  is valid.     
To  prove  $(b) $ in \eqref{1.55},   we  show  (compare with  \eqref{1.7})  that if $ 1  < \la \leq 101/100 $   and   $ x \in B ( 0,   R/\la) 
\sem \{0\}, $  then  for some $ c \geq 1, $ depending only on the data    
\begin{align}  
\label{1.57}   
 c  v (x)   \leq    \frac{ v ( \la x )  - v (x)}{\la - 1}    
\end{align}  
Indeed this estimate is clearly valid  at $ x = 0$.   To show this inequality also holds  
when  $ x \in  \ar B (0, R/\la ) $   we first observe  from   Lemma  \ref{lemma2.4}   
applied to  $ R^{\xi} - v $  and the same argument  as  in  \eqref{1.6}    that   
\begin{align} 
\label{1.58}     
R^{\xi}  -   v (x)   \approx  R^{\xi} \quad \mbox{when}\, \, |x| = R/8.
\end{align}    
Let  $ \ph_2  $ be  as in     \eqref{1.8}. Then $ \ph_2   \equiv 1$  on  $ \ar B (0, 1),  $   
 $  \equiv 0  $ on  $  \ar B (0, 8 ), $   and  $  \ph  ( 8 x/R ) $   is  a     
     subsolution at points in $ B (0, R ) \sem B ( 0, R/8)  $  
     where  $  \nabla v  \neq  0 $,     to the  second order  non-divergence form PDE  
    for $v$  corresponding to  $ \mathcal{A}$-harmonicity when $ N $ is sufficiently   
     large (depending only on the data). Using  these  facts  and  \eqref{1.58}  it follows  
    as in   \eqref{1.10} that 
\[  
R^{\xi}  - v (x)   \geq  c^{-1} R^{\xi}  \ph_2 (8x/R)  \geq  c^{-2} R^{\xi - 1} ( R -  |x| ) 
\]
whenever $x \in B (0, R ) \sem  B (0, R/8)$.  
This inequality   is easily seen to imply   \eqref{1.57} whenever    
  $  x   \in   \ar  B ( 0,  R/\la ),  1  < \la  \leq  101/100$.   
Letting $  \la   \to 1 $  we get the left-hand inequality in  \eqref{1.57}.   
The far right inequality follows from   \eqref{eqn2.2} for  $  |x|  \leq  R/2 $   
and from  a  barrier type argument using 
$ R^{\xi} [ 1 - \ph_1 ( 8 x/R)]$,  $\ph_1  $  as in   \eqref{1.8},   
when  $ R/2 \leq  |x |  < R$.  
Thus   \eqref{1.57}   is  true.  Letting $   \la  \to 1 $ in this inequality  
 we obtain the left-hand  inequality in   \eqref{1.55} $(b).$  
 The right-hand inequality in this display  follows from 
\eqref{eqn2.2} $(ii)$ for $ x \in B ( 0, R/2  ) \sem B (0, 2) $  
and from the above barrier estimates if $ 1 < |x|  < 2 $ or $ R/2 < |x| < R$.   Finally, \eqref{1.55}   $(c)$  follows from uniform  $ (|x|/2, p ) $  fatness of  $ \{0\} $ and   
\eqref{eqn2.6} $(ii)$. 
 \end{proof}        
 
 We now return  to the proof of lemma \ref{lem1.5}.  To avoid confusion we again temporarily write  $ v ( \cdot, R ), \nu ( \cdot, R ),  $  
for  $ v $ in   Lemma \ref{lem1.6}. Using   Ascoli's theorem and Lemmas \ref{lemma2.1}-\ref{lemma2.3} 
for $\mathcal{A}$-harmonic functions  we see as $ m \to \infty $   that  a
 sub-sequence  of  $  \{ v ( \cdot, m  )\}_{m=8}^\infty  $  
 converges uniformly on compact subsets of  
  $ \rn{n} $ to a  $  W^{1,p}_{{\rm loc}} ( \rn{n}) $ function  $  \hat V \geq 0$ which is    $ \mathcal{A}$-harmonic 
  in   $ \rn{n} \sem  \{0\} $  with   $   \hat V (0) = 0$.  
   Let  $ \hat \Theta   $ be the  measure  with support at $ \{0\} $ corresponding to $ \hat V.$   We put  
\begin{align*}
 F(\cdot) := \frac{\hat{V}(\cdot)}{[\hat \Theta (0) )]^{1/(p-1)}}.
\end{align*}  
   Then  $ (i)$ and $(ii)$   of \eqref{1.53} are true for  $F.$  $(iii) $ 
   of \eqref{1.53}    follows from  \eqref{eqn2.6} $(i)$  with $ \ti u = F $ 
 since the point     measure corresponding to $ F $ has total mass $1$ at $\{0\}$.  
       Moreover \eqref{1.54} $(ii) $ is true as we see from Lemma \ref{lem1.6}.

       To prove uniqueness suppose that $ F_1 $  also  satisfies  \eqref{1.53} 
      $ (i)-(iii)$.  Then    from  \eqref{1.53} $(iii)$   we see that  the measure associated to $ F_1 $   
      has a point mass of total mass 1  at $\{0\}$  and since  $  \{0\} $ is 
      uniformly  $ (R, p )$-fat  for  any  $ R > 0, $  it follows from \eqref{eqn2.6} $(ii)$ of Lemma \ref{lemma2.4}  (as in \eqref{1.37}-\eqref{1.40})  that   
\begin{align} 
\label{1.59}   
M ( t,  F_1 ) -  m ( t, F_1)   \approx   t^{\xi } \quad  \mbox{whenever}\,\,  t > 0. 
\end{align}      
From \eqref{1.59}  and  \eqref{eqn2.2} $(\hat a)$  we have 
\begin{align} 
\label{1.60}    M  ( t,  | \nabla F_1| )   \leq   c   t^{(1-n)/(p-1)}  \quad \mbox{whenever}\,\,  t >  0. 
\end{align} 
Again constants depend only on the data.   Using   \eqref{1.59}, \eqref{1.60}, 
and  \eqref{1.54} $(ii) $  for  $ F $  we  can now  use the  same argument as  after  \eqref{1.43} 
to deduce  first that if  $  F_2 = \zeta F$  then  $ F_1 - F_2  $ is a  solution to  
an elliptic  PDE as in \eqref{1.43}, \eqref{1.44}, whose ellipticity is given by   \eqref{1.45}  
with  $n$ replaced by $p. $   Using these facts it then follows as  in  \eqref{1.46} 
that      
  $ \lim_{ x \to \infty} (F_1 -  F)/F  = a_1,$   where $ -1 < a_1 < \infty.$    
  Next    given  $  \ep > 0  $ small  we   find that  $ |F_1 - (1+a_1 ) F| \leq \ep F $
   on  $  \ar B (0, s ) $ for $ s > 0 $ large enough so  
   from  the maximum principle for $ \mathcal{A}$-harmonic functions  
   and $ F(0) = 0 = F_1 (0), $  this inequality also holds in $ B (0, s ). $  Letting $ \ep \to 0 $ 
   we get $  F_1 = (1+a_1)  F. $  Now $a_1 = 0 $ is  proved in the same way as 
   earlier  (see the display before  \eqref{1.53}). Thus $ F $ is the 
   unique function satisfying 
   \eqref{1.53}$(i)-(iii)$.  Finally, since $ \mathcal{A}$-harmonic
    functions are dilation invariant, it is  easily checked that 
  given  $ t >0, $ the function     $ x   \to  t^{-\xi} F (t x)$, whenever $x \in \rn{n}$,   
  also satisfies \eqref{1.53} $(i)-(iii)$  so by uniqueness  equals  $ F. $    
  The proof of  Lemma  \ref{lem1.5} is now  complete. 
  \end{proof}

  \setcounter{equation}{0} 
 \setcounter{theorem}{0}
   \section{$\mathcal{A}$-harmonic Green's function} 
   \label{section5} 
 Fix $ p \geq n, $ and let  $ F $  be the fundamental solution 
  constructed in Lemma \ref{lem1.4} when $p=n$ and in Lemma \ref{lem1.5} 
  when $n<p<\infty$.  If $ p = n $ we assume that $ F (e_1) = 1. $  
  In   this section we  use  $ F $ to define, show the existence of,  
  develop some properties of, and prove uniqueness  for   
  the $\mathcal{A}$-harmonic Green's function for  $ \rn{n} \sem E $ with pole at $ \infty $ 
  when $ p  \geq n$ whenever $ E $ is    a  certain  compact convex set.      
  \begin{definition}[{\bf $\mathcal{A}$-harmonic  Green's function}]
  \label{def5.1} 
  Given a   compact, convex set $ E \subset \rn{n}$  
  we say that $G$ is the $ \mathcal{A}$-harmonic Green's function for  $ \rn{n} \sem E $ 
  with pole at $ \infty, $  if   
  $ G : \rn{n} \sem E \to ( 0, \infty ) $ has continuous 
  boundary value $0$ on $ \ar E$,    $G $ is  $ \mathcal{A}$-harmonic in  $ \rn{n} \sem E$,  and     $ G(x) =  F(x)  +  k(x)  $   where $ k(x) $ is  
  a bounded function  in a neighborhood of $ \infty$. 
  \end{definition}       

  The  statements and proof of  uniqueness as well as some other properties of $ G$ and  $k$  
  are slightly different when $ p  > n $ and  $ p = n$,  so  we consider  each range of $ p $ separately.   We start with the case when $p=n$.  
    \begin{lemma}
    \label{lem5.2}  
Let $p=n$.   Given a  compact  convex set  $E$ consisting of  at least two points,  
there exists a  unique  $ \mathcal{A}$-harmonic Green's function $G$ for $\mathbb{R}^{n}\setminus E$  
with  pole at $ \infty$ in the sense of Definition \ref{def5.1}.   Moreover, 
\begin{align} 
\label{5.2}  
\begin{split} 
&(a) \, | \nabla  G (x) |  \leq  c \, | x - \hat{x} |^{-1} \,\,  \mbox{whenever}\, \,   x\in \mathbb{R}^{n}\setminus E \, \, \mbox{and} \, \, \hat x \in E \\
&\quad \, \, \mbox{with} \, \, |x - \hat x | \geq 8\,  \mbox{diam}(E) \, \, \mbox{where $c$ depends only on the data}. \\
 &(b) \,  \lim_{x \to \infty} k ( x ) =  k ( \infty )  \, \,  \mbox{exists finitely.}  \\
 & (c) \, \mbox{There exist $\beta \in (0, 1)$,  depending only on the data, and}  \, \, \hat r_0 =  \hat r_0 (E)\geq 100 \\ 
 & \quad \, \, \mbox{such that}\,\,   |k ( x) -  k ( \infty ) |    \leq  \hat r_0  |x|^{- \beta} \, \,   \mbox{whenever} \,\, |x |  \geq   \hat r_0. \\ 
& (d) \,  \mbox{If $ \mu $  denotes   the positive Borel measure  associated to $ G $ then $ \mu ( E ) = 1.$ } 
\end{split}
 \end{align}
\end{lemma}    
 \begin{proof}   
 To prove existence in  Lemma   \ref{lem5.2}   for   $  G  $, 
 let  $ E $ be  a  convex set consisting of  at least two points  
 and  let $\hat x$ be as in  Lemma \ref{lem5.2}.   
 We first   suppose  $  \hat r:=\mbox{diam}(E) =  1$ and $\hat  x =  0$.   
 Then   after proving  Lemma   \ref{lem5.2} for these values,  
        we shall  use  translation and dilation invariance of  
        $  \mathcal{A}$-harmonic functions to get this lemma 
        for general $   \hat x$ and $\hat r.  $     Let           $   R > 4  $   
        and    $ v  = v ( \cdot, R ) $ be the $ \mathcal{A}$-harmonic 
        function in          $  B ( 0,  R  ) \sem E $ with  $ v (x)  \equiv  
       \ga \log R  $ on        $ \ar  B ( 0, R ) $ and $ v \equiv 0 $ 
       on $\ar  E$ in the sense of  continuous boundary values. 
       Here $ \ga $  is as in \eqref{1.27a}.  Existence of $ v $ and in fact  
       local H{\"o}lder continuity of  $ v $ in  $ B (0, R ) $   follows  from  
        the note after Definition  \ref{def3.2}. Extend  $ v $  to  a  
        continuous function in  $ B (0, R ) $ by putting  $ v  \equiv 0 $  on $ E.$  
        Let  $  \nu $  be the measure corresponding to $ v $ with support contained in $E$ as in  
        Lemma \ref{lemma2.4}  and let $  \zeta(x)$ be the $n$-capacitary 
        function for  $  E $  relative to   $ B ( 0, R ). $  Since the  $n$-capacity of  a  
        line segment of  length    $1$   with center at the origin in $ B (0, R) $  is  $  \approx   
        (\log R )^{1-n}, $  we see from Lemma  \ref{lemma2.4},   
        and the same  argument as in   \eqref{1.4}   with   $ w (x)  =  (\log R)
         (1  -  \ze(x) )$, as well as \eqref{1.5} and \eqref{1.11}   that  
\begin{align} 
\label{5.3}  
 1  \approx  \nu ( \bar{B} (0, 1  ))^{1/(n-1)}   \approx   M  (2 t,  v )  - m (t, v  )  \approx  M ( 2, v )
 \end{align}  
 whenever $ 2 \leq t  \leq R/2. $    Using   \eqref{5.3}  and  arguing as in  \eqref{1.12}  we  get    
\begin{align} 
\label{5.4}    
M  ( t,  | \nabla v |   )  \leq  c t^{-1}  \quad \mbox{whenever}\quad     4  \leq t  \leq       R/ 4. 
\end{align}   
Moreover, from  \eqref{5.3},   Harnack's inequality for $v,$  the boundary 
maximum principle for  $  \mathcal{A}$-harmonic functions,   and Lemma \ref{lem1.4} we have  
\begin{align} 
\label{5.5} 
|  v ( x  )  -   F ( x ) |  \leq  c  \quad \mbox{whenever}\quad  2  \leq  |x| \leq  R,
\end{align} 
where constants  in \eqref{5.3}-\eqref{5.5} depend only on the data. 
 From   \eqref{5.3}-\eqref{5.5}, Lemmas  \ref{lemma2.1}-\ref{lemma2.3}, and 
      Ascoli's  theorem  we  deduce that  a  subsequence,  say  $ \{ v ( \cdot, R_l) \}, $  
      converges uniformly on compact subsets of  $  \rn{n} $ as $ R_l  \to \infty $ 
       to  a   H{\"o}lder continuous function, $G,$   that is  $ \mathcal{A}$-harmonic
        in  $ \rn{n} \sem   E  $   with  $  G  \equiv  0 $ on  $ E. $   
        Moreover,  $  \{ \nabla v ( \cdot, R_l) \} $ converges uniformly 
        on compact subsets of  $  \rn{n} \sem E  $  to  $  \nabla G, $ 
    and   $  \{ \nu ( \cdot, R_l  ) \}$  converges   weakly as measures to  $ \mu,  $  
    the positive Borel measure corresponding to $ G. $  
     From this convergence, \eqref{5.3} and \eqref{5.4},  we  
     deduce that \eqref{5.2} $(a)$ is true   when $ \hat x  = 0$ and $\hat r = 1.$
     Also \eqref{5.5} holds with $ v $  replaced by $ G,$ so $ G $  as in  Definition 
     \ref{def5.1} exists.

    We next prove uniqueness. To this end,  suppose  $ G_1 $ is another  $ \mathcal{A}$-harmonic   
      function in  $  \rn{n} \sem  E $ with continuous boundary value $0$ on $ \ar E $ and 
  $ G_1 - F$ is bounded in a  neighborhood of  $  \infty. $   Then  from the  
  boundary  maximum principle we see that   $ | G -  G_1 |  \leq  a < \infty   $   on  
  $ \rn{n} \sem E$.   From this inequality, the fact that  $ | G - F |  \leq c, $ 
  and our knowledge of $ F, $ we deduce that for given $ \ep > 0 $ 
  there is  an $ s > 0  $ large such that   $ | G - G_1 |  \leq  \ep G $ 
  on  $ \ar B (0, s). $ Using the boundary maximum principle once 
  again for $ \mathcal{A}$-harmonic functions we conclude that this 
  inequality also holds in  $ B (0, s )  \sem  E. $  Letting  $ \ep \to 0 $ we obtain 
  $ G  \equiv G_1. $  Thus $ G $ is unique.   
  
It remains to prove $ (b),  (c)$ and $(d)$ of  Lemma   \ref{lem5.2}.   To prove these inequalities,
 we note from   \eqref{5.4} and  Lemma \ref{lem1.4} that  $ k = G- F $  satisfies
   the     elliptic  PDE in  \eqref{1.43}-\eqref{1.45}  with $ F_1 = G$,  $F_2 = F, $ 
   and $ \zeta = 1.$ Thus,
  $ k $ is  a solution to a uniformly elliptic PDE in divergence form in  
  $ B (0, 2r)  \sem  \bar B (0, r/2) $  when  $ r  \geq 16. $  Using this fact and 
  boundedness of  $ k $ in  $ \rn{n} \sem B ( 0, 4) $  we can 
  now essentially repeat the argument given for $h$ in the proof 
  of Lemma \ref{lem1.3}. Thus, we first obtain as in 
  \eqref{1.18}  that  
\begin{align}
\label{5.6} 
\lim_{x \to \infty} k ( x ) = k ( \infty) < \infty,
\end{align} which proves   $(b)$ of Lemma  \ref{lem5.2}. 
  Then, using \eqref{5.6}, we argue as in \eqref{1.19}-\eqref{1.25a} to  
  show that if    $ M ( t, k )$ and $m (t, k ) $ are both either ultimately non-decreasing or
    both ultimately non-increasing then   $ k  \equiv  $   a  constant in a neighborhood
     of $ \infty.$  Thus  in  these cases,   $(c)$  of  Lemma   \ref{lem5.2} is trivially true.  Otherwise,  there exists  $ R_0 $ with $ M ( \cdot, k ) $ strictly decreasing  
  and $   m ( \cdot, k ) $  strictly increasing on $ (R_0, \infty). $  Using this fact and  
  Harnack's inequality  we find  as in \eqref{1.26} that      $ (c) $ of   Lemma 
  \ref{5.2} is true in  this case  (when $  \hat x = 0$ and $\hat r = 1$)  with  $ \hat r_0  = c R_0. $     
  
  To prove $ (d) $  assuming  $(c) $ when $ \hat x = 0$ and $\hat r = 1$,    let  $  R \geq 
  10 \hat r_0   \geq 1000 $ and  $ 0 \leq  \he  \in  C_0^\infty ( B (0, 2R ) )  $ with $ \he  \equiv 1 $ on 
  $ B ( 0, R ) $ and $ |\nabla \he|  \leq c/R. $  Using  $ \he $ as a  test function in  
  \eqref{1.2} for  $ G $  we see that  
\begin{align}  
\label{5.7}  
\mu ( E ) = -   \int_{ B ( 0, 2R) \sem B (0, R )} \lan  \mathcal{A} (\nabla G),  \nabla \he \ran dx. 
    \end{align} 
    Now from $ C^1 $   regularity of   $ \mathcal{A}$,  the  structural 
    assumptions on $\mathcal{A}$ in Definition \ref{defn1.1},    Lemma \ref{lem1.4},  \eqref{5.2},  
        we have
\begin{align} 
\label{5.8}  
|\mathcal{A} ( \nabla F )  -  \mathcal{A}( \nabla G ) | |\nabla \he | \leq   c          R^{1 - n}  | \nabla k | \quad \mbox{whenever} \quad R \leq |x| \leq 2R
 \end{align}
 where $ c  > 0 $  is a  constant that  depends only on the data. 
          Also since $ k $ is a solution to a uniformly elliptic 
          PDE in  $  B (0, 4R ) \sem B ( 0, R/2 ) $ it follows from a   
          Caccioppoli type inequality for $ k $ and   $ (c) $ of \eqref{5.2} that  
\begin{align} 
\label{5.9}   
\int_{B( 0, 2R) \sem B(0, R)}  | \nabla k | dx  \leq c R^{n/2}  
        \left( \int_{B( 0, 2R) \sem B(0, R)}  | \nabla k |^2  dx \right)^{1/2}  
           \leq          c^2  \hat r_0  R^{{n - 1 - \be}} 
\end{align}  
From  \eqref{5.7}-\eqref{5.9} and  \eqref{1.2} $(ii)$,   we obtain 
\begin{align} 
 \label{5.10} 
 \begin{split} 
 | \mu( E ) - 1 | &=   \left|   \int_{ B ( 0, 2R) \sem B (0, R ) }     \lan  \mathcal{A} (\nabla G)  -  \mathcal{A} (\nabla F),  \nabla \he \ran
     dx  \right|   \\
     &\leq \,  c R^{1-n} \int_{B( 0, 2R) \sem B(0, R)}  | \nabla k | dx   \\    
     &\leq   c^3 \hat r_0  R^{-\be}  \to 0  \quad \mbox{as} \quad R  \to \infty.  
        \end{split}
        \end{align}   
From \eqref{5.10} we conclude  $(d)$ of \eqref{5.2} and so also Lemma \ref{lem5.2}
     when   $ \hat x = 0$ and $\hat r = 1$.    
         Now suppose  that  $  \hat x \in E $ and   $ \mbox{diam}(E)= \hat r. $   Let  
         $ \hat E =    (1/  \hat r) ( E  - \hat x).$  Then $ 0 \in \hat E $  and 
         $ \mbox{diam}(\hat E)=1$, so from our previous work,  we see that 
         $ \hat G =  F + \hat k $  where $ \hat G $ is the $\mathcal{A}$-harmonic Green's function for  $ \rn{n} \sem \hat E $  with pole at 
         $  \infty.  $     Let  
         \[   
         G ( x)  =   \hat G \left(  \frac{x  - \hat x}{\hat r}  \right) \quad\mbox{whenever} \quad  x \in \rn{n} \sem E.    
         \] 
                 Using  Lemma \ref{lem5.2} for  $ \hat E $  and translation and dilation invariance of  
        $ \mathcal{A}$-harmonic functions we see that $G$ is  $ \mathcal{A}$-harmonic in $ \rn{n}  \sem E $  with  $ G  \equiv 0 $ on $ \ar E. $   Also from 
        Lemma \ref{lem1.4}  and  Lemma \ref{lem5.2}  we can write for $ x \in \rn{n} \sem E, $  that  
\begin{align} 
\label{5.10a} 
G ( x )  = \ga  \log \left( \frac{| x - \hat x |}{\hat r}\right)  + b \left(\frac{x - \hat x}{|x-\hat{x}|}\right)         +   \hat k \left( \frac{x - \hat x}{\hat r} \right)    =  F ( x )   + k ( x ) 
\end{align}
where  
\[ 
k ( x )   =  - \gamma\log \hat r +  \ga \log \left(  \frac{| x - \hat x |}{ | x |} \right) +        b \left( \frac{x - \hat x}{| x - \hat x |} \right) -   b \left( \frac{x}{|x|}\right)    +  \hat k \left( \frac{x - \hat x}{\hat r}\right).
\]   

Using  Lemma  \ref{5.2}   for $ \hat k $  and smoothness  of  $ b $ we see first from  \eqref{5.10a} that 
\begin{align}
\label{5.11}   
k (  \infty) =  \lim_{x \to \infty}   k ( x )  = \hat k ( \infty ) -  \gamma\log(\hat r) 
\end{align}  
 and second that there exists   $ \hat r_0 = \hat r_0 ( E ) $   with  
   \[     
   | k ( x )  -  k ( \infty ) |  \leq  \hat r_0 \, |x|^{-\be}  \quad \mbox{whenever}\quad | x | \geq  \hat r_0.  
   \] 
    Also if  $  \hat \mu$ and $\mu$ denote the measures corresponding to $ \hat G$ and $G $ 
    respectively, then  $ \mu ( E )   =  \mu ( \hat E ) = 1, $ 
    as follows easily from  changing variables in the integral 
    identity  for  $ \hat \mu. $    The proof of Lemma  \ref{lem5.2}  is now complete.        
 \end{proof} 
           
Next we state  
      \begin{lemma} 
      \label{lem5.3}  
Let $n<p<\infty$ be fixed and $\xi = (p-n)/(p-1)$.      Given a  nonempty convex compact set  $E$, there exists a  
      unique  $ \mathcal{A}$-harmonic  Green's function $G$ for $\mathbb{R}^{n}\setminus E$  
      with  pole at $ \infty$ in the sense of Definition \ref{def5.1}. Moreover  
\begin{align}
\label{5.12} 
\begin{split}
& (a) \,   | \nabla  G (x) |  \leq  c | x - \hat{x} |^{\xi -1}    \approx     \frac{G ( x )}{ |x - \hat x |}     \, \,  \mbox{whenever}\, \,  x\in \mathbb{R}^{n}\setminus E \, \, \mbox{and}\,\, \hat x\in E\\
& \quad \, \, \mbox{with}\, \,  | x - \hat x | \geq 8 \mbox{  diam }(E) \, \, \mbox{where $c$ depends only on the data}. \\
&(b) \,   \lim_{x \to \infty} k ( x ) =  k ( \infty )  \, \,  \mbox{exists finitely and}\, \, k\leq 0   \, \, \mbox{in}\, \,   \rn{n} \sem E.  \\
&(c) \,  \mbox{There exist}\, \,  \beta \in (0, 1), \, \, \mbox{depending only on the data, and}
 \\  & \quad \,\, \hat r_0  =  \hat r_0 ( E)  \geq 100
\, \,   \mbox{such that}\, \,                            |k ( x) -  k ( \infty ) |  \leq   \hat r_0    |x|^{ - \beta} \, \,  \mbox{whenever}\, \, |x |  \geq \hat r_0. \\
& (d)\, \mbox{If $ \mu $  denotes   the positive Borel measure  associated to $ G $ then $ \mu ( E ) = 1.$ } 
\end{split}
\end{align}
      \end{lemma}     
 \begin{proof}  
 Once again we first assume that  $ \hat x = 0 $ and  $ \hat r :=\mbox{diam}(E)= 1$.   For fixed $ p > n, $   
 let  $ v = v ( \cdot, R ) $ be the $ \mathcal{A}$-harmonic function in $ B ( 0, R )   \sem E $   
 with boundary values  $ v \equiv 0 $  on $ E $ and  $ v \equiv  R^{\xi}$   on $ \ar B  ( 0, R ) $ 
  in the continuous sense.  Extend  
$ v $  to  a  H{\"o}lder continuous function in $ B ( 0, R) $ by setting $ v \equiv 0 $ 
on $ E. $            Let  $ \nu $ be the positive Borel measure associated to  $ v$ with support contained in $E$ as in Lemma \ref{lemma2.4}.  
Using uniform $(R,p)$-fatness of $ E $  and arguing as in \eqref{1.56},  \eqref{1.59},    
we see that  
\begin{align} 
\label{5.13}   
1  \approx  \nu ( B (0, 2  ))^{1/(p-1)}   \approx  t^{- \xi  } [ M  (2 t,  v )  - m (t, v  ) ] \approx  M ( 2, v ) 
\end{align}  
whenever $ 4 \leq t  \leq R/4. $    Using   \eqref{5.13}  and  arguing as in  \eqref{1.60}
 we  get    
\begin{align}
\label{5.14}    
M  ( t,  | \nabla v |   )  \leq  c t^{(1-n)/(p-1) } \quad \mbox{whenever}\quad  8  \leq t  \leq       R/ 8.   
\end{align}   
where constants in \eqref{5.13}-\eqref{5.14} depend only on the data. 
      From   \eqref{5.13}-\eqref{5.14}, Lemmas  \ref{lemma2.1}-\ref{lemma2.3}, and 
      Ascoli's  theorem  we  deduce that  a  subsequence,  say  $ \{ v ( \cdot, R_l) \}, $ 
       converges uniformly on compact subsets of  $  \rn{n} $ as $ R_l  \to \infty $  
       to  a   H{\"o}lder continuous function, $\bar G$   that is  $ \mathcal{A}$-harmonic
        in  $ \rn{n} \sem   E  $   with  $  \bar G  \equiv  0 $ on  $ E. $   Moreover,
          $  \{ \nabla v ( \cdot, R_l) \} $ converges uniformly on compact subsets 
          of  $  \rn{n} \sem E  $  to  $  \nabla \bar G$ 
    and   $  \{ \nu ( \cdot, R_l  ) \}$  converges   weakly as measures to  $ \bar \mu,  $  
    the positive Borel measure corresponding to $ \bar G. $  Let 
    \[  
    G(\cdot)  =  \frac{ \bar G(\cdot) }{ \bar \mu ( B (0, 2 ))^{1/(p-1)}  } \quad \mbox{and} \quad  \mu(\cdot) =  \frac{ \bar \mu(\cdot)}{ \bar \mu ( B ( 0, 2 )) }.  
    \]   
    Then  $ \mu $ is the measure  corresponding to $ G $ and      $  \mu ( B (0, 2))  = 1. $  
    
    Moreover from \eqref{5.13} and \eqref{5.14} we  deduce that 
    \begin{align} 
\label{5.15}    
|\nabla G ( x ) | \leq  c |x|^{(1-n)/(p-1)}  \approx    G ( x )/|x|  \approx  F ( x)/|x| \quad \mbox{whenever} \quad   |x|  \geq  8.   
\end{align}   
Using  these facts  and  the maximum principle for  $ \mathcal{A} $-harmonic functions we see that  $ F - G $ is a  non-negative  weak solution in $ \rn{n}  \sem  E $  to the elliptic equation in \eqref{1.43}-\eqref{1.44}  with $ F = F_1$, $G = F_2, $ $ \zeta= 1$, and $n$  replaced by $p$.  Also \eqref{1.45} is valid with $ n $ replaced by $p. $  Using this version of \eqref{1.43}-\eqref{1.45}  we deduce  as in \eqref{1.46}  that       
\begin{align} 
\label{5.16}   
\lim_{x \to \infty}   \frac{( F  - G ) (x)}{G(x)} = a_1  > - 1 . 
\end{align}  
     Then from the maximum principle for $ \mathcal{A}$-harmonic functions we see for given small $ \ep > 0  $ 

 \[    - \ep F   -  M ( 2, F )  \leq   (1+a_1) G - F   \leq \ep F    \quad  \mbox{on}\, \,  \ar (B ( 0, s )\sem E )  
 \]  
 provided   $ s $ is large enough,  so this inequality also holds in   $ B ( 0, s ) \sem E. $ Letting $ \ep \to 0 $ we get 
\begin{align} 
\label{5.17} 
0  \leq    F - (1 + a_1 ) G  \leq   M ( 2, F )  \quad  \mbox{in}\, \, \rn{n}\setminus E.  
\end{align}  
Using  \eqref{5.17}  and the new version of  \eqref{1.43}-\eqref{1.45} we can now    
once again  essentially  repeat  the argument given in the proof of  Lemma \ref{lem1.3} with $ h  = F - ( 1 + a_1)  G $.  
Thus if      $ k  = ( 1 + a_1 )  G  -  F $ we first obtain  \eqref{5.6}.     Next  
     using  \eqref{5.6}  we   argue as in  
     \eqref{1.18}-\eqref{1.25}  to eventually conclude that $(c)$ of  Lemma  \ref{lem5.3} is true for this $ k$.  
      It remains  to  show  that  $ a_1 = 0 $ in order to  prove  $(b), (c), $ of Lemma \ref{lem5.3}.  
      To  do this  let   $ R > 16,   0 \leq  \he  \in  C_0^\infty ( B (0, 2R ) )  $ with $ \he  \equiv 1 $ on 
  $ B ( 0, R ) $ and $ |\nabla \he|  \leq c/R. $  Using  $ \he $ as a  test function in  
  \eqref{1.2} for  $ G $  we see that   
\begin{align}  
\label{5.18}  ( 1 + a_1 )^{p-1}     \mu ( E ) =   ( 1 + a_1 )^{p-1}   =   -   \int_{ B ( 0, 2R) \sem B (0, R ) } 
    \lan  \mathcal{A} ((1+a_1)\nabla G),  \nabla \he \ran dx.
    \end{align}
    Then    as  in  \eqref{5.8}     we get from \eqref{5.15},  \eqref{1.54} $(ii)$, and 
   the  structure assumptions on  $ \mathcal{A}$  for $ k $ as defined above,    
   \begin{align} 
\label{5.19}  
|\mathcal{A} ( \nabla F )  -  \mathcal{A}( (1+ a_1 ) \nabla G ) | |\nabla \he | \leq c  ( 1 + |a_1|)^{p - 2}   R^{ \frac{n(2 - p) - 1 }{p-1} } \, \,  | \nabla k | \quad \mbox{whenever}\, \,  R \leq |x| \leq 2R.
\end{align}  
Also \eqref{5.9}   remains valid for our  $k.$     
           Using these  inequalities we obtain as in  \eqref{5.10} 
\begin{align} 
\label{5.20} 
\begin{split}
| (1+a_1)^{p-1} -1|&= \left|   \int_{ B ( 0, 2R) \sem B (0, R ) }     \lan  \mathcal{A} ((1 + a_1 ) \nabla G)  -  \mathcal{A} (\nabla F),  \nabla \he \ran
     dx  \right|  \\
     &\leq \,  c  ( 1 + |a_1|)^{p-2}   R^{ \frac{n   ( 2 - p) - 1 }{p-1}}  \int_{B( 0, 2R) \sem B(0, R)}  | \nabla k | dx     \\    
     & \leq          c^2  ( 1 + |a_1|)^{p-2} \, \hat r_0 R^{  \frac{ ( n - 1) - (p-1) ( 1 + \be )}{ p-1} } \to 0  \quad  \mbox{as}\, \,  R  \to \infty.   
\end{split}
\end{align}  
From \eqref{5.20} we conclude  $ a_1 = 0 .$    Thus $(b)$ and $(c)$ of  Lemma   \ref{lem5.3} are  
        true when $ \hat r = 1$ and $\hat x = 0 $.  
        
        To prove uniqueness in this case suppose  $ G_1 $  is also a $\mathcal{A}$-harmonic Green's function 
        for $\mathbb{R}^{n}\setminus E$ with pole at $  \infty. $ 
        Then from the boundary maximum principle 
        for $ \mathcal{A}$-harmonic functions  we 
        have  $ |  G - G_1 | $  bounded in  $ \rn{n} \sem E $ 
        so given $ \ep > 0 $ it follows from      Lemma  \ref{lem5.3} $(b)$ 
        that if  $ s > 0 $ is large enough, then  
        \[  
        | G - G_1 |  \leq  \ep  G\quad   \mbox{on}\, \,  \ar ( B ( 0, s ) \sem E).
         \]
Once again from the maximum principle this inequality holds in  $ B (0, s ) \sem E. $ 
Letting $ \ep \to 0 $  we conclude $ G_1 = G. $               
                   
To remove the assumption that $  0 \in E $ and   $\mbox{diam}(E) = 1$, we suppose  that  $  \hat x \in E $ and   $\mbox{diam}(E) = \hat r$ for some $\hat x$ and $\hat r>0$.   Let  
         $ \hat E =    (1/  \hat r) ( E  - \hat x).$  Then $ 0 \in \hat E $  and 
         $ \mbox{diam}(\hat E)=1 $, so from our previous work,  we see that $ \hat G =  F + \hat k $  
         where $ \hat G $ is the $\mathcal{A}$-harmonic Green's function for  $ \rn{n} \sem \hat E $  with pole at 
         $  \infty  $ and $  p>n $ fixed.     Let  
\[   
G ( x)  =  (\hat r)^{\xi}  \hat G \left(  \frac{x  - \hat x}{\hat r}  \right) \quad \mbox{whenever}\, \, x \in \rn{n} \sem E.    
\] 
Using  Lemma \ref{lem5.2} for  $ \hat E $  and translation and dilation invariance of  
        $ \mathcal{A}$-harmonic functions we see that $ G $ is  $ \mathcal{A}$-harmonic in $ \rn{n}  \sem E $  with  $ G  \equiv 0 $ on $ \ar E. $   Also from 
        Lemma \ref{lem1.4}  and  Lemma \ref{lem5.2}  we can write for $ x \in \rn{n} \sem E, $  that  
\begin{align} 
\label{5.21}     
G ( x )  =    F ( x - \hat x )         +  \hat r^\xi \,  \hat k \left( \frac{x - \hat x}{\hat r}\right )    =  F ( x )   + k ( x )  
\end{align}
       where  
\[ 
k ( x )   =   F ( x - \hat x )  - F ( x )  + \hat r^{\xi}  \hat k \left( \frac{x - \hat x}{\hat r}\right).
\]   
From this display and Lemma \ref{lem5.3} in the  previous  special case  we conclude  first  that    
\begin{align}  
\label{5.22} 
\lim_{x \to \infty} k ( x )  = k ( \infty ) =  \hat r^{\xi} \, \hat k ( \infty ) 
\end{align}
so    $ (b) $ of Lemma \ref{lem5.3} is valid.   Second  we deduce the validity of $ (c) $  of  Lemma \ref{lem5.3}.  Finally   
        if  $  \hat \mu$ and $\mu$ denote the measures corresponding to $ \hat G$ and $G $ respectively then  $ 
        \mu ( E )   =  \mu ( \hat E ) = 1, $ as follows easily from  changing variables in the integral 
        identity involving  $ \hat \mu. $    The proof of Lemma  \ref{lem5.3}  is now complete.        
              \end{proof}                
We are now able to define $\mathcal{C}_{\mathcal{A}}( E )$  alluded to in the introduction.
\begin{definition} 
\label{defn5.4} 
If  $ E $ is a nonempty convex compact set when $ p > n $ or  
$ E $ contains at least two points when $ p = n $  we put 
\[ 
\mathcal{C}_{\mathcal{A}}( E ) :=
\left\{
\begin{array}{ll}
  e^{- k ( \infty )/\ga} &  \mbox{when}\, \, p = n, \\ 
 (- k ( \infty ))^{p-1} & \mbox{when}\, \,  p > n. 
\end{array}
\right.
\] 
Here $\gamma$ is the constant as in Lemma \ref{lem1.4}.
\end{definition}  
\begin{remark} 
\label{rmk5.5} 
From the definition of  $\mathcal{C}_{\mathcal{A}}( E )$  and
     translation and dilation invariance of  $\mathcal{A}$-harmonic functions  we note that if 
     $ y \in \rn{n}$, $r > 0$, and  $ E $ is  a  convex compact set  containing  at least two points with  
     $ \hat E  =  r ( E + y ), $  then  
 \begin{align} 
 \label{5.23}
 \begin{split}
\mathcal{C}_{\mathcal{A}} (\hat E )=
  \left\{
  \begin{array}{ll}
      r \mathcal{C}_{\mathcal{A}} (E)  & \mbox{when}\, \,  p =n,\\
      r^{p-n}  \mathcal{C}_{\mathcal{A}} (E) & \mbox{when} \, \, p > n.
\end{array}    
           \right.
\end{split}
\end{align}        
         \end{remark}    
 \setcounter{equation}{0} 
\setcounter{theorem}{0}
\section{Proof of Theorem \ref{theoremA}}
\label{section6}
Let  $  \mathcal{A}  \in M_p ( \al ) $  and  $ p  \geq n $ be fixed.  Let  $G$ be    
the  $ \mathcal{A}$-harmonic  Green's function  for the complement of a  nonempty compact convex set, $ E, $   
containing at least two points   when  $p = n $  (see Lemma  \ref{lem5.2}  for $ p = n $  and  Lemma \ref{lem5.3}  when $ p > n$). 
In this section, we first study the levels of $G$ and then we prove that $ \mathcal{C}_{\mathcal{A}} (\cdot)$ satisfies the 
Brunn-Minkowski inequality and  also   obtain the   conclusion of  Theorem \ref{theoremA}    when  equality 
occurs. We  begin with 
\begin{lemma}
\label{lem6.1}
Let $p$, $\mathcal{A}$, $E$, and $G$  be as above and  put  $ G  \equiv 0 $ in $E$. 
Then   $ \{ x :  G (x) < t\}$ is convex for every   $t  \in  (0, \infty)$.  
Moreover,  $  \nabla  G  \neq 0 $ in   $ \rn{n} \sem  E $   and if     $  B ( 0, r_0 ) \subset E$ for some $r_0>0$ and   $\mbox{diam}(E)  = 1$   then  
\begin{align} 
\label{6.0} 
\begin{split}
&(a) \hs{.1in}   | \nabla G  (x)  |  \approx  \lan \nabla G,  x/|x|  \ran  \approx  | x |^{(1-n)/(p-1)} \, \, \mbox{wnenever}\, \,   |x| \geq 4,  \\  
& (b) \hs{.1in} \mbox{There exists} \, \,   \ti \he   >  0\, \, \mbox{and}\, \, c \geq 1 \, \,  \mbox{depending only on $r_0$ and the data with}  \\  
&\hs{.3in}  \,   \,  | k (x)  -  k ( \infty ) |  \leq c   | x |^{- \ti \he}  \, \, \mbox{whenever}\,\, | x | \geq 4.   
\end{split}
\end{align}
Here proportional constants  depend only on $ r_0 $   and   the data.  
\end{lemma} 
\begin{proof} 
Our proof of  Lemma  \ref{lem6.1}  is  similar to the proof  of  Lemma  4.4  in  \cite{AGHLV}.    
Thus we shall often refer to this lemma for details.  We first assume that  
\begin{align} 
\label{6.1}  
\mbox{diam}(E)=1, \, \,  0 \in E,  \, \  B (0, r_0 )  \subset E,   \, \,  
  \mbox{and}\, \,   \eqref{eqn1.8} (i)\, \,   \mbox{holds  for}\, \,   \mathcal{A}. 
\end{align}  
Given  $ R  \geq   4  $,  let $v(\cdot)=v(\cdot, R)$ be the $\mathcal{A}$-harmonic function, defined as in the proof of 
Lemmas \ref{lem5.2} and \ref{lem5.3},  by  
\begin{align}
\label{6.2}
\begin{split}
\left\{
\begin{array}{ll}
\nabla \cdot \mathcal{A}(\nabla v)=0 & \mbox{in} \, \, B(0,R)\setminus E,\\
v=0 & \mbox{on} \, \, \partial E,\\
v=\left\{
\begin{array}{ll} 
\gamma \log R & \mbox{when} \,\, p=n \\
  R^{\xi } &\mbox{when}\, \, n<p<\infty \mbox{ where } \xi = \frac{p-n}{p-1} 
\end{array}
\right.
& \mbox{on}\, \, \partial B(0,R).
\end{array}
\right.
\end{split}
\end{align}    
Here $\gamma$ is as in Definition \ref{defn5.4}.  We  claim there exists  $c_+  \geq 1 $  such that  if  $ x  \in  B ( 0, R ) \sem E$ then    
\begin{align} 
\label{6.3}
\begin{split}
\begin{array}{ll}
(a) \hs{.2in}    c_+^{-1}   \leq   \lan x, \nabla v (x)  \ran  & \mbox{when}\, \,    p = n, \\ 
(b)  \hs{.2in}  c_+^{-1}  v (x) \,  \leq   \lan x, \nabla v (x) \ran & \mbox{when} \, \, p > n    
\end{array}
\end{split}
\end{align}   
where  $ c_+ $  depends only on the data and $ r_0$.    To  prove  \eqref{6.3}  
we  show   for  $ 1  <  \la \leq  101/100 $  that  if  $ x  \in   B (0, R/\la )  \sem E $  then 
\begin{align} 
\label{6.4}  
\begin{split} 
\begin{array}{ll}
(a)  \hs{.2in}  {\ds \frac{ v ( \la x )  -  v (x)}{\la - 1}}\geq    c^{-1}_{++}  &  \mbox{when}\, \,  p   = n,   \\  
(b)  \hs{.2in} {\ds \frac{ v ( \la x )  -  v (x)}{\la - 1}} \geq    c^{-1}_{++} \, v (x) &   \mbox{when}\,\,  p   > n
\end{array}
\end{split}
\end{align}
where $ c^{++} $  depends only on the data and $r_0$.  To prove \eqref{6.4}  suppose 
$ x,  z  \in \ar  E $  with $ | z  - \la x | =  d ( \la x, E). $   Observe from convexity of  
$ E $  and \eqref{6.1}  that   $ d ( \la x,  E )   \approx   \la  - 1,   $  
where constants depend only on $ r_0 $  and the data.  
Let  $  w   =    8  \frac{  \la x  - z }{ |\la x  - z |}  $   and let $  \ph_1$ and $\ph_2 $  be as in  \eqref{1.8}.      
Then  as  in \eqref{1.9} we  have for  $ N $    large enough  that      
\begin{align}  
\label{6.5}   
v ( \la x   )  \geq    c^{-1} v ( w )  \,   \ph_2 ( \la x  -  w )  \geq  \ti c^{-1} ( \la - 1 ),   
\end{align}
  where $ \ti c $  has the same dependence as $ c_{++}$.  Thus  \eqref{6.4}  holds when $ x \in  \ar E. $ 
   The proof of  \eqref{6.4}    for   $  x \in \ar B ( 0, R/\la ) $  is the same as in  the   
   right hand  inequality in  \eqref{1.10} and  below  \eqref{1.58}. Using the boundary maximum principle for 
$\mathcal{A}$-harmonic functions we conclude the validity of   \eqref{6.4} and letting 
$  \la  \to 1 $  we obtain   \eqref{6.3}.   

Extend  $ v  $  to   a  H{\"o}lder continuous function in  $  \rn{n} $  by  setting  $ v \equiv 0 $ 
on  $ E $ while  $ v =  \ga \log R $   when $ p = n $   and  $ v =  R^\xi $ when $ p > n$ on $\rn{n}  \sem  B (0, R)$.  Next we   show that 
$\{x :  v(x)<t\}$ is  convex for every $t\in (0, \gamma \log R)$ when $p=n$ and $t\in (0, R^{\xi})$ when $p>n$. 
 In  order to  make easy reference to  the proof of  Lemma 4.4  in  \cite{AGHLV}      we  define  $u(\cdot)=u(\cdot, R)$ by
\begin{align}
\label{eqn6.2}
\begin{split}
u(x):=\left\{
\begin{array}{ll}
1-\frac{\displaystyle v(x)}{\displaystyle R^{\xi }} & \mbox{when} \, \, n<p<\infty,\\
1-\frac{\displaystyle v(x)}{\displaystyle \gamma\log R} & \mbox{when} \, \, p=n.
\end{array}
\right.
\end{split}
\end{align}
Then it can be easily seen that $0\leq u\leq 1$ in $\mathbb{R}^{n}$, $u\equiv 1$ in $E$, $u\equiv 0$ 
on $\mathbb{R}^{n}\setminus B(0,R)$ continuously, and $u$ is 
$\tilde{\mathcal{A}}$-harmonic in $B(0,R)\setminus E$ where
 $\tilde{\mathcal{A}}(\eta)=- \mathcal{A}(-\eta)$ for $n\leq p<\infty$. 
 Note as in  Remark \ref{rmk1.3}   that $\tilde{\mathcal{A}}$ also satisfies the same structural properties
as $\mathcal{A}$ (i.e., $\tilde{\mathcal{A}}\in M_p(\alpha)$ whenever $\mathcal{A}\in M_p(\alpha)$).
From this observation, it can be easily seen that convexity of $\{x\in \mathbb{R}^{n};\, \, v(x)< t\}$ for  $ t $ in the range stated above  is 
equivalent to convexity of $\{x\in \mathbb{R}^{n};\, \, u(x)>t\}$ for  $ 0 < t < 1. $ 
 Moreover,   it is shown in   Lemma 4.4  of \cite{AGHLV} that $\{x\in \mathbb{R}^{n}: \, \, u(x)>t\}$ is 
convex for every $t\in (0,1)$.     We remark that the range of $p$ considered in \cite{AGHLV} is $1<p<n$  
and  $ u $ is the $\mathcal{A}$-capacitary function  for $ E,  $ so defined in 
$ \rn{n} \sem E. $   
However, the  contradiction type   argument  in \cite{AGHLV} uses only  the  maximum 
principle for $ \mathcal{A}$-harmonic  functions,  assumption  \eqref{6.1},  
and  Lemmas \ref{lemma2.1}-\ref{lemma2.3}  all of which hold in our situation.    
 Thus  $ \{x: v(x)  <  t  \} $ is  convex for 
$ t $ in the intervals stated above when \eqref{6.1} holds.   
Now  we saw, in the proof of Lemmas \ref{lem5.2} and \ref{lem5.3}, that a subsequence
  of $v(\cdot)=v(\cdot, R)$ converges uniformly on compact subsets of $\mathbb{R}^{n}$
   to $G$ as $R\to\infty$  when $ p =n $   and  to   $ b\, G,  b = $ constant  when  $ p > n. $  Moreover,  the corresponding  sequence of  gradients converges 
   uniformly on compact subsets  of  $ \rn{n}\sem E $  to  $  \nabla G$ for $ p = n $ and to  $ b \nabla G $  when $ p  > n.$    
   Thus    $  \{x :\, \, G(x) <  t  \}  $  is  convex  for  $ t \in (0, \infty)$   and \eqref{6.3} holds with 
   $ v $ replaced by  $ G $ whenever  $ x \in   \rn{n}  \sem   E. $  To   remove the 
   assumption on  $ \mathcal{A}$  in  \eqref{6.1}  we  approximate    $ \mathcal{A} $  by a  sequence of  smooth 
   $ \{\mathcal{A}^{(l)}\}$   where $\mathcal{A}^{(l)}  \in M_p  ( \al )$ for $l=1,2,\ldots$ and take 
   limits of the corresponding sequence  $\{ G^{(l)}\} $   (see the  paragraph following (4.35) of Lemma 4.4 in
       \cite{AGHLV} for more details).   Thus  \eqref{6.3}  is  valid  for  $ \mathcal{A}  \in M_p ( \al ) $  whenever  $  B (0, r_0)   \subset  E $  and 
    $ \mbox{diam}(E)=1$. Now \eqref{6.3}  and   \eqref{5.2} $(a)$    imply  \eqref{6.0} $(a) $ when $ p = n $ while
      \eqref{6.3}, \eqref{5.12} $ (a), (d), $   and  \eqref{eqn2.6} $(ii)$  give \eqref{6.0} $(a)$ when $ p > n.$  
      To  prove   \eqref{6.0}  $ (b) $   observe   that $ m ( t, k ) $ is  non-decreasing and  $ M ( t,  k ) $  
      non-increasing on  $ ( 2, \infty ) $  as follows  from the  maximum principle for  
      $  \mathcal{A}$-harmonic functions,  boundedness of  $k$,  and  our knowledge of  $ F$.    
      Using this fact and arguing as in  \eqref{1.26} we now get \eqref{6.0} $(b)$.
   
    Next  if   $ E $ has nonempty  interior,   we  can  translate and dilate  $ E $  to  
    get  a  convex  set  $  \hat E $   with  diameter  1  and  $ B (0, r_0 )  \subset \hat E $ 
    for some $ r_0 > 0. $   Using   \eqref{6.3} for the $\mathcal{A}$-harmonic Green's function
     for $\mathbb{R}^{n}\setminus \hat E $ with pole at infinity and  then   using  a  homothetic  
     transformation to get back  to  $ E $  we conclude  from  
     Remark \ref{rmk1.3} that   Lemma \ref{lem6.1}  is  valid whenever  $ E $ 
     has nonempty interior.    
     
     If  $  E $  has empty interior  choose $ \hat x  \in E $ and $ \rho > 0 $ with 
   $  E  \subset  B ( \hat x, \rho )$.  Let  $  G^{(l)} $  denote the  $\mathcal{A}$-harmonic
    Green's  function for   $\mathbb{R}^{n}\setminus \{ x : \, d ( x, E )  < 1/l \}$ 
    for $l = 1,  2,  \cdots  $ with pole at infinity  and $ \ti G $  the 
    $\mathcal{A}$-harmonic Green's  function for  $\mathbb{R}^{n}\setminus B ( \hat x, \rho )$ 
    with pole at infinity. Then for $l$  large enough it follows from the same argument used  for 
   proving uniqueness of  $ G $  that   $ \ti G  \leq  G^{(l)} \leq   G. $   Using this  fact,   
  Lemmas \ref{lemma2.1}-\ref{lemma2.3}, and  Ascoli's theorem we find that a 
  subsequence of    $\{ G^{(l)}\} $ converges  uniformly to  $ G $ on compact subsets of  $ \rn{n}$ so  
   $ \{ x : G(x)  < t \}$ is  convex for  $ t \in (0,\infty)$.  Finally,   we note that  $  G  -  t  $ is 
   the $\mathcal{A}$-harmonic Green's function for  $ \rn{n} \sem \{ x :  G (x)   \leq   t  \} $ with a pole at infinity.  Since  
   this set is convex with nonempty interior we have  $  \nabla G (x)  \neq  0 $  whenever  
   $ G ( x ) > t $ and  $ t  >  0$.  The proof of  Lemma  \ref{lem6.1}  is now complete.   
   \end{proof}       
 \begin{remark} 
 \label{rmk6.2}  
 We  note that  if  $ \mathcal{A} \in M_p (\al)  $ for fixed $\alpha\in(1,\infty)$  and $  F $ is the  fundamental solution
  with pole at  $ \{0\}$  when  $ p  = n$ defined in Definition \ref{def4.1}  then   
 $  \{ x :  F ( x) <  t   \} $  is convex when  $  t   \in   (- \infty,  \infty ) $   as follows
   from our construction of  $ F $   and  essentially the same proof we gave for $ G. $  Also observe that 
 this remark  is not needed when  $ p > n  $  since the  $\mathcal{A}$-harmonic Green's function $G$ 
 with pole at infinity  for  $ \rn{n} \sem  \{0\} $  and fundamental solution $F$ with pole at $\{0\}$ defined in Definition \ref{def4.2}  are the same.
 \end{remark}   
\subsection{Proof of the Brunn-Minkowski inequality}
\begin{proof}[Proof of \eqref{BMn} and \eqref{BMpn} in Theorem \ref{theoremA}]
Let $E_1$ and $E_2$ be compact convex sets.  Note  from    \eqref{5.23} that if $ p > n $ 
 and either $E_1$  or $ E_2$ is a single point, then equality holds in
     \eqref{BMpn} and  Theorem \ref{theoremA} is trivially true.    
 Thus we assume  $ E_1$ and $E_2 $  each contain at  least two points
when $p\geq n$. For $i=1,2$, let $G_i$ be the $\mathcal{A}$-harmonic 
 Green's function for $\mathbb{R}^{n}\setminus E_i$ with pole at $\infty$ 
 obtained in Lemmas \ref{lem5.2} and \ref{lem5.3} for $n\leq p<\infty$. 
 Let $\mathcal{C}_{\mathcal{A}}(\cdot)$ be defined as in Definition \ref{defn5.4}. 
 For fixed $\lambda\in (0,1)$, let $G$ be the $\mathcal{A}$-harmonic Green's 
 function for $\mathbb{R}^{n}\setminus [\lambda E_1+(1-\lambda)  E_2]$ with 
 pole at $\infty$. As observed in  \cite{B1,CS,AGHLV},  it is enough to prove  
  for arbitrary convex compact sets $ E_i'$ containing at least two points when $p\geq n$ for  $i =1, 2$ that
\begin{align}
\label{eqn4.16pn}   
\mathcal{C}_{\mathcal{A}}( E_1'   + E_2')^{\frac{1}{p-n}}     \geq    
\mathcal{C}_{\mathcal{A}} ( E_1')^{\frac{1}{p-n}}  \, + \,        \mathcal{C}_{\mathcal{A}} ( E_2')^{\frac{1}{p-n}}  \quad \mbox{when}\, \, n<p<\infty
\end{align}                                    
 and
 \begin{align}
\label{eqn4.16n}   
\mathcal{C}_{\mathcal{A}}( E_1'   + E_2') \geq    \mathcal{C}_{\mathcal{A}} ( E_1') \, + \,        \mathcal{C}_{\mathcal{A}} ( E_2') \quad \mbox{when}\, \, p=n. 
\end{align}   
To get \eqref{BMpn} from   \eqref{eqn4.16pn} when 
 $n<p<\infty$ (and \eqref{BMn} from \eqref{eqn4.16n} when $p=n$) put  
 \[
 E_1'  =  \la E_1 \quad \mbox{and} \quad E_2' =  (1-\la) E_2
 \]
and use $(p-n)$-homogeneity of $\mathcal{C}_{\mathcal{A}}(\cdot)$ 
(and $1$-homogeneity of $\mathcal{C}_{\mathcal{A}}(\cdot)$ when $p=n$).  To get \eqref{eqn4.16pn} from
  \eqref{BMpn} (and to get \eqref{eqn4.16n} from  \eqref{BMn}) one 
  can take $\lambda=1/2$ and use $(p-n)$-homogeneity of
   $\mathcal{C}_{\mathcal{A}}(\cdot)$ when $n<p<\infty$ (and $1$-homogeneity
    of $\mathcal{C}_{\mathcal{A}}(\cdot)$ when $p=n$) once again. 
    On the other hand, to  prove  \eqref{eqn4.16pn} when $n<p<\infty$ and \eqref{eqn4.16n} when $p=n$   
it suffices to show,  for  all $  \la  \in  (0,1) $, that    
\begin{align}  
\label{eqn4.17}
\mathcal{C}_{\mathcal{A}}(\la E_1''   + (1-\la)  E_2'')     \geq    \min  \left\{ \mathcal{C}_{\mathcal{A}}( E_1''),     
\mathcal{C}_{\mathcal{A}}( E_2'')  \right\}
\end{align}   
whenever  $ E_i''$ for $i =1, 2$  are compact convex sets containing at
 least two points when $p\geq n$ and  for fixed $p$ with $n\leq p<\infty$. 
 To  get   \eqref{eqn4.16pn} from \eqref{eqn4.17} when $n<p<\infty$ let  
\[  
E_i''  =  \frac{E_i'}{\mathcal{C}_{\mathcal{A}} ( E_i')^{\frac{1}{p- n} } }  \quad \mbox{for}\, \,  i  = 1, 2 
\quad \mbox{and}\quad
  \la  = \frac { \mathcal{C}_{\mathcal{A}}( E_1')^{\frac{1}{p-n}} }{ 
 \mathcal{C}_{\mathcal{A}}( E_1')^{\frac{1}{p-n}} +
 \mathcal{C}_{\mathcal{A}}( E_2')^{\frac{1}{p-n}}}
  \]   
 then use $(p-n)$-homogeneity of $\mathcal{C}_{\mathcal{A}}(\cdot)$ and do
  some algebra. Similarly, to get \eqref{eqn4.16n} from \eqref{eqn4.17} when $p=n$ we let
 \[  
E_i''  =    \frac{E_i'}{\mathcal{C}_{\mathcal{A}} ( E_i')}  \quad \mbox{for}\, \,  i  = 1, 2 
\quad \mbox{and}\quad
  \la  = \frac {\mathcal{C}_{\mathcal{A}}( E_1')}{  \mathcal{C}_{\mathcal{A}} ( E_1')+  \mathcal{C}_{\mathcal{A}} ( E_2')}.
  \]
Finally, one can easily get \eqref{eqn4.17} from \eqref{BMpn} 
when $n<p<\infty$ (from \eqref{BMn} when $p=n$). Hence we 
conclude that \eqref{BMpn}, \eqref{eqn4.16pn}, and \eqref{eqn4.17} 
are all equivalent when $n<p<\infty$. Similarly, \eqref{BMn}, \eqref{eqn4.16n}, 
and \eqref{eqn4.17} are all equivalent when $p=n$.  
Therefore, we focus on proving \eqref{eqn4.17} for $n\leq p<\infty$ and also for ease of notation 
we shall just use $E_i$ instead of $E''_i$ in \eqref{eqn4.17}.  For fixed $ \la  \in (0,1) $, we claim that
\begin{align}
\label{3uustar}
\begin{split}
G(x)\leq G^*(x)  :=\inf
\left\{
\max\{G_1(y), G_2(z)\};
\begin{array}{l} 
x=\lambda y+(1-\lambda) z, \\
\lambda\in [0,1], y,  z   \in \mathbb{R}^n 
\end{array}
\right\}
\end{split}
\end{align}
whenever $x\in\mathbb{R}^{n}$. Assume that \eqref{3uustar} 
holds for the moment and we show how to 
get \eqref{eqn4.17} from \eqref{3uustar}. It follows from \eqref{3uustar} and definition of the $\mathcal{A}$-harmonic Green's function that
\begin{align*}
-k(x)=F(x)-G(x) &\geq F(x)-G^*(x) \\
&\geq \min\{F(x)-G_1(x), F(x)-G_2(x)\} =\min\{-k_1(x), -k_2(x)\} 
\end{align*}
where $F$ is as in Lemma \ref{lem1.4} when $p=n$ and as in 
Lemma \ref{lem1.5} when $n<p<\infty$. Here $k_1,k_2,k$ are the functions 
appearing  in Lemma \ref{lem5.2} when $p=n$ and in Lemma \ref{lem5.3} 
when $n<p<\infty$ associated to $E_1, E_2, \lambda E_1+(1-\lambda) E_2$ 
respectively. Using the definition of $\mathcal{C}_{\mathcal{A}}(\cdot)$ and this inequality  when $n<p<\infty$ we have   
\begin{align*}
\mathcal{C}_{\mathcal{A}}(\lambda E_1+(1-\lambda)E_2) &=(-k(\infty))^{p-1}=\lim\limits_{|x|\to\infty} [F(x)-G(x)]^{p-1} \\
&\geq \min\left\{ \lim\limits_{|x|\to\infty}[F(x)- G_1(x)]^{p-1}\, \,, \, \,\lim\limits_{|x|\to\infty} [F(x)-G_2(x)]^{p-1}\right\} \\
&=\min\{(-k_1(\infty))^{p-1},(-k_2(\infty))^{p-1}\} =\min\{\mathcal{C}_{\mathcal{A}}(E_1), \mathcal{C}_{\mathcal{A}}(E_2)\}.
\end{align*}
While when $p=n$, the inequality above gives us (where $\gamma$ is as in Remark \ref{rmk5.5})
\begin{align*}
\mathcal{C}_{\mathcal{A}}(\lambda E_1+(1-\lambda)E_2) &=e^{-k(\infty)/\gamma}=\lim\limits_{|x|\to\infty} e^{\frac{F(x)-G(x)}{\gamma}} \\
&\geq \min\left\{ \lim\limits_{|x|\to\infty} e^{\frac{F(x)-G_1(x)}{\gamma}}\, \,, \, \, \lim\limits_{|x|\to\infty} e^{\frac{F(x)-G_2(x)}{\gamma}}\right\} \\
&= \min\{e^{-k_1(\infty)/\gamma}, e^{-k_2(\infty)/\gamma}\} = \min\{\mathcal{C}_{\mathcal{A}}(E_1), \mathcal{C}_{\mathcal{A}}(E_2)\}.
\end{align*}
Hence \eqref{eqn4.17} is true for all $p$ with $n\leq p<\infty$ and in view of 
our earlier remarks we conclude that  the proof of \eqref{BMpn}
when $n<p<\infty$ and \eqref{BMn} when $p=n$ in Theorem \ref{theoremA}   are complete assuming
 \eqref{3uustar}. 

We now return to the proof of \eqref{3uustar}. As in  the proof of Lemma \ref{lem6.1}, for $R>>1$   
we consider $\tilde{v}(\cdot)=\tilde{v}(\cdot,R)$ which is the solution to the 
Dirichlet problem as in \eqref{6.2} where $v$ is replaced by $\tilde{v}$ 
whenever $\tilde{v}\in \{v_1, v_2, v\}$ corresponding to $\ti{E}\in \{E_1, E_2, \lambda E_1+(1-\lambda)E_2)\}$.
 Extend $\tilde{v}$ to $\mathbb{R}^{n}$ by putting $\tilde{v}\equiv 0$ on $\tilde{E}.$ Also if $ p = n, $ set  $ \ti v = \gamma\log R$  and if $ p > n,$  set  $ \ti v = R^{\xi}$ 
on $ \rn{n} \sem B (0, R ). $  (where $\xi=(p-n)/(p-1)$ and $\gamma$ is as in Definition \ref{defn5.4}) when $n<p<\infty$. We know from the proof of Lemmas \ref{lem5.2} 
 and \ref{lem5.3} that a subsequence of $\tilde{v}(\cdot)=\tilde{v}(\cdot, R)\in \{v_1, v_2, v\}$  
converges uniformly on compact subsets of $\mathbb{R}^{n}$ to $\tilde{G}\in \{G_1, G_2, G\}$. 
From this observation, we see that in order to prove \eqref{3uustar}, it is enough to show that
\begin{align}
\label{vvstarv1v2}
\begin{split}
v(x)\leq v^*(x)  :=\inf
\left\{
\max\{v_1(y), v_2(z)\};
\begin{array}{l} 
x=\lambda y+(1-\lambda) z, \\
\lambda\in [0,1], y,  z   \in \mathbb{R}^n 
\end{array}
\right\}
\end{split}
\end{align}
whenever $x\in\mathbb{R}^{n}$. Once again as in Lemma \ref{lem6.1}, 
we consider $\tilde{u}(\cdot)=\tilde{u}(\cdot, R)\in\{u_1, u_2, u\}$ which 
is defined as in \eqref{eqn6.2} with $u$ is replaced by $\tilde{u}$ whenever
 $\tilde{v}\in \{v_1, v_2, v\}$. From this, we observe that $0\leq \tilde{u}\leq 1$ 
 in $\mathbb{R}^{n}$, $\tilde{u}\equiv 1$ in $\tilde{E}$, $\tilde{u}\equiv 0$ 
on $\mathbb{R}^{n}\setminus B(0,R)$ continuously, and $\tilde{u}$ is 
$\tilde{\mathcal{A}}$-harmonic in $B(0,R)\setminus \tilde{E}$ where
 $\tilde{\mathcal{A}}(\eta)=- \mathcal{A}(-\eta)$ for $n\leq p<\infty$. 
 As observed earlier $\tilde{\mathcal{A}}$ also satisfies the same structural properties
  as $\mathcal{A}$. We see that \eqref{vvstarv1v2} is equivalent to
\begin{align}
\label{uustaru1u2}
\begin{split}
u(x)\geq u^*(x)  :=\sup
\left\{
\min\{u_1(y), u_2(z)\};
\begin{array}{l} 
x=\lambda y+(1-\lambda) z, \\
\lambda\in [0,1], y,  z   \in \mathbb{R}^n 
\end{array}
\right\}
\end{split}
\end{align}
whenever $x\in\mathbb{R}^{n}$. A proof of \eqref{uustaru1u2} can be found in \cite{AGHLV} at (5.13). Once again the range of $p$ in that
 article is for  $1<p<n$ and   $\mathcal{A}$-capacitary functions   but  uses only Lemmas \ref{lemma2.1}-\ref{lemma2.3}   
 and  the maximum principle for $ \mathcal{A}$-harmonic functions so is valid in our situation.   This finishes the proof of
  \eqref{uustaru1u2} and in view of our earlier observations proof of \eqref{BMpn} 
when $n<p<\infty$ as  well as  \eqref{BMn} when $p=n$ in Theorem \ref{theoremA}.
\end{proof}
\subsection{Final proof of Theorem \ref{theoremA}}
In this subsection our aim is to prove that  if equality occurs in \eqref{BMpn}
when $n<p<\infty$ or in \eqref{BMn} when $p=n$ in Theorem \ref{theoremA}  and
 if $\mathcal{A}$ satisfies the additional structural assumptions in \eqref{eqn1.8} 
  then $E_2$ is a translation and dilation of $E_1$. 
  To do this, using the additional structural assumptions on $\mathcal{A}$, 
  we first construct the corresponding fundamental solution explicitly.
\subsubsection{Construction of the fundamental solution}
In this subsection   we begin  with some observations when   
$   \mathcal{A}  \in  M_p ( \al ) $  and additionally satisfies \eqref{eqn1.8}  (ii): 
\begin{align}
\label{eqn6.3}
\mathcal{A}_{i}=\frac{\partial f}{\partial \eta_i}(\eta), \, \, 1\leq i\leq n, \, 
\mbox{where}\, \, f(t\eta)=t^{p} f(\eta)\, \, \mbox{when}\, \, t>0,\,\,\eta\in\mathbb{R}^{n}\setminus \{0\}
\end{align}
  so $ f $  has continuous  second partials on $ \rn{n} \sem\{0\}$. 
 Using some ideas from \cite{CS1}, the fundamental solution was 
 constructed in \cite[section 7.2]{AGHLV} when $1<p<n$ associated to $\mathcal{A}$-harmonic PDEs. In this 
 part   we extend this   construction to  include   $n\leq p<\infty$. 

 We can write $f(\eta)=(k (\eta))^{p}$ and from \eqref{eqn6.3} we see that 
$ k ( \eta )$ for $\eta \in  \rn{n} \sem  \{0\}$ is  homogeneous of degree $1$ 
and has continuous second partials on $ \rn{n} \sem \{0\}. $  Also  in (7.4)  of  \cite{AGHLV}  it is  shown that  
\begin{align}
\label{eqn6.4} 
  k^2 \, \,  \mbox{is strictly convex  on}\, \,  \mathbb{R}^n.   
  \end{align}  
From        \eqref{eqn6.4} and  well-known properties of  the  support 
function for  a  convex set    we see  that if $ X \in  \rn{n} \sem \{0\},$  then      
\[   
h  ( X )   =  \sup  \{  \lan \eta,  X  \ran :\, \,   \eta \in  \{k \leq  1\}  \}  
\] 
has continuous second partials  and  $ h $ is homogeneous of  degree $1$.   Moreover,  
\begin{align}
\label{eqn6.5} 
\nabla  h ( X ) =   \eta  ( X )\, \, \mbox{where}\, \,  \eta  \, \, \mbox{is the point in}
\, \, \{  k = 1 \}\, \, \mbox{with}\, \,   \frac{X}{|X|}=  \frac{\nabla k (\eta  )}{| \nabla k  (\eta )|}.
\end{align}    
From  calculus and  Euler's formula for $1$-homogeneous
 functions it now follows that  if  $ X  \in   \mathbb{S}^{n-1}$ then   
  \begin{align*}
  h (X)  =   \lan  \eta ( X ),  X  \ran  =  |X|  \lan \, \eta (X),  \frac{ \nabla k (\eta)}{ | \nabla k (\eta) | } \, \ran  =  \frac{ |X|  }{ | \nabla k (\eta) |}.
  \end{align*}  
  Using this equality  we obtain first  
  \begin{align*}
\nabla k  (  \nabla h ( X ) )   \, =  \, \frac{ |  \nabla k (\eta)|  X}{ |X| } =  
\frac{ X }{ h ( X ) }
\end{align*}
and  second using  $1$-homogeneity of  $ k, h $ as well as 0-homogeneity of  $ \nabla k,  \nabla h, $  that 
\begin{align}
\label{eqn6.8} 
  k [ h (X)  \, \nabla h ( X ) ]    \, \,   \nabla k  [ h (X)  \nabla h ( X ) ]  =
h (X) \,  k  [ \nabla h (X)]   ( X/ h (X) )  = X. 
\end{align}
Thus  $ k  \, \nabla k  $  and  $ h \,  \nabla h $  are inverses of  each other on $  \rn{n} \sem \{0\}. $ 

For fixed $p$, $n< p<\infty$, $\xi=(p-n)/(p-1),$ and for $X \in \rn{n} \sem \{0\}$ we define
\[
\hat{\mathcal{F}}(X)=
\left\{
\begin{array}{ll}
(h(X))^\xi &  \mbox{when}\, \,  n<p<\infty,\\
\log h(X) &\mbox{when}\, \, p=n.
\end{array}
\right.
\]
We claim that  $ \hat{\mathcal{F}}$ is  a constant multiple of
  the fundamental solution for  $\mathcal{A} $ in  \eqref{eqn6.3} in the sense of Definition \ref{def4.1}. 
  Indeed,  if    $ X \in \mathbb{R}^{n}\setminus \{0\}, $  it  follows from \eqref{eqn6.5}-\eqref{eqn6.8} that
\begin{align}
\label{eqn6.9}
\begin{split}
(\nabla f) (\nabla \hat{\mathcal{F}}(X)) &=  p\,k^{p-1} (\nabla \hat{\mathcal{F}} (X))\, \, (\nabla k (\nabla \hat{\mathcal{F}}(X)) ) \\
 & =p\,  \frac{X}{h(X)}  \, k^{p-1} (\nabla \hat{\mathcal{F}} (X) )  \\ 
  &=  p\,  X \xi^{p-1} \, h^{[( \xi  - 1 )(p - 1) - 1 ] } (X) \, k^{p-1}  (\nabla  h (X))  \\
  &=   p\, X\, \xi^{p-1}  h^{-n}(X)
\end{split}
\end{align}
when $n<p<\infty$.  If   $p=n$, then using \eqref{eqn6.5}-\eqref{eqn6.8}, we also have 
\begin{align}
\label{eqn6.9p=n}
\begin{split}
(\nabla f) (\nabla \hat{\mathcal{F}}(X)) &=  n\, k^{n-1} (\nabla \hat{\mathcal{F}} (X))\, \, (\nabla k (\nabla \hat{\mathcal{F}}(X)) ) \\
 & = n\frac{X}{h(X)} \,  \, k^{n-1} (\nabla \hat{\mathcal{F}} (x) )  \\ 
  &=  n \frac{X}{h(X)}\, h^{1-n} (X) \, k^{n-1}  (\nabla  h (X))  \\
  &=   n\, X h^{-n}(X).
\end{split}
\end{align}
In both cases when $n<p<\infty$ and $p=n$ we have that the right hand side in 
\eqref{eqn6.9} and \eqref{eqn6.9p=n} are a constant times $X h^{-n}(X)$.   
Now  $ X \mapsto   h^{-n} ( X/|X|) $ is homogeneous of  degree $0$  so   
\[
\lan  X,  \nabla [ h^{-n} (X/|X| )] \ran  = 0
\]
by Euler's formula.  From this  observation and  \eqref{eqn6.9} when $n<p<\infty$ and \eqref{eqn6.9p=n} when $p=n$ we deduce    
\begin{align*}  
\begin{split}
   \nabla \cdot \left(   (\nabla f) (\nabla \hat{\mathcal{F}}(X)) \right) &= h^{-n} (X/|X| ) \, \, \nabla \cdot ( X |X|^{-n} )  +   |X|^{-n}  \lan X, \nabla [ h^{-n} (X/|X| )]  \ran \\
    &= 0
    \end{split}
\end{align*}   
whenever  $  X  \in  \rn{n}  \sem  \{0\}$ and for fixed $p$, $n\leq p<\infty$. Hence  $  \hat{\mathcal{F}}$ is  $  \mathcal{A}=\nabla f$-harmonic in 
$ \rn{n} \sem \{0\}. $  Also  from  \eqref{eqn6.9} when $n<p<\infty$ and \eqref{eqn6.9p=n} when $p=n$  we  note  that   
\[  
 | \nabla f(\nabla \hat{\mathcal{F}} (X) ) | \approx  |  X  |^{1-n} \quad \mbox{on}\, \,   \rn{n}  \sem \{0\}.  
 \]   
 If   $ \he  \in  C_0^\infty ( \rn{n} ) $  then from the above display we deduce  that  the function $ X  \mapsto  
 \lan   \nabla f  (\nabla\hat{\mathcal{F}} (X) ),  \nabla  \he (X) \ran $ is integrable on  $ \rn{ n}. $ 
 Using this fact,  smoothness of  $  f,  h,  $  and an  integration by parts,  we get   
\begin{align}
\label{eqn6.11}
\begin{split}
\int_{\rn{n}}  \lan   \nabla f  (\nabla \hat{\mathcal{F}} (X)),  \nabla  \he (X)  \ran dx  &=
 -  \lim\limits_{r \to  0}  \int_{\ar  B (0, r )} \,  \he (X) \,  \, \lan \nabla f ( \nabla \hat{\mathcal{F}} (X) ), X/|X| \ran 
\, d \mathcal{H}^{n-1}    \\
&=  C \,  \he (0).
\end{split}
\end{align}
Using   \eqref{eqn6.9} once again  when $n<p<\infty$ it follows that   
\begin{align}
\label{eqnbp}
\begin{split}   C&=  - \, \lim_{r\to  0 }    \int_{\ar B (0, r) } \, \lan \nabla f ( \nabla \hat{\mathcal{F}} (X) ), X/|X| \ran 
\, d \mathcal{H}^{n-1} =  p \xi^{p-1}     \, \int_{\ar B (0, 1)}  h^{-n} ( X/|X| ) \, d \mathcal{H}^{n-1}  \\
&= p \xi^{p-1}\, \int_{\mathbb{S}^{n-1}}  h^{-n}(\omega)  \, d\omega. 
\end{split}
\end{align}
While when $p=n$ we use  \eqref{eqn6.9p=n} to get
\begin{align}
\label{eqnbn}
\begin{split}
C&=  - \, \lim_{r\to  0 }    \int_{\ar B (0, r) } \, \lan \nabla f ( \nabla \hat{\mathcal{F}} (X) ), X/|X| \ran 
\, d \mathcal{H}^{n-1} = n  \, \int_{\ar B (0, 1)}  h^{-n} ( X/|X| ) \, d \mathcal{H}^{n-1}\\
&=n\, \int_{\mathbb{S}^{n-1}}  h^{-n}(\omega)  \, d\omega. 
\end{split}
\end{align}
 \begin{remark} 
\label{rmk7.1}
In view of \eqref{eqn6.11} and \eqref{eqn6.9} when $n<p<\infty$ and \eqref{eqn6.9p=n} when $p=n$
\begin{align}
\label{fundamentalsolutionF}
\begin{split}
\mathcal{F}(x)=C^{\frac{-1}{p-1}}\hat{\mathcal{F}}(x)=
\left\{
\begin{array}{ll}
C^{\frac{-1}{p-1}} h(x)^{\xi}  &\mbox{when} \, \, n<p<\infty,\\
C^{\frac{-1}{n-1}} \log h(x)  & \mbox{when} \, \, p=n
\end{array}
\right.
\end{split}
\end{align}
where $\xi=(p-n)/(p-1)$ and $C$ is as in \eqref{eqnbp} when $n<p<\infty$ and as in \eqref{eqnbn} when $p=n$. 
\end{remark}

Finally using the fact that $ k  \nabla k $ and $ h \nabla h $  are inverses it  follows, as in 
the argument from  (7.10) in the proof of  Lemma 6.2 of  \cite{AGHLV},  
that for some $ \tau >0,$ depending only on the data,
 \begin{align}
\label{eqn5.1} 
 \frac{\mathcal{F}_{\om \om} (X) }{ |\nabla \mathcal{F}(X) |}  \geq  \tau  >  0  \quad \mbox{whenever}\, \,   \om, X  \in \mathbb{S}^{n-1}\, \, \,  \mbox{with}\, \, \,  \lan \nabla \mathcal{F}(X), \om \ran  =  0.
\end{align} 
To  set the stage for our next lemma we now assume that in  addition to  $ \mathcal{A} = \nabla  f$ (as in  \eqref{eqn6.3}) that  \eqref{eqn1.8}  $(i)$ holds.    Let $E_1$ and $E_2$ be convex compact sets containing at least two points  and 
let $\lambda \in (0,1)$ be fixed.  Let $\tilde{G}\in \{G_1, G_2, G\}$ be the 
$\mathcal{A} =   \nabla f$-harmonic Green's function for $\mathbb{R}^{n}\setminus \tilde{E}$ 
with pole at infinity whenever $\tilde{E}\in\{E_1, E_2, \lambda E_1+(1-\lambda) E_2\}$. 
Let $ \mathcal{F} $ be the corresponding  fundamental solution as in \eqref{fundamentalsolutionF}. 
Using \eqref{eqn5.1} we next show  
\begin{lemma}  
\label{lemma5.1}
There  exists $ R_1 =  R_1 (  \tilde{G},  \al, \La, p, n ),    $  such that  if  $ \tilde{G}   \in  \{ G_1, G_2, G  \}, $ then   
\begin{align*}
  \frac{\tilde{G}_{\ti\om \ti \om} (x) }{ |\nabla  \tilde{G}(x) |}  \geq  \frac{\tau}{2 |x|}  >  0  \quad  \mbox{whenever}\,\, \ti \om  \in \mathbb{S}^{n-1} \, \, \mbox{and} \, \, |x| >  R_1\, \,   \mbox{with}\, \,  \lan \nabla \tilde{G}     (x), \ti \om \ran  =  0.
\end{align*} 
 \end{lemma}        
\begin{proof}  
The statement and proof of   Lemma   \ref{lemma5.1}   is  essentially the same as in Lemma  6.1  of  \cite{AGHLV}.  It should be noted  though that in  Lemma 6.1,  $ G $ is the   $ \mathcal{A}$-harmonic fundamental solution with pole at  0   while  $ \ti u $  plays the role of the   $ \mathcal{A}$-harmonic  Green's function.   Also since the  fundamental solution  tends to zero at  $ \infty $  when  $ 1 < p < n$  the  curvatures on levels of this function are necessarily negative.   

To give a  brief outline of the proof of Lemma  \ref{lemma5.1},  let   $ \ti G = 
  \mathcal{F}  +  \ti k. $   From  Lemmas \ref{lem1.4} and  \ref{lem5.2} when $p=n$ and 
      Lemmas   \ref{lem1.5} and \ref{lem5.3} when $ p > n $  we  see as in  \eqref{1.43}-\eqref{1.45}  that  $ \ti k  $  is  a weak  solution to   
   \begin{align}  
    \label{eqn5.5}
\mathcal{L} \ti k :=\sum_{i,j=1}^n\frac{ \ar }{ \ar y_i} \biggl(  \bar a_{i j }( y )  \frac{ \ar \ti k}{ \ar y_j}\biggr )=0  
\end{align}  
 on  $ B (x,  |x|/2)$  with  $|x|  \geq  R_0$    when $ \ti E  \subset  B ( 0,  R_0/2) $. Here   
\begin{align*}
\bar a_{ij} ( y )   =   \int_0^1   \frac{ \ar^2  f }{\ar \eta_i \ar \eta_j }   \left( t  \nabla  \ti   G(y)  +  (1-t)  \nabla \mathcal{F} (y)   \right) d t.
\end{align*}
Moreover, for  some  $ c  \geq 1, $ independent of $ x, $  we also have
\begin{align}
\label{eqn5.7}   
c^{-1} \bar \si (y)  \,   | \xi |^2  \, \leq \,   \sum_{i,j=1}^n   \bar a_{ij} ( y  )  \xi_i  \xi_j   \leq c \, \bar \si (y) \, |\xi|^2  
\end{align}
whenever $\xi  \in \rn{n} \sem \{0\}$ where  $\bar \si$ satisfies 
\begin{align}
\label{eqn5.8} 
\bar \si ( y  )  \approx  ( |  \nabla \bar G (y) | +    |  \nabla \mathcal{F}(y)  |  )^{p-2}  \approx |y|^{\frac{(1 - n)(p-2)}{p-1}} 
 \end{align}
for $|y|  \geq  R_0$.  Also  from  \eqref{5.2} $(c)$   when $ p = n $  and   
\eqref{5.12}  $(c)$ when $n<p<\infty$ we  see there exists  $ \hat r_0  >  R_0$ and  $\be > 0$  such that  
\begin{align}  
\label{6.31} 
|\ti k (y) - \ti k  ( \infty)  |   \leq   \hat r_0 |y|^{- \be }  
\end{align}
when  $ | y|  \geq \hat r_0. $    Constants depend  on various quantities but are independent
 of  $ y $ provided  $ |y|  \geq  \hat r_0 .$   From  well-known  results for  
 uniformly elliptic  PDE (see \cite{GT})  we deduce from  \eqref{eqn5.7}-\eqref{6.31} that  
   \begin{align}   
\label{eqn5.9}
\begin{split}
|x|^{-n/2} \left(  \int_{ B ( x, |x|/4)} |\nabla \ti k |^2   \, dy \right)^{1/2}  \,    &\leq c \, |x|^{-1}  \max_{B ( x, |x|/2)}  \, | \ti k   - \ti  k ( \infty) |   \\
& =  O\left( |x|^{- 1 - \be } \right)\, \, \mbox{as}\, \, x \to  \infty
\end{split}
\end{align}
where $ c $  as above  depends on various quantities but is  independent of  $ x. $     Using  \eqref{eqn5.9},   Lemma \ref{lem6.1}, 
as well as  Lemma \ref{lemma2.2}  for $ \mathcal{F},  \ti G, $  and  arguing as in  (6.8)-(6.15) of  \cite{AGHLV}, we  eventually  obtain  
\begin{align}
\label{eqn5.17}      
|\nabla  \ti k  |  =   o \left(|x|^{\frac{1-n}{p-1}}\right)  \quad \mbox{and}\quad  \sum_{i, j =1}^n   \left|  \frac{ \ar^2  \ti k}{ \ar x_i  \ar x_j }  \right| 
=  o\left( |x|^{\frac{2-n-p}{p-1}} \right)  \mbox{ as  } x \to  \infty. 
\end{align}
Now \eqref{eqn5.17} and \eqref{eqn5.1}    imply   Lemma \ref{lemma5.1} as in the paragraph following  (6.15)    of  \cite{AGHLV}.    
 
We next consider $G^*$ defined as earlier in \eqref{3uustar}.
If equality   holds in   either  \eqref{BMn} when $ p  = n$  or   \eqref{BMpn}  when $ p > n,  $ 
then using  convexity  of  the  domains bounded by the levels of $ G,G_1, G_2, $ and    
repeating the    argument   above  with  $u_1, u_2, u$  replaced by $G_1, G_2, G$, we see that   $ G^* 
= G $  and  so    
\begin{align}
\label{u=ustar}
\{G(x)\leq t\} = \lambda \{G_1(y)\leq t\} + (1-\lambda)\{G_2(z)\leq t\}
\end{align}
whenever $n\leq p<\infty$ and $t\in (0,\infty)$. 

  We now  use  \eqref{u=ustar} and Lemma \ref{lem6.1}   to study the 
  support functions of     the  convex  domains bounded by the levels of $G_1, G_2$, 
and $G.$     
Let $t>0$ be fixed and let $h_i(\cdot, t)$ be the support function 
of  $\{G_i\leq t\}$ for $i=1,2$ while  $h(\cdot, t)$ is the support 
function of $\{G\leq t\}$ defined for $X\in\mathbb{R}^{n}$ and $0<t<\infty$ by
\[
h_{i}(X,t):=\sup_{x\in \{G_i\leq t\}}\langle X,x\rangle, \, \,\mbox{for}\,\, i=1,2, 
\quad \mbox{and}\quad h(X,t):=\sup_{x\in \{G\leq t\}}\langle X,x\rangle.
\] 
From the properties of a support function and \eqref{u=ustar} we have 
\begin{align}
\label{h=h1h2}
h(X,t)=\lambda h_1(X,t)+(1-\lambda)h_2(X,t)
\end{align}
whenever  $X\in\mathbb{R}^{n}$ and $t>0$.

We  note from  Lemma \ref{lem6.1}  and Lemma \ref{lemma2.2} 
 that  
$   \nabla \tilde{G}  \neq 0 $  and  
   $ \tilde{G}$ has  locally H\"{o}lder continuous second partials  
   in    $ \{ \tilde{G}  >   0\}  $ whenever  $  \tilde{G}   \in \{ G_1, G_2,  G   \}. $  
   From Lemma \ref{lemma5.1}  we see there exists  $ t_0$ large 
   and  $\tau_0  >0   $  small  and  $ R_0 $  large such that 
   if $ \tilde{G}   \in \{ G_1, G_2,  G  \}$ then 
 \begin{align}
 \label{eqn6.25}
\begin{split} 
 & (\star)  \hs{.4in}  \{ \tilde{G} \geq  t\}   \subset  \rn{n} \sem \bar B (0, R_0) \,  \mbox{ for  } \, 0<t_0<t,  \\ 
 &(\star\star)  \hs{.23in} \frac{\tilde{G} _{\ti\om \ti \om} (x) }{ |\nabla  \tilde{G}(x) |}  \geq  \tau_0/|x| \,  \mbox{ whenever }\, \ti \om  \in \mathbb{S}^{n-1} \, \, \mbox{and}\, \,  |x| \geq   R_0 \, \,  \mbox{with} \, \, \lan \nabla \tilde{G}  (x), \ti \om \ran  =  0. 
\end{split}
\end{align} 
 Hence we conclude from  \eqref{eqn6.25}  that  the curvatures  
 at  points on   $ \{ \tilde{G} = t \} $  are bounded away from  $0$ 
 when  $  t  \geq  t_0 $  and  $\tilde{G}\in \{G_1, G_2, G\}$.  
 This yields 
  \[
  \frac{\nabla \tilde{G} }{ | \nabla \tilde{G}  | }\, \, \mbox{is  a  1-1  mapping    from} \, \,  \{ \tilde{G} = t \}\, \,  \mbox{onto}\, \,  \mathbb{S}^{n-1}
\]
and  
  \[
   \left(\frac{\nabla \tilde{G} }{ | \nabla \tilde{G}  | }, \tilde{G} \right)\, \,  \mbox{is  a  1-1 mapping  from}\, \,   \{ \tilde{G} > t_0 \} \, \,  \mbox{onto} \, \, \mathbb{S}^{n-1}  \times  (t_0,\infty). 
   \]
It follow from \eqref{eqn6.25}  and  the inverse function theorem  
that if    $ \tilde{h} $ is the support function corresponding to  
$ \tilde{G} \in \{G_1, G_2, G \} $ and  
$ t_0<t<\infty, $   then   $\tilde{h}$ has  H\"{o}lder 
continuous second partials in $X$  and     
\begin{align} 
\label{eqn5.26}
\nabla_X \tilde{h} ( X, t ) =  \tilde{x}( X, t ) \in  \{ x_1 ( X, t ),  x_2 ( X, t ),  x ( X, t ) \}  
\end{align}
where  $ \tilde{x}$ is the point in $ \{ \tilde{G}= t \} $ with
\[
\frac{X}{|X|}  =\frac{\nabla \tilde{G} (\tilde{x})}{| \nabla \tilde{G} (\tilde{x}) |}.
\]
We now repeat the argument from (6.22) to (6.42) in \cite{AGHLV} to eventually conclude  for  fixed $ t > t_0 $ and all  $ X \in \mathbb{S}^{n-1} $ that
 
\begin{align}
\label{eqn6.45}
  \frac{\ar}{\ar X_i } \left(\frac{| \nabla G_2 |  (x_1) }{ | \nabla G_1 |  (x_2)} 
\right) = 0   \quad \mbox{and}\quad  \frac{\ar}{\ar X_i }  \left(    x_1    -     x_2  \,   
\frac{| \nabla G_2 |  (x_1) }{ | \nabla G_1 |  (x_2) } 
\right)   = 0 .  
\end{align}

Since   $ x_1, x_2 $ are  smooth for every fixed $t$, there exists $ a = a(t),  b = b(t) \in \re $ with 
\[
x_2 ( X, t )   =  a \,  x_1 ( X, t )  +  b \quad \mbox{whenever} \, \,  X  \in  \mathbb{S}^{n-1}.
\]   It now follows from  uniqueness of  the  $ \mathcal{A}$-harmonic Green's   function  in  Lemmas 
\ref{lem5.2},    \ref{lem5.3},  and  Remark  \ref{rmk1.3}   for  $  \mathcal{A}$-harmonic functions that     
\[
G_2  ( x )  =  G_1  ( a x + b )\quad \mbox{whenever}\, \,  G_2 ( x )   >  t\, \,  \mbox{and}\,\, t > t_0.
\]   
Next  from  Lemma \ref{lem6.1},    \eqref{eqn2.3},  homothetic invariance of 
$ \mathcal{A}$-harmonic functions, and the same argument as in \eqref{1.43}-\eqref{1.45}  
or   \eqref{eqn5.5}-\eqref{eqn5.8} we see  that  $  G_2 (x)   - G_1 ( a x + b) $  is a solution 
to a locally uniform   elliptic divergence form PDE  with Lipschitz   coefficients.   
This fact and a unique continuation theorem in   \cite{GL} imply as after (6.44) in  \cite{AGHLV}  that the above  
equality  holds whenever  $ x \in \rn{n}  \sem E_2 $  or equivalently 
that $ E_1  = a E_2+ b. $  The proof of Theorem  \ref{theoremA} is now complete.    
\end{proof}  
\part{A Minkowski problem for $\mathcal{A}$-harmonic Green's function}
\setcounter{equation}{0} 
\setcounter{theorem}{0}
\setcounter{section}{6}
\section{Introduction and  statement of  results}
\label{section7}
In  this  section  we  study  the   Minkowski problem  associated with  an  $ \mathcal{A}  =  \nabla f $-harmonic  Green's 
function with pole at  $ \infty $   when    $ f $  is as in Theorem \ref{theoremA}.  To be more specific,  
suppose  $ E  \subset  \rn{n} $  is  a  compact  convex set with nonempty interior.  
Then for  $ \mathcal{H}^{n-1} $  almost every $ x \in \ar E,  $   there is  a  well defined  
outer unit normal, $ \mathbf{g} ( x, E) $  to $ \ar E.  $   The function $ \mathbf{g}(\cdot, E): \ar E  \mapsto  \mathbb{S}^{n-1}$  
(whenever  defined) is called the  Gauss map for $ \ar E.$     Let  $ \mu $  be  a  finite   positive Borel  measure on  $ \mathbb{S}^{n-1} $  satisfying
\begin{align}  
\label{eqn7.1} 
\begin{split}
(i)&\, \,   {\ds \int_{ \mathbb{S}^{n-1}} } | \lan \he, \ze  \ran | \, d \mu ( \ze )  >  0  \quad \mbox{for all} \, \,  \he \in \mathbb{S}^{n-1},\\
(ii)& {\ds  \int_{ \mathbb{S}^{n-1}} } \ze   \, d \mu ( \ze )  = 0. 
\end{split}
\end{align}  
We prove 
\begin{mytheorem}  
\label{mink}  
Let   $  \mu $ be   as in    \eqref{eqn7.1} and  $p$  be fixed,  $ n \leq   p < \infty$.   
Let  $   \mathcal{A} =  \nabla f $  be  as in \eqref{eqn1.8} and Definition \ref{defn1.1}. 
Then there exists a  compact              convex set $ E $ with nonempty interior   
and $  \mathcal{A}$-harmonic Green's function  $u$ for $\mathbb{R}^{n}\setminus E$ with a pole at infinity satisfying  
\beq \bea{l}  
\label{eqn7.2} 
      (a) \hs{.2in}  {\ds  \lim_{y\to x} } \nabla u (y)  =  \nabla u (x)  \, \, \mbox{exists for $\mathcal{H}^{n-1}$-almost every  $ x  \in   \ar E$}  \\ \hs{.43in}  \mbox{ as $ y \in \rn{n} \sem E $   approaches $ x $ non-tangentially.}  \\ \\  (b)  \hs{.2in}  {\ds \int_{\ar E}  f ( \nabla u(x) ) \, d\mathcal{H}^{n-1}  <  \infty}.    
  \\  \\ (c) \hs{.2in}     {\ds \int_{\mathbf{g}^{-1} ( K, E )  }  f ( \nabla  u(x)  )  \, d \mathcal{H} ^{n-1}} =  \mu  (K)  \quad  \mbox{whenever }      K \subset \mathbb{S}^{n-1}\, \,\mbox{is a Borel set}. \\ \\
         (d) \hs{.2in}     \mbox{$E$  is  the unique  set up to translation for which  $ (c) $  holds.} 
\ea \eeq 
     \end{mytheorem}  
We  remark  that   Minkowski originally  considered a  similar  problem   
for surface area measure (in  Theorem \ref{mink} omit $(a), (b), $  and  replace $ f ( \nabla u (x))$  in    $(c) $ by  $1$).   
We also mention that Jerison in \cite{J} proved  a  Minkowski type  theorem  similar to the above when  $ u $ 
 is  the Newtonian capacitary  function of  a  compact convex set.   His result was generalized in 
      \cite{CNSXYZ}  to  $p$-harmonic functions when  $ 1 < p < 2. $  
      In  \cite{AGHLV}, the first,  second, and fourth authors of this article, along with Jasun Gong and Jay Hineman,  
      obtained an analogue  of  Theorem \ref{mink}  when $ 1 < p < n $  for  
    the   $\mathcal{A} = \nabla f$-capacitary function of  a  
    compact convex set $ E$ with nonempty interior. For more historical details see  \cite[section 8]{AGHLV}.   

As a  broad outline of our proof  of Theorem  \ref{mink},  we  follow   \cite{AGHLV}  who 
in  turn  used ideas from \cite{J}  and   \cite{CNSXYZ}.  Part of  the preliminary work  
for the  analogue of  Theorem  \ref{mink}  in  \cite{AGHLV}  involved generalizing   
work from   \cite{LN,LN1,LN2,LN3} for  positive  $p$-harmonic functions  
vanishing on a portion  of  the boundary of  a Lipschitz domain to  positive   
$ \mathcal{A}  =  \nabla f$-harmonic  functions vanishing on a portion of 
the boundary of  a  Lipschitz domain when $  1 < p < n.$    
Much of this work extends  without  change    to the  $  p  \geq n$ and  
  $  \mathcal{A}  =  \nabla f$  situation  so   we shall often refer to  
  Lemmas  in  \cite{AGHLV} for proofs. This  generalization is done in  \S\ref{section8}.  
 From  our  work  in  \S\ref{section8}  it follows that  if $ \ti E $  is  a  compact convex set with non empty interior  and if $ \ti u $ is  the  $\mathcal{A}$-harmonic  Green's function for  $ \ti E $ with pole at $ \infty, $   then  \eqref{eqn7.2} $(a), (b) $ hold with $ u, E $  replaced by  $ \ti u, \ti E. $  The Gauss map and corresponding measure $ \ti \mu $     can then be defined relative to $ \ti u, \ti E $  as  in  \eqref{eqn7.2} $(c)$ 

In  \S\ref{section9}   we consider   a sequence of  compact convex  sets, say 
$ \{\ti E_m\}_{m\geq 1}$  with nonempty  interiors  which converge in the sense of  
Hausdorff  distance  to   $ \ti  E$, a  compact convex set with $0$  in the  interior of $\ti E$.  
For a  fixed $ \mathcal{A} =  \nabla f $ as  in   Theorem  \ref{mink},   
let  $\ti u_m$ for $m=1,2,\ldots,$ and   $ \ti u$  be the $\mathcal{A}$-harmonic Green's  functions 
with pole at $ \infty $  for $\mathbb{R}^{n}\setminus \ti E_m$  and $\rn{n} \sem \ti E$ respectively.  If     
  $\ti \mu_m$ for $m=1,2,\ldots,$ and $\ti \mu$  denote the  Borel  measures  corresponding to  $ \ti u, 
\ti u_m $ as    in the above discussion,   we    show that  
$  \{\ti \mu_m\} $  converges weakly to  $ \ti \mu $  on    $ \mathbb{S}^{n-1}. $     In 
\S\ref{section10}  we  use  this  result  to derive the Hadamard variational formula for
  the  derivative of  $  t  \to \mathcal{ C}_{\mathcal{A}} ( t  E_1  +  (1-t)  E_2 )  $  whenever $  t  \in  [0, 1) $  
  and  $ E_1,  E_2 $  are compact  convex sets with nonempty interiors.   
   Proofs  for  $  p  \geq n $ are  more delicate  than in the  case  $ 1 < p < n $ considered in \cite{AGHLV},  primarily because our  $\mathcal{A}$-harmonic
  Green's function blows up at  $ \infty $   when  $ p  \geq n, $  whereas  $ \mathcal{A}$-harmonic   capacitary  functions
    have limit $0$ at $ \infty $  when $ 1 < p < n.  $  In  \S\ref{section11},   we give  the proof  of  Theorem  \ref{mink}.  
As  in    \cite{AGHLV}   the  proof  essentially consists in showing  that  a  certain   minimum   problem has  a   solution, say  $ \ti  E, $ in the class of compact convex sets with nonempty interior.   To rule out the possibility that  $ \ti E $  has 
  Hausdorff  dimension  $ k \leq n - 1, $    
   we  argue as in   \cite{AGHLV}  when   $ k  < n - 1. $  However if  $ k = n - 1 $ 
  we  are not able to use the same argument as in  the case $ 1 < p < n. $   
  Instead  we  study positive  $  \mathcal{A}$-harmonic  solutions vanishing  
  on  a   ray in  $  \rn{2}$  when  $ p  \geq n  $  and use our results from this study to show   that $ \ti E $ cannot  be  
 $ n - 1 $  dimensional.    Finally, uniqueness in  Theorem  \ref{mink} is proved in  the last part of  section  \ref{section11} using  Theorem \ref{theoremA}.  
  \setcounter{equation}{0} 
 \setcounter{theorem}{0}
    \section{Boundary behavior of $\mathcal{A}$-harmonic functions in  Lipschitz domains}   
    \label{section8} 
    Throughout this and later  sections, the data continues to be  $ p, n,  \al, \La. $ 
We  begin this section with  several definitions and  lemmas  copied from   \cite{AGHLV}. 
Recall that    $ \ph :   K  \to \mathbb R $ is said to be Lipschitz on $ K  $ provided there
exists $  \hat b,  0 < \hat b  < \infty,  $ such that
\begin{align} 
\label{eqn8.1}   
| \ph ( z ) - \ph ( w ) |  \,  \leq \, \hat b   \, | z  - w | \quad  \mbox{whenever}\, \, z, w \in  K.  
\end{align}
The infimum of all  $ \hat b  $ such that  \eqref{eqn8.1} holds is called the
Lipschitz norm of $ \ph $ on $ K, $ denoted $ \| \ph  \hat  \|_{K}$.   
It is well-known that if $ K  \subset \mathbb R^{n-1} $ is compact, then  $ \ph $  has  an  extension to    
$ \rn{n} $   (also denoted $ \ph$)  which is differentiable  almost everywhere in $ \rn{n},  $
and     
\[  
\| \ph  \hat \|_{\mathbb R^{n-1}} = \| \, | \nabla \ph | \,  \|_\infty  \leq c   \| \ph  \hat \|_{K
}. 
\]    
Now suppose that  $ D $ is  an open set, $ w \in \ar D, $  and  
\begin{align}   
\label{eqn8.2}   
\begin{split}
    \partial D\cap B(w, 4  \hat r )&=\{y=(y',y_n)\in\mathbb R^{n} : y_n=
  \phi ( y')\}\cap B(w, 4 \hat r), \\
     D\cap B(w, 4  \hat r )&=\{y=(y',y_n)\in\mathbb R^{n} : y_n > 
  \phi ( y')\}\cap B(w, 4 \hat r)
  \end{split}
\end{align}   
  in an appropriate coordinate system  for some   Lipschitz function $\phi$ on $ \mathbb{R}^{n-1}.$
 Note from elementary geometry  that   if   $ \ze   \in\partial D  \cap  B ( w, 2\hat r) $  and  $0<s<\hat r  $,  we can find points    \[ 
 a_s(\ze )\in D\cap B(\ze ,s)\quad \mbox{with} \quad d(a_s(\ze),\ar D)\geq c^{-1}s
 \]
for a constant $c$ depending on  $ \| \nabla \ph \hat \| $. In the following,    
we let $a_s(\ze )$ denote one such point.   Also if  $ \ze  \in \ar D \cap B ( w, 2 \hat r ), $ and $ t > 1 $ 
 let   \[   
\Ga ( \ze  ) = \Ga ( \ze, t ) = \{ y \in D \cap B( w, 4 \hat r)   :  | y - \ze | <  \, t  \, d ( y, \ar D ) \}. 
\]   Unless otherwise stated  we always assume     $ t $  is fixed and   so large that   $ \Ga ( \ze ) $  contains the  inside of  a  truncated cone with vertex at $ \ze,  $  height  $  \hat r ,$  axis along the positive $ e_n $ axis, and of  angle opening $ \he  = \he ( t ) > 0. $  
 Given a measurable function  $ g $ on $ D \cap B ( w, 4 r ) $, where $0 < r <   \hat r$,  put  $  \De ( w, r  )  = \ar D \cap  B ( w, r)  $  and  define  the \textit{non-tangential maximal
function}
\[
\mathcal{N}_r (g)  :  \De  ( w,  r ) \to \re 
\]
 of $ g $ relative to  $ D \cap B ( w, 4r) $ by  
\[     
\mathcal{N}_r(g)( x ) =  \sup_{y \in \Ga ( x ) \cap  B ( w, 4 r ) }  |g| ( y ) 
\quad  \mbox{whenever}\,\,  x \in \De ( w,  r ).  
\]
Next   we    note  as in       Lemmas   \ref{lemma2.3} and \ref{lemma2.4}: 
\begin{lemma} 
\label{lem8.1}  
Let  $D, w, \hat r, \ph$ be  as  in \eqref{eqn8.2} and     $ 1 <  p  <  \infty$. 
Suppose  $w \in\partial D$, $0< 4 r<\hat r$, and  $ v $ is a  
positive $ \mathcal{A}$-harmonic function in $ D \cap B (w,4r)$ with  $ v \equiv 0 $ on 
$\ar D \cap  B ( w, 4r ) $  in the  $ W^{1,p}  $  Sobolev sense.  Then  $ v  $  
has a representative in  $ W^{1,p} ( D \cap  B ( w,  s )), s < 4r$   which extends to
  a  H\"{o}lder continuous function on   $ B  ( w, s) $ (denoted also by $ v $)  with $  v  \equiv  0 $  on  $  B ( w, s )  \sem D.$      
  Also,
    there exists $ \bar c \geq 1, $ depending only on the data and  $ \| \ph  \hat \|,$  
such that   if  $ \bar r  =  r/\bar c, $ then   
\begin{align}  
\label{eqn8.3}  
\bar  r^{ p - n}   \int\limits_{B ( w, \bar  r)}   | \nabla v |^{ p }  dx  \leq \bar c    (v ( a_{2\bar r} (w)))^{p}.
   \end{align} 
Moreover, there exists $\hat \si \in(0,1),  $ depending only on the data and $\| \ph \hat \|$,  such that 
\begin{align}   
\label{eqn8.4}  
| v ( x ) -  v ( y ) | \leq \bar c \left( \frac{ | x - y |}{\bar r}\right)^{\hat \si}     v   (  a_{2\bar r}  ( w ) ) \quad \mbox{whenever}\, \, x, y \in B ( w, \bar r ).
 \end{align} 
  Finally there  exists a unique finite positive
Borel measure  $ \tau$ on   $ \mathbb{R}^{n}$, with support contained in
$ \bar \Delta(w,r)$, such that
\begin{align}
\label{eqn8.5} 
\begin{split}
&(a)  \hs{.2in}    
{\ds \int  \lan  \nabla f   ( \nabla v  ),   \nabla \psi \ran dx  =  -   \int   \psi \,  d \tau} \quad \mbox{whenever} \, \, \psi  \in C_0^\infty (  B(w,r) ), \\
 &  (b)   \hs{.2in}  
\bar c^{ - 1} \,  \bar r^{ p - n}   \tau ( \Delta (  w,  \bar r ))\leq (v  ( a_{2\bar r} ( w ) ))^{ p - 1}\leq \bar c  \, 
\bar r^{ p - n }   \tau (\Delta (  w, \bar  r )).  
\end{split}
\end{align}
   \end{lemma}      
   Next we  state   as  lemmas   some   results  given  for  starlike Lipschitz domains   in  Lemma 9.5, Proposition 9.7,  Lemma 10.9,  and Corollary 10.10  of \cite{AGHLV}    when  $ 1 < p < n.$    For the definition of  a    starlike Lipschitz domain and justification for using these results in  the present setting,  see  Definition \ref{defn8.4}  and the  discussion following this definition.      
\begin{lemma} 
\label{lem8.2}  
Let  $D, w, \hat r, r, \ph, p,  v, \tau $ be  as  in Lemma  \ref{lem8.1}.     
There exists  $  c_{\star}   \geq 1,  $  depending only on  the data and  $ \| \phi  \hat \| $ 
such that if   $  4 \ti r  =   r/  c_{\star} $ and $ x  \in  B ( w,\ti  r )  \cap   D,   $   then 
\begin{align}
\label{eqn8.6}   
\begin{split}
&(a)  \hs{.2in}  c_{\star}^{-1}  { \ds \frac{v (x)}{ d ( x, \ar D )}  \leq          \lan  \nabla v (x),   e_n \ran   \leq   |  \nabla v (x)  |  \leq    
    c_{\star} \frac{v (x)}{ d (x, \ar D )} } \,   ,\\
&(b) \lim_{\substack{x\to y \\ x \in  \Ga ( y)\cap B(w,2r)}}  \nabla v ( x ) \stackrel{def}{=} 
\nabla v ( y ) \, \, \mbox{exists}  \quad \mbox{for}\, \,  \mathcal{H}^{n-1}\mbox{-almost every}\, \, y \in\Delta( w, \ti r ).
\end{split}
\end{align}
Moreover, $ \De ( w,  \ti r ) $ has a tangent plane
  for  $ \mathcal{H}^{n-1}$-almost every $ y  \in   \De ( w, \ti r )  $.   If
 $ \mathbf{n}( y ) $ denotes the unit normal to this tangent plane pointing
 into
 $ D \cap B ( w, 2 \ti r ), $  then
\begin{align}
\label{eqn8.7}  
\nabla v  ( y )
=  | \nabla v ( y )  | \, \mathbf{n} ( y ) \quad \mbox{for $\mathcal{H}^{n-1}$-almost every} \,\,  y\in\De ( w, 2 \ti  r ) 
\end{align}
and 
\begin{align}   
\label{eqn8.8}  
 \frac{ d \tau }{ d  \mathcal{H}^{n-1}} (y) =   p  \frac{ f ( \nabla v(y))}{|\nabla v (y)|} \quad \mbox{for}\, \, \mathcal{H}^{n-1}\mbox{-almost every}\, \, y\in \Delta(w,2\ti r). 
 \end{align}
\noindent Finally,   there exists    $  q >  p/(p-1)$ and $c_{\star \star}  $  with the same dependence  as  $ c_{\star}  $   such that  
\begin{align}
  \label{eqn8.9} 
  \begin{split}
  &(a)  \hs{.2in}  {\ds   \int_{ \De ( w,     \ti  r )  } \,  
 \left(\frac{ f ( \nabla v )}{|\nabla v |}  \right) ^q \, d\mathcal{H}^{ n - 1 } 
\,   \leq \, c_{\star \star}  \,    r^{ (n -  1)(1 - q) }    \left( \int_{\De ( w,   \ti  r )  
 } \, \frac{ f ( \nabla v )}{ |\nabla v|} \, d\mathcal{H}^{ n - 1} \, \right)^{q}.  
 }       \\
& (b) \hs{,2in}  {\ds   \int_{ \De ( w,     \ti  r )  } \,  
   \mathcal{N}_{\ti r}   ( \nabla v ) ^{q (p-1)} } \, d\mathcal{H}^{ n - 1 } 
\,   \leq \, c_{\star \star}  \,    r^{ (n -  1)(1 - q) }   {\ds  \left( \int_{\De ( w,   \ti  r )  
 } \, \mathcal{N}_{\ti r}  ( \nabla v )^{(p-1)}   \, d\mathcal{H}^{ n - 1} \, \right)^{q}.}  
\end{split}
\end{align}
  
 \end{lemma}            
            
 \begin{lemma} 
\label{lem8.3} 
  Let  $D, w, \hat r, r,  p, $    be  as  in Lemma  \ref{lem8.1}. 
  Also  let $ v_i$, for $i = 1, 2$  be as in this lemma with  $v$  replaced by  $v_i$.    
  There exist  $  \al_+ \in (0, 1)$ and $c_{+}  \geq  1$, depending only on  the data and 
  $ \| \ph \hat \|$,  such that if $r^+  =   r/c^+$ then   
\begin{align}   
\label{eqn8.10} 
\left|   \frac{ v_1 ( x )}{ v_2 (x)}  -  \frac{ v_1 (y )}{ v_2 (y)} 
  \right|  \, \leq \,   c_+    \left( \frac{ | x - y | }{ r^+ }   \right)^{\al_+ } \,  \frac{ v_1 (a_{r^+} (w)  )}{v_2 ( a_{r^+} (w))}
  \end{align}
whenever $x, y   \in  D  \cap B (w, 2r^+)$. 
\end{lemma}  
To  outline the proof  of  these lemmas  we need  a  definition.            
\begin{definition}[\bf Starlike Lipschitz domain]
 \label{defn8.4} 
A bounded domain $ D \subset \rn{n} $   is said to be starlike Lipschitz
 with respect to $ z \in D$ provided  
 \begin{align*}
 \ar D = \{  z + \mathcal{R} ( \om  )  \om &: \om \in \ar B ( 0, 1 ) \} \\
 &\mbox{where}\, \, \log   \mathcal{R} :  \ar B ( 0, 1 )\to\re\,\,  \mbox{is Lipschitz on}\,\, \ar B ( 0, 1 ). 
 \end{align*}
 \end{definition}   
\noindent Under the above scenario  we say   that $ z  $  is the center of  $ D$
  and  $ \| \log \mathcal{R}  \hat \|_{\mathbb{S}^{n-1}}$ is the starlike Lipschitz constant for $ D$.  
      We note that if  $ D, w,  \phi$, and $\hat r$  are as in   \eqref{eqn8.2},  $ 0 < 4r <  \hat r$  and   $   w'  =  w   +  r e_n/2,  $ then there 
     exists   $  c'   \geq 100, $  depending only on  $  \| \ph  \hat   \|$,   such that  the  following is true:  Let $ r' = r/c' $ and  let 
     $   D'  $ denote  the interior of the  set  obtained 
     from the union of  all line segments connecting  points  in $ \ar D  \cap B (w, r' )  $  
     to points in  $  B  ( w',  r' ).  $  Then  $  D' $  is  starlike Lipschitz with respect  to  $ w'$.  
     Moreover, if       $  \mathcal{R}'  $ is the  graph function for  $ D', $  then  
\begin{align} 
\label{eqn8.11}    
\|   \log  \mathcal{R}'   \hat  \|_{\mathbb{S}^{n-1}}   \leq   c'  (  \| \phi \hat \|   +  1   ).  
\end{align}
Also if  $  D'   $  is a starlike Lipschitz domain with center  at  $w'$,  graph function  
     $  \mathcal{R}'$,  $ w \in \ar  D' $,   $   \frac{w' - w}{ |w' - w|} =   e_n,$      and   $  10^{-2}   d ( w', \ar D')    <  s  <   
   10^{-1}     d(w', \ar D' ),   $  then  there exists   $  c''   \geq 1,  $   depending only on  
      $    \|\log \mathcal{R}'  \hat   \|_{\mathbb{S}^{n-1}} $,   such that if $ s'  =  s/c''$  then  
\[    
\ar D'  \cap  B( w, s' )   = \ar D'  \cap  \{ ( y',  \ph' (y')  \} \quad 
  \mbox{and}\quad     D'  \cap  B( w, s' )   = D'  \cap  \{  y  :   y_n   >   \ph' (y')  \}     
  \]  
where   $ \ph' $  is Lipschitz on  $ \rn{n-1} $   and    
\begin{align} 
\label{eqn8.12}  
\| \ph' \hat \| \leq   c''   (  | \log \mathcal{R}'  \hat  \|_{\mathbb{S}^{n-1}} + 1  ).
\end{align}
   
         In    \cite{AGHLV}  results  analogous   to   Lemmas  \ref{lem8.2},  \ref{lem8.3} were first  proved  for   $  1  < p  <  n $  when  $ D $  is  a starlike  Lipschitz  domain.   The results obtained were later  used    as in   \eqref{eqn8.11},  to  prove  similar results in   the  Lipschitz graph  setting  (see Lemma 10.11  in  \cite{AGHLV}).       
         To briefly outline the proofs  given in  \cite{AGHLV} for  $ 1  < p   < n $   and  starlike  Lipschitz  domains, we note that       \eqref{eqn8.6}$(b)$,  \eqref{eqn8.7},  \eqref{eqn8.8},   \eqref{eqn8.9},   were proved in  \cite{AGHLV}  under the assumption  that   
  an analogue of \eqref{eqn8.6} $(a)$  holds  for  $ v. $   This  was  done in Lemmas  9.5, 9.6, and Proposition 9.7.       In section 10  of   \cite{AGHLV}  the authors  used  the  results in  Proposition 9.7   
 to  state and prove   some rather difficult estimates  for  a certain  elliptic measure  in Lemmas 10.1-10.3.  In  Lemmas 10.4,  10.5,  the  authors  define and study the  $ \mathcal{A}$-harmonic  Green's  function,  say $ G' $    for  a  starlike Lipschitz domain,$ D',$  with pole at  
  the center, $w'$,   of this domain.  They  obtain an analogue of   \eqref{eqn8.6} $(a)$  for $ v  = G' $    when   $ 1 <  p  < n. $     These  results are  then used   in Lemma  10.7  to show  that the  ratio of  $ v_1/v_2$,    in the analogue  of   Lemma  \ref{lem8.3} for starlike  Lipschitz  domains,   is at least bounded.      Finally using  a  perturbation         type argument in  Lemma   10.9 and Corollary  10.10   they  eventually get   
 their  version of Lemma  \ref{lem8.3} for $ 1 < p < n $  and  essentially  simultaneously  \eqref{eqn8.6} $(a)$
   for $  v_1,  v_2.  $       
 This result is then restated for  Lipschitz domains  in  Lemma 10.11 of  \cite{AGHLV} for $ 1 < p < n.  $ 

The proof of   Lemmas 10.4,  10.5,  in  \cite{AGHLV}   required   a somewhat  lengthy  study   
  of  the  Green's  function,  which however  was also used (see  Lemma 13.7 in \cite{AGHLV})   to prove the important  Proposition 13.6  of  that  paper.     If $ p > n $  this approach  can no longer be used  to  get  an analogue  of  Proposition  13.6 in \cite{AGHLV}.  
There  are also certain  questions  
    which we do not want to consider when  $ p > n$,  such as  should  $ G' $    
    have a  positive point mass (in which case $ G' \leq 0$),  or a negative point 
    mass at $w'$. To avoid  these deliberations,    one  can replace  $ G' $   in  the  above  proof  scheme  for fixed $ p \geq n, $   by     the $ \mathcal{A}$-harmonic   function $  v' $  in   $  D'  \sem \bar B ( w', s' )$ 
    with  continuous boundary values: $ v' \equiv 1 $ on $ \ar B ( w', s' ) $ and  $ v' \equiv 0 $  on  
$ \ar  D' .$   Here    $  D'  $  is   a   starlike  Lipschitz  domain with  center at $ w' $ and graph function $ \mathcal{R}' .$    Also    $  s'   =  d ( w', \ar  D' )/c' $   where $ c'  \geq 100 $  is fixed.        We prove      

 \begin{lemma} 
 \label{lem8.5}
  Let  $ p,   D',  v',  w',s' $  be as above.   
  There exists  $ c    \geq 1 $  depending only on the data, $ c', $  and $ \| \log \mathcal{R}' \hat \|_{\mathbb{S}^{n-1}} $  such that 
\begin{align} 
\label{eqn8.13}  
\begin{split}
&(a) \hs{.2in}     
     0 < | \nabla v'  ( x ) |    \, \leq \, c \,    \lan {\ts  \frac{w' -
x}{ | w' - x |} } \,,  \, \nabla  v'  ( x ) \ran    
    \quad  \mbox{whenever} \,\,  x \in   D' \sem \bar B ( w', s' ).    \\
&  (b) \hs{.2in}  
c^{-1}  \,    \frac{ v'  ( x )}{d ( x, \ar    D'  ) }\leq    \, 
|   \nabla  v' ( x ) |  \,  \leq  \, c  \frac{v'  ( x )}{d ( x, \ar
  D' )} \quad \mbox{whenever}\,\,   x \in    D'  \sem \bar B ( w', 2 s'). 
\end{split}
 \end{align}  
\end{lemma} 
  \begin{proof}  
 The proof of this lemma is essentially the same as Lemma   10.5   in  \cite{AGHLV}  
 however since we are using a different function we give some details.     
 To start the proof  we assume, as we may,  
 (since $  \mathcal{A}$-harmonic functions are invariant under
translation and dilation  and \eqref{eqn8.13} is also invariant under translation and dilation) 
 that 
\begin{align*}
w' = 0 \quad  \mbox{and} \quad \mbox{diam}(D')= 1.
\end{align*}    
Using Lipschitz  starlikeness   of $ \ar  D', $  the maximum principle for $  \mathcal{A}$-harmonic functions, 
and arguing as in   the proof of   \eqref{1.7}   we find   for some  $ \ti  c  \geq  1 $   and $  \la  >  1 $   near 1, that 
\[ 
\frac{v' (  x ) - v'  ( \la x )}{  \la - 1 }  \geq   \frac{v'(x)}{\ti{c}} \quad \mbox{whenever} \,\, x  \in D' \sem \bar B (0, s')   
\] 
where $ \ti c $   depends  only on  the data, $ c' $  and   $ \| \log \mathcal{R}' \hat \|_{\mathbb{S}^{n-1}}$.  Letting  $ \la \to 1 $   we obtain      
\begin{align}
\label{eqn8.14}   
- \ti c \, \lan \nabla v'  ( x ),  x \ran  \geq  v'  ( x )  \quad \mbox{whenever}\,\,    x  \in    D'   \sem 
\bar B( 0, s' ).
  \end{align}   
      To   get estimates  near  $ \ar  D',  $   let     
\[
        \mathcal{P} ( x )  = -  \lan  \nabla v'  (x), x   \ran \quad \mbox{whenever} \,\,   x \in D' \sem   \bar B (0, s').
\]
As in    \eqref{1.15}       we  note that     
$  \psi  =  v' _{x_i}$ for  $1 \leq i \leq n$,   $\psi  = v'$, and  $  \psi   =  \mathcal{P}$   
      are all  weak solutions in    $ D'  \sem B ( 0, s' )  $  to  
\begin{align} 
\label{eqn8.15}   
\sum_{i, j = 1}^n   \frac{\ar}{ \ar x_i }  ( \hat  b_{ij} \psi_{x_j} )  = 0     
\end{align}
   where  
  \begin{align}  
  \label{eqn8.16}   
  \hat b_{ij} ( x )  =   f_{\eta_i \eta_j} ( \nabla v' (x) ) \quad \mbox{whenever} \quad  x \in   D' \sem \bar B  (0, s').  
\end{align}
We  temporarily assume that $  \mathcal{R}' $  has an extension to $ \rn{n}$ (also denoted   
$\mathcal{R}'$)  with  
\begin{align}  
\label{eqn8.17}  
 \mathcal{R}' \in C^\infty ( \rn{n} ). 
\end{align}
 \noindent Then  from a theorem of Lieberman in \cite{Li}    we deduce that 
 $ \mathcal{P} $ and  $ v'_{x_i}$ for  $1 \leq  i  \leq n$  have   continuous  extensions to 
 the  closure of  $  D' \sem B ( 0,  s' )$.      Using     Lipschitz   starlikeness of  $   D',  $  and  
 \eqref{eqn8.14}  we find for some  $  \breve{c} \geq 1 $   depending only on the data, 
 $ c' $,   and   $ \| \log \mathcal{R}' \hat   \|_{\mathbb{S}^{n-1}}  $  that   
  \[      
  \breve{c}\,  \mathcal{P} (x) \, \geq \pm   v'_{x_i}  (x) \quad \mbox{on}\quad  \ar   D'   \cup  \ar B ( 0,  s' )  
  \]    
  when $ 1 \leq i  \leq n.  $   From  this inequality  and the boundary maximum principle for the  PDE in   \eqref{eqn8.15}, 
  we  conclude that   \eqref{eqn8.13} $(a)$ is valid when  \eqref{eqn8.17} 
  holds with constants  depending only on the data, $ c' , $    and  $ \| \log \mathcal{R}' \hat \|_{\mathbb{S}^{n-1}}  $.   
   As for     \eqref{eqn8.13} $(b)$ the right-hand inequality in this display   
   follows      from  \eqref{eqn2.2} $(\hat a)$.     Thus we prove only the left-hand inequality in 
  \eqref{eqn8.13} $(b)$.    To accomplish  this 
  if   $ x  \in  D' \sem \bar B ( 0,  2 s'), $ 
                we   draw a ray $ l $  from 0  through  $ x $  to  a
point   in $  \ar D'.$  Let $ y  $ be the first point on  $ l $  (starting from $x$)
with    $ v' ( y ) = v' ( x )/2. $ 
  Then from 
 elementary calculus there exists $ \hat w  $ on the part of $ l $
between $ x, y $ with 
\begin{align} 
\label{eqn8.18}   
v' ( x )/2   =   v' ( x ) -  v' (y)    \, \leq  \, | \nabla   v'  ( \hat w ) | \,  | y -  x |.
\end{align}
From    \eqref{eqn8.4}   we deduce the existence of $  c \geq 1 $ depending only on the data,
$c' , $  
and  $  \| \log \mathcal{R}' \hat \|_{\mathbb{S}^{n-1}}  $ with 
\begin{align}  
\label{eqn8.19}   
y, \hat w \in B [ x,  (  1 - c^{ - 1} )  d ( x, \ar D') ]. 
\end{align}    
Using \eqref{eqn8.19},      Harnack's  inequality for $ \mathcal{P},  $  and  $  \eqref{eqn8.13} (a),$   it follows 
 for some $ c, $  depending only on  the data, $  c',  $  and  $ \| \log \mathcal{R}' \hat \|_{\mathbb{S}^{n-1}},  $ that   
 \[ 
 | \nabla v'    ( \hat w  ) | \, \leq \, c \, | \nabla v'   ( x ) |.
 \]     
This  inequality,     \eqref{eqn8.18}, and \eqref{eqn8.19} imply that  
 \[   
 v' ( x ) \, \leq   \, c \, | \nabla v' (x  )  | \,  d ( x,  \ar    D'   ).  
\]
   We conclude that the   left-hand inequality   in  \eqref{eqn8.13} $(b)$ is valid for   
   $ x  \in   D' \sem \bar B ( 0,  2s' )  $    when  $ c $ 
   is suitably large and      \eqref{eqn8.17}  holds.  To remove this assumption choose 
   $  \mathcal{R}_m  \in C^\infty (  \rn{n} )   $     for $ m = 1, 2, \dots, $ with 
\[ 
\|  \log   \mathcal{R}_m  \hat \|_{\mathbb{S}^{n-1}}     \leq c    \| \log   \mathcal{R}' \hat \|_{\mathbb{S}^{n-1}} 
 \] 
 and $  \mathcal{R}_m \to  \mathcal{R}' $ as $ m \to \infty $ uniformly on $  \mathbb{S}^{n-1}$.   Here $ c $ depends only on $n.$ 
  Let $ (D_m)  $ be the    corresponding sequence of  starlike  Lipschitz domains  and for large $ m,  $    
      let   $ v_m $ be  the  $\mathcal{A}$-harmonic function in $ D_m \sem \bar B (0, s') $ with $ v_m   \equiv    
   1 $  on $ \ar B (0, s') $     and $ v_m \equiv 0 $ on $ \ar D_m.$     Using   Lemmas  \ref{lemma2.2} and \ref{lem8.1},  we see that 
\begin{align*}
\{v_m, \nabla v_m\}& \, \, \mbox{converge to}\, \, \{v', \nabla v'\} \\
&\mbox{ uniformly on compact subsets of}\, \,  D' \sem \bar B(0,s'). 
\end{align*} 
Applying  Lemma \ref{lem8.5} to  $ v_m $,  and using the fact that 
 the constants in this lemma 
 are independent of 
 $ m $    we conclude  upon taking limits  that Lemma \ref{lem8.5}   holds for $ v' $  without hypothesis  \eqref{eqn8.17}.  \end{proof}

  Other than substituting    Lemma   \ref{lem8.5} for    Lemmas  10.4,  10.5  in  
  \cite{AGHLV},    the  proof  in \cite{AGHLV},   outlined  above,   can  also  be  used if  $p \geq n  $ to  prove  Lemmas  \ref{lem8.2},  \ref{lem8.3}.   Thus  we  omit further details in the proofs of these lemmas.      
 
 For use in  later sections    we make the following definition:        
 \begin{definition}[\bf  Lipschitz domain]
 \label{defn8.6}
 A   domain $D \subset\mathbb{R}^{n}$  is   called a bounded Lipschitz domain provided   there exists a finite set of    
balls $\{B(x_i,r_i)\}$   covering  an 
open neighborhood of $\partial D,$  such that 
 \eqref{eqn8.2} holds  with $r_i = \hat r $ and $ \ph_i  = \ph. $  
        The Lipschitz constant of $ D$ is defined 
  to be $M=\max_i\||\nabla\phi_i|\|_\infty$.  
  \end{definition} 

   \setcounter{equation}{0} 
   \setcounter{theorem}{0}
\section{Weak convergence of  certain  measures on  $\mathbb{S}^{n-1}$}    
\label{section9}       
Let   $   E $ be a  compact convex set with $ 0 $ in the interior of    $ E$, $p  \geq n,  $  
and  let $ u $   be  the   $ \mathcal{A}$-harmonic Green's  function for $ \mathbb{R}^{n}\setminus E$ with pole at infinity.  
We note that $   \mathbb{R}^{n}\setminus E $ is  a Lipschitz domain   so Lemma  \ref{lem8.2} holds with $ v =  u,  
D = \rn{n}  \sem E $ provided $ \hat r > 0 $  is small enough. From Lemma  \ref{lem8.2} we see that if   
$\mathbf{g}(x, E)=\mathbf{g}: \partial E\to \mathbb{S}^{n-1}$ is  defined by
\[
\mathbf{g}( x, E) = -\frac{\nabla u (x)}{ |\nabla u  (x) |}
\]
then this equality is well-defined  on a set  $ \He    \subset   \ar E   $ with   $  \mathcal{H}^{n-1} ( \ar E \sem \He  ) =  0. $   Also from   Lemma \ref{lem8.2} 
we  see that  if   $  F  \subset  \mathbb{S}^{n-1}$  is  a  Borel  set,  then  $\mathbf{g}^{-1} (F, E) $ is 
 $ \mathcal{H}^{n-1} $  measurable.    Define  a measure  
$ \mu(\cdot)= \mu_{E,f} ( \cdot )   $  on  $  \mathbb{S}^{n-1}$  by    
\begin{align} 
\label{eqn9.1}  
\mu ( F )  :=   \int_{ \He   \cap \mathbf{g}^{-1} ( F, E)}  f ( \nabla u ) \,  d \mathcal{H}^{n-1}  \quad \mbox{whenever}\, \, F  \subset  \mathbb{S}^{n-1}\, \, \mbox{is a Borel set}. 
\end{align}   
Next suppose that $ \{E_m \}_{m\geq 1}$        is  a  sequence of  compact 
convex  sets with nonempty interiors  which converge to  $  E $ 
 in  the  sense of  Hausdorff distance. That is,  $ d_{\mathcal{H}} ( E_m, E ) \to 0 $ as 
 $ m \to \infty $ where   $d_{\mathcal{H}}$  was defined at  the beginning of section \ref{NSR}.  
 Let     $ u_m  $  be   the  corresponding  $ \mathcal{A} = \nabla f$-harmonic Green's  
 function for $\mathbb{R}^{n}\setminus E_m$ when $m=1,2, \dots,$ with pole at infinity.   
 Then for $ m $ large enough say  $ m  \geq m_0 $  we see that Lemma  \ref{lem8.2} is valid with  
 $ D_m  = \rn{n} \sem E_m$, $u_m = v$, and $ w \in   \ar D_m$. Here  $ \hat r  > 0 $ can be chosen independent of   
 $m$   and  constants depend only on the data as well as the Lipschitz constant  for   $ E. $   
 For fixed $  m,  $  let  $  \mu_m  $  be the measure on  $  \mathbb{S}^{n-1}$   
 defined  as in   \eqref{eqn9.1} relative  to  $ f$,  $u_m$, and $\mathbf{g} (  \cdot, E_m)$.       
  We prove  
  \begin{proposition}
  \label{proposition9.1}  
   Fix   $  p \geq n  $    and define  $  \mu_m$ and $\mu $  relative to  $ E_m$ and $E $ respectively as above. 
   Then
   \[  
   \mu_m \rightharpoonup  \mu \quad \mbox{weakly as} \,\,   m \to \infty.
   \] 
   \end{proposition}  
   \begin{proof}    
  The proof of this proposition is given in  \cite[Proposition 11.1]{AGHLV}  for  $  1  <  p   < n.$ The proof  just uses  
  Lemmas \ref{lem8.1}-\ref{lem8.3},  although in one place  (see the proof of  claim 11.9 in \cite{AGHLV}) 
  the authors  use  $\mathcal{A}$-capacitary functions   to construct  an   $ \mathcal{A}$-harmonic function
   $ v $   in    $  K=K (\hat \he )  =  \{ x :  x_n  > |x| \cos (\hat \he)  \}.$ Here $v$ is   continuous in  $ \rn{n} $  
   with  $ v \equiv  0 $ on  $ \rn{n} \sem K $  for  $ \hat \he > \pi/2, $  but near $ \pi/2$ satisfying $ v (e_n)  = 1. $     
    To get $ v $ in our situation let  
  $ v_m$   be   $  \mathcal{A}$-harmonic function  in   $  B ( 0,  2m  ) \sem  [ \bar B ( 0, m ) \sem  K]  $
  and  continuous on  $ \bar B  (0, 2m ) $   with  $  v  \equiv 0  $ on  $ \bar B (0, m ) \sem K $ 
  and  $v \equiv  a_m  $  on  $  \ar  B  (0, 2m) $ where $ a_m$  is chosen so that  $ v_m ( e_n ) = 1. $     
  Existence of $ v_m $   follows  from uniform $ (2m, p)$-fatness of $  \ar  B (0, 2m)$  and  
  $ \bar B (0, m )  \sem K. $  Using   Harnack's inequality, Lemmas  \ref{lem8.1} and  \ref{lemma2.2},     
  and  Ascoli's theorem we see that a  subsequence of  $ \{v_m\} $  converges to $ v $  with the desired properties  as $ m \to \infty. $    
  For the rest of the proof  of   Proposition  \ref{proposition9.1}  see \cite[section 11]{AGHLV}. 
  \end{proof}    
\setcounter{equation}{0} 
   \setcounter{theorem}{0}
\section{The Hadamard variational formula for $\mathcal{A}$-harmonic PDEs}  
\label{section10}
Let $ E_1$ and $E_2$ be   compact convex sets  and suppose  $ 0 $ is in the interior of  
$ E_1 \cap E_2. $  Fix $ p  \geq n $  and let   $ u ( \cdot, t ) $   
be the $ \mathcal{A} = \nabla f$-harmonic Green's  function  
for  $  \mathbb{R}^{n}\setminus (E_1  + t  E_2) $  with pole at infinity when $ t  \geq 0$.  
Also  let   $  \mu_{E_1 + t E_2} $ be the measure defined  on  $ \mathbb{S}^{n-1}$   in   \eqref{eqn9.1}  relative to   $ u (  \cdot,  t )$.   
In this section  we  prove 
\begin{proposition} 
\label{proposition9.2}  
With  the  above  notation let  $  h_1$ and  $h_2 $   be  the  support   functions for $ E_1$ and $E_2, $  
respectively and let  $\mathbf{g}( \cdot,  E_1 + t E_2 ) $  be the  Gauss map for $  \ar  (  E_1 + t E_2 )$.  
Then   for  $  t   \geq  0 $ and $ p = n$  ($\gamma$ is as in Definition \ref{defn5.4})  we have   
 \begin{align} 
 \label{eqn9.2} 
 \begin{split}   
 \frac{d \,\mathcal{C}_{\mathcal{A}} (E_1 + t E_2 )  }{dt} = n  \ga^{-1}  \,   \mathcal{C}_{\mathcal{A}}(E_1 + t E_2 )     
      \, {\ds  \int_{  \ar  ( E_1 + t  E_2) } }  h_2  ( \mathbf{g}(x,  E_1 + t E_2 ) )   f( \nabla u  ( x,  t ) ) \, d  \mathcal{H}^{n-1}.
\end{split}
\end{align}      
While for $n<p$ we have  
\begin{align}
 \label{eqn9.3}  
\begin{split} 
   \frac{d \,\mathcal{C}_{\mathcal{A}} (E_1 + t E_2 )  }{dt}  =        p (p-1)   \mathcal{C}_{\mathcal{A}}(E_1 + t E_2 )^{\frac{p-2}{p-1}}  
 {\ds  \int_{  \ar  ( E_1 + t  E_2) } }  h_2  ( \mathbf{g}(x,  E_1 + t E_2 ) )   f( \nabla u  ( x,  t ) ) \, d  \mathcal{H}^{n-1}.
      \end{split}
      \end{align} 
\end{proposition}  
\begin{proof}  
To lay the groundwork for the  proof of  Proposition  \ref{proposition9.2}  we first argue as in  (12.2)-(12.15) of   \cite{AGHLV}.  
We  begin by assuming for  $  i = 1, 2 $  that 
\begin{align} 
\label{eqn9.4} 
\begin{split}
\mbox{$ \ar E_i $ is  locally the graph of}&\mbox{ an infinitely differentiable} \\
& \mbox{and strictly convex  function on  $  \mathbb{R}^{n-1}. $ } 
\end{split}
\end{align}   
We  note from Lemma  \ref{lemma2.2}   that   $ u ( \cdot,  t_i) $, for $i=1,2$,   
has  H\"{o}lder  continuous  second  partials  in    $\mathbb{R}^n \sem (E_1 + t_2 E_2).   $  
Let   
\[     
 \zeta (x, t_1)  =   \frac{ u ( x, t_1 )  -  u ( x, t_2)}{ t_2 - t_1} \quad \mbox{whenever} \,\, x  \in  \mathbb{R}^n \, \,\mbox{and}\,\,  0 <  t_1    <  t_2. 
 \]
Then from  Lemmas  \ref{lem5.2}, \ref{lem5.3},   and the maximum principle for $ \mathcal{A}$-harmonic functions 
we see that $     \ze  \geq  0.$  
   Moreover, from Lemma \ref{lem6.1} we find  as in  \eqref{1.43}-\eqref{1.45}  that  $ \zeta$  
 is  a   weak  solution to    
\begin{align}  
\label{eqn9.5}  
 \sum_{i, j = 1}^n  \frac{ \ar}{ \ar x_i} (\bar{d}_{ij}  \ze_{x_j} ) 
= 0   
\end{align}  
in    $ \mathbb{R}^n  \sem  ( E_1 + t_2  E_2)  $ where  
\begin{align*}   
\bar{d}_{ij} (x)  =  \int_0^1 f_{\eta_i \eta_j} ( s  \nabla u ( x, t_1 ) + (1-s) \nabla u ( x, t_2 ) ) ds. 
\end{align*} 
Also,   
\begin{align}  
\label{eqn9.6}   
c^{-1} \bar \si (x)  \,   | \xi |^2  \, \leq \,   \sum_{i,j=1}^n   \bar d_{ij} ( x  )  \xi_i  \xi_j   \leq c \, \bar \si (x) \, |\xi|^2  
\end{align}
whenever $\xi  \in \mathbb{R}^n \sem \{0\}$ and $ x   \in  \mathbb{R}^n  \sem  ( E_1 + t_2  E_2)$ with
\begin{align}
\label{eqn9.7} 
\bar \si ( x  )  \approx  ( |  \nabla  u (x, t_2 ) | + |  \nabla  u (x, t_1 ) |  )^{p-2}.
\end{align}
The constants  in \eqref{eqn9.6} and  \eqref{eqn9.7} may   depend  on  $ t_2 $  and  the  radius  
of the largest ball contained in $ E_2 $  but  are  independent of  $x$  as above  and  $  t_1 $ 
whenever  $ t_1 \in [t_2/2, t_2). $          
  Also  from  Lemma   \ref{lemma2.2},   \eqref{eqn8.6},  Lemma \ref{lem6.1}  and  
  the  theorem  in \cite[Theorem 1]{Li}  mentioned earlier,  
  we see that   $ \nabla  u (\cdot, t_i  )$ for $i = 1, 2, $      extend   to   H\"{o}lder continuous  
  functions in the closure of  $  \mathbb{R}^n \sem  (E  _1+ t_i E_2)$.  More  specifically, if  
  \[
   t_2/2 \leq t_1 <  t_2 \quad \mbox{and} \quad \rho  =   4 (t_2 +1 ) \left(\mbox{diam}(E_1) +  \mbox{diam} (E_2) \right)
   \]
    then there  exist $ \be \in (0,1)$ and  $C^{\star}  \geq 1, $ independent of $ t_1, $  such that for $ i = 1, 2, $ 
\begin{align}
\label{eqn9.8} 
\begin{split}
&(a)  \hs{.2in} \, | \nabla  u ( x, t_i  ) - \nabla  u ( y,  t_i ) | \, \leq \,
      C^{\star}     | x - y |^{ \be },
      \\  
  &    (b) \hs{.2in} (C^{\star})^{-1}  \leq    | \nabla u (x, t_i )  |  \leq C^{\star}  
\end{split}
      \end{align}
 whenever $x,y$ are  in the closure of   $B ( 0, \rho )  \sem ( E_ 1 + t_i E_2 )$. 
  We point out that  the lower bound  in    \eqref{eqn9.8} $(b)$  follows from  a contradiction type argument using  Ascoli's   theorem and   Lemma  \ref{lem6.1}.    
   From \eqref{eqn9.8} $(b)$ and the  mean value theorem from  calculus we see that   there exists $C^{\star\star}$ independent of $ t_1  \in  [t_2/2, t_2 )  $  and  $ x $ such that
\begin{align}   
\label{eqn9.9}   
  0   \leq    \ze   \leq C^{\star\star} 
\end{align}  on      $ \ar  (E_1 + t_2  E_2 ).$
  From  \eqref{eqn9.9},  Lemmas  \ref{lem5.2},  \ref{lem5.3}, and the  maximum principle for  $ \mathcal{A}$-harmonic functions,   we  deduce for   given  $ \ep  > 0 $  that there exists  $ s = s ( \ep)  > \rho $ such that  
\[  \ze  \leq  C^{**}  +  \ep   u ( \cdot, t_1 )  \mbox{ in }   B ( 0,  s ( \ep)) \sem (E_1 + t_2 E_2).  \]  Letting $ \ep \rar 0 $  it follows that  \eqref{eqn9.9}  is  valid  in  the closure of  $ \rn{n} \sem (E_1 + t_2 E_2). $  In a similar way we get  that   
$ M ( \cdot,  \ze )$ is decreasing and  $ m ( \cdot, \ze ) $ is increasing on  $ ( \rho, \infty). $  Using this fact and arguing as   in  \eqref{6.0} $(b), $   we find  that  if  $ u ( \cdot, t_i ) =   F  + k_i ,$ for  $ i  = 1, 2, $  then there exists     $ \hat \he > 0$     such that   
\begin{align}
\label{eqn9.10}  
\left|  \ze (x )  -   \ze (\infty) \right|  =   \left|  \ze (x )  -   \frac{ k_1 (\infty)  - k_2 (\infty)}{t_2- t_1} \right|  \leq  2 C^{\star\star} \left( \frac{\rho}{|x|}\right) ^{ \hat \he}  \mbox{ for  }  | x | \geq  \rho. 
\end{align}
     In this discussion,   $F$ is the Fundamental solution as in Lemma \ref{lem1.4} when $p=n$ and as in Lemma \ref{lem1.5} 
   when $n<p<\infty$. Also, $k_i$ is the function associated to $E_1+t_i E_2$ for $i=1,2$ as in 
   Lemma \ref{lem5.2} when $p=n$ and as in Lemma \ref{lem5.3} when $n<p<\infty$. Moreover, 
   the constant    $\hat \he $  is  independent of  $ t_1  \in  [t_2/2, t_2 )  $  and  $ x $ as above.      
   
   From \eqref{eqn2.3}  and  Lemmas      \ref{lem6.1} we see that   $  \bar d_{ij} (x) $ in \eqref{eqn9.5} are Lipschitz    continuous in     $ \rn{n}  \sem B ( 0, \rho) $  
   with  Lipschitz norm independent of  $  t_1   \in  [t_2/2,  t_2). $   Using this  fact,  elliptic 
   PDE theory  (see  \cite{GT})  and  \eqref{eqn9.10}  we see that  
   $ \nabla \ze $ is   locally  H\"{o}lder  continuous in  $ \rn{n}  \sem B (0, \rho  ) $  and 
\begin{align} 
\label{eqn9.11a}  
 |\nabla   \ze  ( x ) |  \leq   \ti C  | x |^{-1 - \hat \he } \quad \mbox{whenever}\, \,   |x| \geq \rho 
 \end{align}  
 where $ \ti C $ has the same dependence as $ C^{\star\star}$.  From    \eqref{eqn9.8} $(a)$,   
 Lemma  \ref{lem6.1},  and   uniqueness of  $ u  ( \cdot, t_i )$ for $i  = 1, 2,  $         we deduce that   
\begin{align}
\label{eqn9.11} 
\nabla u  ( \cdot, t_1 ) \to  \nabla u ( \cdot, t_2 ) \, \, \mbox{uniformly on  the  closure of}\, \,  \mathbb{R}^n \sem (E_1 + t_2 E_2)\, \,  \mbox{as}\, \, t_1 \to t_2.   
\end{align}  
From  \eqref{eqn9.9}-\eqref{eqn9.11}, Lemma \ref{lem6.1},   \eqref{eqn9.5}-\eqref{eqn9.7},  
 and  Caccioppoli type estimates for locally uniformly elliptic  PDE,  
we deduce that if  $ t_1 \to  t_2  $  through  an increasing   sequence,  say $\{s_m\}$,   
then a  subsequence of $\{ \zeta =  \ze ( \cdot, s_m) \}$, also denoted   by $ \{\ze ( \cdot, s_m) \}$,   converges uniformly  on   compact subsets of   
$ ( \mathbb{R}^n   \cup   \{ \infty \}  )\sem ( E_1 + t_2 E_2 )    $  to  a locally   H\"{o}lder $  \hat \be $  continuous    function,   say  $ \breve \ze.  $  
 Moreover,  this  subsequence also converges      to   $  
\breve \ze $ locally  weakly in  $  W^{1,2} $ of   $ \mathbb{R}^n   \sem ( E_1 + t_2 E_2 ). $ Finally,     
\begin{align}
\label{eqn9.12}   
\sum_{i, j = 1}^n  \frac{ \ar}{ \ar x_i} ( f_{\eta_i \eta_j} ( \nabla u ( x, t_2 ) ) \,  \breve \ze_{x_j} (x))  
= 0
\end{align}    
locally in the weak sense in  $ \mathbb{R}^n   \sem ( E_1 + t_2 E_2 ) $
 and   \eqref{eqn9.10}, \eqref{eqn9.11a} are valid   with   $ \ze $ replaced by  $ \breve \ze$.   

Next  we  show that  $ \breve  \ze $  has boundary values that are  independent of  the choice of  sequence.   To do this, for $k = 1, 2$ we let      
$x_k  ( Z) =     \nabla  h_k  (  Z  ) $   whenever  $ Z  \in  \mathbb{S}^{n-1}$   and  recall that   
$  x_k (Z) \in \ar E_k $  with   
             $ Z = \mathbf{g}( x_k (Z), E_k)  $  for $ k = 1, 2$ .
We fix  $ X, Y   \in  \mathbb{S}^{n-1}, $    write  $ x, y $  for   $ x_1 ( X ) + t_2  x_2 (X),   x_1 (Y)  +  t_2   x_2 (Y)$  respectively and note that   
\[
x  =  \mathbf{g}^{-1} ( X,  E_1  + t_2 E_2 )\in   \ar ( E_1 + t_2 E_2) \quad \mbox{and}\quad  y   =  \mathbf{g}^{-1} ( Y,  E_1  + t_2 E_2 )\in   \ar ( E_1 + t_2 E_2).
\]   
 We  consider two cases.  First if  
\[
|  x  - y |  \leq   d ( x,  E_1 +  t_1 E_2 )/2
\]  
then from  \eqref{eqn9.8} $(a)$ and the mean value theorem of   calculus    we have    
\begin{align}
\label{eqn9.13}   | \ze (x)  -   \ze (y)  -  \lan \nabla  \ze (x),  x -  y \ran  |  \leq   \hat C
| x - y |^\be  
\end{align} 
where  $ \hat C $ is independent of  $ t_1. $   Second,   if   
\[
  | x - y | >  d ( x,  E_1 + t_1 E_2   )/2
  \]   
  then   using $ u ( \cdot, t_1  )  \equiv 0 $ on 
 $ \ar ( E_1 + t_1 E_2) $  and the  same strategy as  above we see that 
\begin{align}   
\label{eqn9.14} 
\begin{split}
  | \ze (x)  -    \lan \nabla  & u ( x_1 ( X ) + t_1 x_2 (X),  t_1 ),   x_2  (X) \ran  |  \\
   &+  | \ze (y)  -     \lan \nabla  u (x_1 (Y ) +  t_1  x_2 (Y), t_1 ),  y_2 (Y) \ran  |  \\
   & \hs{.8in} \leq   \hat  C | x - y |^\be.  
\end{split}
\end{align}   
Now   $ h_1  + t_1 h_2 $ is the support  function  for  $  E_1 + t_1  E_2 $  and so     
\begin{align}
\label{eqn9.15} 
\begin{split}
 \lan  \nabla u ( x_1 ( X ) + t_1 (x_2 (X)),  t_1 ),   x_2  (X) \ran &=   | \nabla u ( x_1 ( X ) + t_1 (x_2 (X)),  t_1 ) | \lan X,  x_2 ( X ) \ran \\
&=    | \nabla u ( x,  t_1 ) | h_2 ( \mathbf{g} (x, E_1 + t_2 E_2 ) )   +  \la (x)  \\
&=III_1+ \la (x)
\end{split}
\end{align}  
where   $  |\la (x) | \leq  \bar C   | x - y |^\be $  and  $  \bar C $ is independent of   $ t_1. $   Similarly,  
\begin{align}  
\label{eqn9.16} 
\begin{split}
 \lan \nabla  u ( x_1 ( Y ) + t_1 x_2 (Y),  t_1 ),   x_2  (Y) \ran &=   | \nabla u ( y,  t_1 ) |   h_2 ( \mathbf{g} (y,  E_1   + t_2  E_2  )  )  +  \bar \la (y) \\
 &=III_2+\bar \la (y) 
\end{split}
\end{align} 
where $ \bar \la (y) $  satisfies the same inequality  as  $  \la (x). $ 
From  \eqref{eqn9.14}-\eqref{eqn9.16} and the triangle  inequality we find that  
\begin{align} 
\label{eqn9.17} 
\begin{split}
 | \ze ( x ) -  \ze (y) | &  \leq  \left| III_1-III_2\right| +|\la ( x ) | + |\bar \la (y) | \leq  \ti C | x  - y |^{\be}
\end{split}
 \end{align}  
 where $  \ti C $ is independent of $ t_1 \in  [t_2/2, t_2)$.  Here we have also  used  Lipschitzness of  $ h_2 $  and \eqref{eqn9.8} $(a)$ to estimate  $ | III_1-III_2| .$
  From  \eqref{eqn9.17},  we deduce that     $ \ze  = \ze (\cdot, t_1)  $ is  H\"{o}lder  $  \be$-continuous  on  $ \ar ( E_1 + t_2 E_2)$ with  H\"{o}lder norm    
  bounded  by  a  constant independent of  $ t_1 \in   [t_2/2, t_2 ). $  From well-known theorems for  
  divergence form  uniformly elliptic  PDE  with bounded measurable coefficients  
   it now follows   that   $ \ze $ is  H\"{o}lder  $ \tau  $ continuous   
   on the closure of   $ B (0,  \rho)  \sem  ( E_1 + t_2 E_2) $  for some  $ \tau  > 0$    with  H\"{o}lder norm   
   independent of    $ t_1 \in [t_2/2, t_2) $     
  This fact and  the same reasoning as in   \eqref{eqn9.11a} yield       
 \begin{align}
   \label{eqn9.17a}  
   |\nabla \ze (x) |    \leq C  d ( x, \ar (E_1 + t_2E_2) )^{-1  + \tau}.
   \end{align}   
Taking limits we conclude from  Ascoli's  theorem  that  $ \breve  \ze $ is  
the uniform limit of    $ ( \ze ( \cdot,  s_m ) ) $ in  the closure of   $ B (0, \rho )  \sem (E_1  +   t_2 E_2) $.  
Thus $ \breve  \ze $ is also  H\"{o}lder  $ \tau$-continuous  in the closure of    $ B (0, \rho )  \sem (E_1  +   t_2 E_2)$ 
and  \eqref{eqn9.17a} holds for $ \breve{\ze}. $  Finally,   using \eqref{eqn9.8} and  arguing as in \eqref{eqn9.15} we see that  
\[  
\left|   \ze (x, t_1) -       | \nabla u ( x,  t_2 ) |\,  h_2 ( \mathbf{g} (x,  E_1   + t_2  E_2  ) ) \right|   \leq   c   |  t_1  -  t_2 |^{\be}. 
\] 
From this estimate we conclude that 
 \begin{align}  
\label{eqn9.18}   
 \ze (x, t_1)  \to     | \nabla u ( x,  t_2 ) | \, h_2 ( \mathbf{g} (x,  E_1   + t_2  E_2  ) ) \quad  \mbox{as}\quad   t_1   \to   t_2 
  \end{align}
 whenever  $  x  \in  \ar  (E_1 + t_2 E_2)$.  From \eqref{eqn9.18} we see that  
 every convergent subsequence of   $  \{ \ze ( \cdot, t_1 ) \} $ converges to  a  weak solution of      
\eqref{eqn9.12} satisfying   \eqref{eqn9.17a}  with continuous boundary values   
\[ 
\left.  |  \nabla  u  ( \cdot, t_2  ) \right|  h_2 (\mathbf{g} ( \cdot,  E_1  + t_2  E_2 )  )  \quad   \mbox{on}\quad   \ar ( E_1 + t_2 E_2 ).
  \]    
   
To  begin the  proof  of  Proposition \ref{proposition9.2}  in the smooth 
case and when $ t_2/2  \leq t_1 <  
t_2 $  recall  from the display above  \eqref{eqn9.8},   that   $  E_1 +  t_2 E_2  \subset   B ( 0, \rho )$. 
If $ R > 4 \rho $ we apply the divergence theorem to    $    u  ( \cdot, t_i )  (\nabla f ) ( \nabla u  ( \cdot, t_i ) $  for $ i = 1, 2$ 
and use $ \mathcal{A}$-harmonicity  of $ u ( \cdot, t_i)$,   $p$-homogeneity of $ f $,  
and smoothness  of  $ u ( \cdot, t_i )$ on $\ar (E_1 + t_i E_2), i = 1, 2 $ to get 
\begin{align}
\label{eqn9.19}
\begin{split}   
I_i (R)  &=    p  \int_{B(0, R) \sem (E_1 + t_i E_2 )}  f ( \nabla u (x, t_i ) )  dx     \\
&   =       \int_{\ar B(0, R) } u ( x, t_i )  \lan ( \nabla    f ) ( \nabla u ( x, t_i ) ), x/|x| \ran \, d \mathcal{H}^{n-1}. 
\end{split}
\end{align}
For brevity  we write  for fixed $ t_2, $ 
\[  
\frac{ I_1 (R) -  I_2 (R)}{t_2 - t_1}  =   \int_{\ar B(0, R)} J  (  \cdot, t_1 ) d \mathcal{H}^{n-1}. 
\] 
To make calculations first observe that 
 \[     
\lan \frac{(\nabla f) (  \nabla  u ( x, t_1 ) )  -    (\nabla  f) ( \nabla u ( x,  t_2 )  )}{t_2-t_1}, x/|x| \ran  =   \sum_{ l, j = 1}^n   \bar d_{lj} ( x )    \ze _{x_j} (x)   ( x_l/|x|)   
 \]
where  $ (\bar d_{lj} ), 1 \leq l, j \leq n, $  is as  in  \eqref{eqn9.5}. 
Letting $ t_1 \to  t_2$ in $(s_m)$   we see that   
\begin{align}  
\label{eqn9.20}  
  \lan \frac{(\nabla f) (  \nabla  u ( x, t_1 ) )  -    (\nabla  f) ( \nabla u ( x,  t_2 )  )}{t_2-t_1}, x/|x| \ran  \to     \sum_{ l, j = 1}^n  \bar a_{lj} ( x )   \breve  \ze _{x_j} (x)   ( x_l/|x|)   \end{align}
 where  $  \bar  a_{lj} (x )   =    \frac{ \ar^2  f }{\ar \eta_l \ar \eta_j }   \left(   \nabla u (x, t_2)   \right). $ 
 Next  we use  the same algebra as in  the calculus  argument  for finding the derivative of  a product.  
 After that we note from  \eqref{eqn9.20} and  \eqref{eqn9.11a}   that  
\begin{align}  
\label{eqn9.21}
J  (x, t_1 )  \to  J_1 (x)  =    \breve \ze (x)           \lan (\nabla f ) ( \nabla u ( x, t_2 ) ),  x/|x| \ran  +  u ( x, t_2) \sum_{l, j = 1}^n  \bar a_{lj} (x)  
   \breve \ze_{x_j} (x)  x_l/|x|
\end{align}   
as $ t_1  \to  t_2 $ in $(s_m)$. 

Using   \eqref{6.0} $(a)$ and $(b)$  with  $ k_2  = k $ and arguing   in  a now well known  way   
we deduce that  $ k_2 $  is  a  solution in  $ B (x, |x|/2) $ to  a  
uniformly elliptic PDE in divergence form with Lipschitz continuous coefficients when 
$|x| > 4 R_0 $.   From \eqref{6.0}$(b)$   and   elliptic PDE  theory  it  follows  that   
\begin{align}  
\label{eqn9.22} 
| \nabla  k_2 (x) |  \leq  C  |x|^{ - 1  -  \ti \he} \quad \mbox{for}\, \, | x| = R
\end{align}
 where  once again $ C $ is independent of $ x $. 
        Using  \eqref{eqn9.22}, \eqref{eqn1.8} $(i)$,  and  our  knowledge of  $ F $    it follows that  
\begin{align}   
\label{eqn9.23}  
u ( x,  t_2 )   \bar a_{lj} (x )   =  F ( x )   \frac{ \ar^2  f }{\ar \eta_l \ar \eta_j }   \left(  \nabla F (x)  \right) +  o\left(   \de (R)  R^{\frac{(1-n)(p-2) + (p-n) }{p-1}} \right)  \quad \mbox{as}\, \,          R \to \infty  
\end{align}        
         where  the  $ o $ term is independent of  $ t_1 \in  [t_2/2, t_2). $   Here  
         $ \de  (R)  = 1 $ if  $ p  > n $  and  $ \de (R)  = \log R, $ for  $ p = n. $  Likewise, as $R\to\infty$, 
\begin{align}  
\label{eqn9.24}       
 \lan (\nabla f ) ( \nabla u ( x, t_2 ) ),  x/|x| \ran   =  \lan (\nabla f ) ( \nabla F ( x) ),  x/|x| \ran  + o(R^{1-n} ).
\end{align} 
Using     \eqref{eqn9.11a},  \eqref{eqn9.23},   \eqref{eqn9.24},  and   $  \breve \ze (x)  =   \breve \ze (\infty)   +  o(1)   $   as $ R  \to \infty$ in    
             \eqref{eqn9.21}  we conclude first that  
\begin{align}  
\label{eqn9.25}
   J_1 (x)  =   \, \breve \ze (\infty)    \lan (\nabla f ) ( \nabla F ( x ) ),  x/|x| \ran   \, + \,   o( R^{1-n}  ) 
\end{align}
       and  second  since $F$ is a fundamental solution   that  
\begin{align}  
\label{eqn9.26}  
\begin{split} 
\lim_{m\to  \infty }  \frac{  I_1 ( R ) -  I_2 ( R )}{ t_2  -  t_1}  &=    \, \breve \ze (\infty)   \int_{\ar B (0, R)}  \lan (\nabla f ) ( \nabla F ( x ) ),  x/|x| \ran  d \mathcal{H}^{n-1} 
     \, + \,   o(1) \\
     &=   \breve \ze (\infty)  + o(1).
\end{split} 
\end{align}           
Next    we  write  
\begin{align} 
\label{eqn9.30}
     \frac{  I_1 ( R ) -  I_2 ( R )}{ t_2  -  t_1}     =        T_1 ( R  )  +   T_2  (R) 
\end{align}     
where 
\[  
T_1  (R)    :=p  ( t_2 - t_1)^{-1} {\ds  \int_{B(0, R)  \sem ( E_1 + t_2 E_2 ) } } (  f (  \nabla  u  ( x,  t_1 ) )  -   f (  \nabla  u  ( x,  t_2 ) ) )  dx
\]
and
\[
T_2 (R) := p  ( t_2 - t_1)^{-1} {\ds  \int_{  (E_1 + t_2 E_2)  \sem ( E_1 + t_1 E_2) } }  ( f (  \nabla  u  ( x,  t_1 ) )  dx.
\]
From   \eqref{eqn9.11} and \eqref{eqn9.18} we have
\begin{align} 
 \label{eqn9.31}  
 \begin{split} 
T_2 (R)    &=   \, ( t_2 - t_1)^{-1} {\ds  \int_{  (E_1 + t_2 E_2) 
  \sem ( E_1 + t_1 E_2) } }  \nabla \cdot  [  (u  ( x,  t_1 )  )   \, \nabla f (  \nabla  u  ( x,  t_1 ) ) ]  dx  \\
  &  =      {\ds  \int_{  \ar  ( E_1 + t_2 E_2) }   \ze (x)  | \nabla u ( x, t_2 )  |^{-1}   
\lan \nabla   f( \nabla u  ( x,  t_1 ) ),   \, \nabla u  ( x,  t_2 )  \ran \, d \mathcal{H}^{n-1}   } \\
   & \to  \hat T_2:=     p {\ds  \int_{  \ar  ( E_1 + t_2 E_2) } }  h_2  ( \mathbf{g}(x,  E_1 + t_2 E_2 ) )   f( \nabla u  ( x,  t_2 ) ) \, d  \mathcal{H}^{n-1}       \quad \mbox{as} \, \, t_1 \to t_2\, \, \mbox{in} \,\, (s_m).
\end{split}
\end{align}    

	  To handle  $ T_1 (R)  $  observe  as  in \eqref{eqn9.20} that   
\begin{align}   
\label{eqn9.32} 
   p \, \frac{f (  \nabla  u ( x, t_1 ) )  -      f ( \nabla u ( x,  t_2 )  )} {t_2-t_1}   =   \sum_{  j = 1}^n  q_{j} ( x )    \ze _{x_j} (x) 
\end{align}
where
         \[
  q_{j} (x )   =  p  \int_0^1       \frac{ \ar  f }{\ar  \eta_j }   \left( s  \nabla  u (x, t_1)   +  (1-s)  \nabla u (x, t_2) \right)  d s.
\]   
Using   \eqref{eqn9.32}     in  the definition of  $  T_1 (R)  $  and letting $ t_1 \to t_2 $ in $(s_m)$  
we obtain  from  \eqref{eqn9.17a}   and  the  Lebesgue dominated convergence theorem that      
\begin{align*}
 T_1 (R)  \to   \hat T_1 (R)   &=  p 
{\ds \sum_{j=1}^n   \int_{B(0, R)  \sem ( E_1 + t_2 E_2 ) } }  \frac{\ar f }{\ar \eta_j} ( \nabla u ( x, t_2 ) )  \, \breve \zeta_{x_j} (x) dx  \\ 
&  =  \frac{p}{p-1} {\ds \sum_{l,j=1}^n   \int_{B(0, R)  \sem ( E_1 + t_2 E_2 ) } }  \bar a_{lj} (x)   u_{x_l}(x, t_2)  \breve \zeta_{x_j} (x) dx
\end{align*}
 where $\bar a_{lj} (x)= \frac{ \ar^2  f }{\ar \eta_l \ar \eta_j }   \left(   \nabla u (x, t_2)   \right)$  as earlier and 
 we have used the $(p-1)$-homogeneity of $\nabla f$. 
 From   \eqref{eqn9.17a}   we  see that  the above  integral  can be integrated by  parts to obtain,  
   \[  
   \hat T_1 (R)   = \frac{ p }{p-1}  \sum   \int_{\ar B(0, R)  } u(x, t_2)   \sum_{l, j  = 1}^n  \bar a_{lj} (x)   \breve \zeta_{x_j} (x)  x_l/|x| \,    d  \mathcal{H}^{n-1}
  \]   
  where we have also used \eqref{eqn9.12}.   Letting  $ R  \to \infty  $  we deduce as in the derivation  of  \eqref{eqn9.26}  that      
  $  \lim_{R\to  \infty}   \hat  T_1 (R)  =   0.  $   Combining    this equality,  \eqref{eqn9.31}    
  and  letting  $  R \to \infty $  in   \eqref{eqn9.26}     we  arrive at  
\begin{align} 
\label{eqn9.33}  
\breve \ze (\infty)  =   p  {\ds  \int_{  \ar  ( E_1 + t_2 E_2) } }  h_2  ( \mathbf{g}(x,  E_1 + t_2 E_2 ) )   f( \nabla u  ( x,  t_2 ) ) \, d  \mathcal{H}^{n-1}.    
\end{align}    
We  note that    \eqref{eqn9.33}, the remark after   \eqref{eqn9.18},  
and  the  usual maximum principle  argument imply  that $  \breve \ze $ is independent of the choice of  $ (s_m )$.    Thus we  put      
 \[ 
 \breve \zeta (\infty) = \lim_{t_1 \rar t_2^-  }\, \ze ( \infty, t_1) = -  \frac{d k}{dt_2} (\infty)   
 \] 
 where  $ - k  ( \cdot,  t  )  =    F ( \cdot)  -  u ( \cdot, t ) $  for    $ t \in  [0, \infty). $     
  Now  from  Theorem \ref{theoremA}  we  see  that   
  $ t   \mapsto  \mathcal{C}_{\mathcal{A}} (E_1 + t E_2 )    $  is concave on
    $[ 0, \infty)$ when $ p  = n $  and    $ t \mapsto \mathcal{C}^{\frac{1}{p - n }}_{\mathcal{A}} (E_1 + t E_2 )   $  is concave on $ [0, \infty) $ when $ p >n. $  
    Thus      $t   \mapsto     \mathcal{C}_{\mathcal{A}}  (E_1  + t E_2 )     $   is Lipschitz and differentiable off  a  countable set when $n\leq p<\infty$.  
    From this observation, \eqref{eqn9.33},   and the   chain rule 
  we have   when $p =   n$   
\begin{align} 
\label{eqn9.34}  
\begin{split}
&  \frac{d \,\mathcal{C}_{\mathcal{A}} (E_1 + t E_2 )  }{dt}  =    
      -  \ga^{-1}  \,   \mathcal{C}_{\mathcal{A}}(E_1 + t E_2 )   \, \frac{d k}{dt}  (\infty) \\
      & \hs{.4in}  =  n \ga^{-1}  \,   \mathcal{C}_{\mathcal{A}}(E_1 + t E_2 )     
      \, {\ds  \int_{  \ar  ( E_1 + t  E_2) } }  h_2  ( \mathbf{g}(x,  E_1 + t E_2 ) )   f( \nabla u  ( x,  t ) ) \, d  \mathcal{H}^{n-1}
\end{split}
\end{align}      
and  for $p>n$ 
\begin{align} 
\label{eqn9.34a}    
\begin{split}
&\frac{d \,\mathcal{C}_{\mathcal{A}} (E_1 + t E_2 )  }{dt} =    
         - (p-1)  \,   \mathcal{C}_{\mathcal{A}}(E_1 + t E_2 )^{\frac{p-2}{p-1}}    \, \frac{d k}{dt}  (\infty) \\ 
         &  \, \, = p (p-1)  \,   \mathcal{C}_{\mathcal{A}}(E_1 + t E_2 )^{\frac{p-2}{p-1}}    \, \,     
      \, {\ds  \int_{  \ar  ( E_1 + t  E_2) } }  h_2  ( \mathbf{g}(x,  E_1 + t E_2 ) )   f( \nabla u  ( x,  t ) ) \, d  \mathcal{H}^{n-1}
\end{split}
\end{align}     
except for at most $t$ in a countable set.    Now from properties of  
support functions and  \eqref{eqn9.11}   we see  that   the     right-hand side  in 
both  equalities  is  continuous as  a  function of  $ t  $  so    from  the usual calculus 
argument   $  \frac{d \,\mathcal{C}_{\mathcal{A}} (E_1 + t E_2 )  }{dt} $  exists on
 $ (0, \infty) $ and has a left-hand derivative at $ t = 0. $ Thus Proposition  \ref{proposition9.2}   is valid under  assumption       \eqref{eqn9.4}.       

To show this  assumption can be removed,   choose  sequences of uniformly bounded
 convex domains  $\{E_1^{(l)}\}_{l\geq 1}$ and $\{E_2^{(l)}\}_{l \geq 1}$  with
   $ E_i \subset  E_i^{(l)} $  for $ i = 1, 2 $  and $l=1,2,\ldots,$ satisfying   \eqref{eqn9.4}   
   with $ \ar  E_i $  replaced by  $  \ar E_i^{(l)}, i = 1, 2$ and $l=1,2,\ldots$   
   We also choose these sequences so  that  $  E_i^{(l)} $  converges to  $ E_i $ in the sense of Hausdorff distance as  $ l  \to \infty$.     
 Let $ u^{(l)} ( \cdot, t ) $  
be the  $\mathcal{A}$-harmonic Green's  function  for the complement of $ E_1^{(l)} + t  E^{(l)}_2 $  with pole at $ \infty $. We  set
  $ -  k^{(l)} ( \cdot, t ) = F ( \cdot ) - u^{(l)} ( \cdot, t ) $  while  as earlier  $ u ( \cdot, t )$ and $k  ( \cdot, t )  $  have  the  same  meaning relative to  $ E_1  + t E_2. $      We claim   that 
\begin{align}  
\label{eqn9.35} 
\lim_{l \to \infty}  k^{(l)} ( \infty, t) = k ( \infty, t )  \quad   \mbox{for}\, \,   t  \in  [0, \infty). 
\end{align}  
     This    claim  will be proved in more generality in  \eqref{eqn10.0}  
 so  we  reserve its proof  until  the next section.  
Let   \[   
   \Ph (t) =  \int_{\ar (E_1 + t  E_2)}   h_2 ( \mathbf{g}(x,  E_1 + t E_2) )\,    f(\nabla u (x, t ) )  d\mathcal{H}^{n-1}   
\]  and  let  $  \Ph_l ( t ) $ denote the function  
 in this display with  $ E_i, $ replaced by $ E_i^{(l)},   i = 1, 2, l =1,2, \dots.$   Given  $ 0 < a  < \infty $ we claim   there exist  $ m =  m (a)$ and $M =  M (a)$  such that   for  $ l  \geq  m, $  we have  
 \begin{align}   
 \label{eqn9.36}   
 0 <  \Ph_l ( t )   \leq   M   \quad \mbox{for} \quad  t  \in  [0,a].  
 \end{align} 
 To verify this  assertion   fix  $l,  t,  $  let    
 \[
 E_0   =  E_1^{(l)}  + t  E^{(l)}_2
 \]
   and let  $ h_0$ be the support function for $E_0$ and let  $\mathbf{g}(\cdot, E_0)$  be the Gauss map for $\partial E_0$ while   $u_0$ is  the $  \mathcal{A}  =   \nabla  f$-harmonic Green's function for $\mathbb{R}^{n}\setminus E_0$ with pole at infinity.   Applying  Proposition \ref{proposition9.2}  in this  case  with $ E_1,  E_2,  t,  $  replaced by  $ E_0,  E_0, 0, $  and  using the fact  that  
\[    
\mathcal{C}_{\mathcal{A}} ( (1+t)  E_0)    = 
\begin{cases}
  (1+t) \mathcal{C}_{\mathcal{A}} (  E_0)  & \mbox{when} \, \, p=n, \\
   (1+t)^{p - n}     \mathcal{C}_{\mathcal{A}} (  E_0)  & \mbox{when} \, \, n<p<\infty
\end{cases}
\] 
 we obtain     when $ p = n$         
\begin{align}
 \label{eqn9.37}  
 \ga      = n \int_{\ar E_0 }  
 h_0 (\mathbf{g}(x,E_0))  \,   f ( \nabla u_0(x)) \, d\mathcal{H}^{n-1}.  
 \end{align} 
While when  $ p > n $  we have
\begin{align}
 \label{eqn9.38}  
 \frac{p-n}{p-1}\,   \mathcal{C}_{\mathcal{A}} (  E_0)^{1/(p-1)}      = p   \int_{\ar E_0 }  
 h_0 (\mathbf{g}(x,E_0))  \,   f ( \nabla u_0(x)) \, d\mathcal{H}^{n-1}.  
\end{align} 
Since  $  E_0 $  is  uniformly bounded and  $ h_0  \geq \min_{\mathbb{S}^{n-1}}  h_1  >  0, $  where $ h_1 $  is the support function for $ E_1, $ 
it follows from \eqref{eqn9.37},  \eqref{eqn9.38},    and properties of $\mathcal{C}_{\mathcal{A}} (  \cdot ) $ that    \eqref{eqn9.36}  is  true.          
  From  \eqref{eqn9.35}, \eqref{eqn9.36},  Proposition \ref{proposition9.1}, Proposition 
\ref{proposition9.2} in the smooth case,  and    Lipschitzness  of  support functions,   
as well as   the  Lebesgue dominated  convergence  theorem,  we  find that 
\begin{align}
  \label{eqn9.39}  
  \begin{split}
     k  ( \infty, t ) -  k  ( \infty, 0 )        & =   \lim_{l\to \infty}   [ k^{(l)}  ( \infty, t ) -  k^{(l)}  ( \infty, 0 ) ]  \\
&=- \lim_{l \to \infty}    p \int_0^t \Ph_l (s) ds= -  p  \int_0^t \Ph (s) ds.
  \end{split}
\end{align}
     Also $\Ph $ is  continuous on $ [0, \infty)$   by   Proposition  \ref{proposition9.1}.  
     Now \eqref{eqn9.39}, the  definition of  $ \mathcal{C}_{\mathcal{A}} (  \cdot),  $     and  the Lebesgue differentiation theorem yield  Proposition 
\ref{proposition9.2} without assumption  \eqref{eqn9.4}.   
\end{proof}
    \begin{remark}  
\label{remark9.3}  
We note   that  Proposition \ref{proposition9.2} remains valid   for  $ t > 0 $  if    we  assume only  that  $0 \in E_1,  $  rather than  
$0$  is in the interior of  $ E_1 $  (so  $ \mathcal{H}^n(E_1)=0 $  is possible but from
 the definition of $E_2$ we still have $0$ in the interior of $E_2$).  
 To handle  this case  we    put  $ E'_1 = E_1 + t E_2$ and 
 $ E'_2   = E_2 .$   Then    $ E_1',  E_2' $  are compact convex  sets and  $0$ is in the interior of  $ E_1'  \cap E_2' .$       
 Applying  Proposition \ref{proposition9.2}  with  $ E_1, E_2 $  replaced by  $  E_1',  E_2'  $  
 respectively     and   at  $ t=0 $  we obtain  the above generalization  of     
Proposition \ref{proposition9.2}.
\end{remark}   
\setcounter{equation}{0} 
   \setcounter{theorem}{0}
   \section{Proof of  Theorem \ref{mink}} 
       \label{section11}      
Finally we are in  a  position to prove existence of the measure in Theorem \ref{mink} in the discrete case. 
Once again our argument is similar to the one in  \cite[section 13]{AGHLV} for $ 1 < p < n $.  However, in  our  opinion  
not  all of  this argument is    so  well  known  and  involves  different calculations for $ p \geq n$,  
so we  will  give mostly full details.     
To  begin   we  note that  if  $ p  \geq n $  is  fixed,  $ E' $  is  a  compact convex set 
and  $ ( E'_l ) $  is  a sequence  of  convex compact sets  converging to $ E'$ in the sense of  Hausdorff distance,   
then  either  $ E'  $  is  a  single  point  in which case  $ \lim_{l \to \infty}   \mathcal{C}_{\mathcal{A}} (E'_l )  = 0 $  or  
\begin{align} 
\label{eqn10.0}  
\lim_{l \to \infty}   \mathcal{C}_{\mathcal{A}} (E'_l )  =   \mathcal{C}_{\mathcal{A}} (E' )  > 0.  
\end{align} Note  that    \eqref{eqn10.0} and  the definition of     
$   \mathcal{C}_{\mathcal{A}} ( \cdot ) $ give claim \eqref{eqn9.35}. 
If  $ E' $  is  a  single point it follows from \eqref{5.23}   that the above  limit is  zero.   
   Otherwise,  let  $G'_l$, $l=1,2,\ldots,$ be the corresponding  sequence of 
   $\mathcal{A}$-harmonic Green's  functions  for  $ \mathbb{R}^{n}\setminus E_l'$ for $l=1,2,\ldots,$ 
   and let $ G' $  be the $\mathcal{A}$-harmonic Green's  functions  for  $ \mathbb{R}^{n}\setminus E'$ with pole at infinity.    
   Then   from  Lemma  \ref{lemma2.3}  and  uniqueness of  $ G' $   we  see  that   $ (G'_l)$    converges  
   uniformly on  compact subsets of  $ \rn{n} $  to  $ G'.$   
   From   Lemma \ref{lem6.1} we see that     $  \He  =  \{ x :\,  G'(x)  \leq  1 \} $ 
    is convex and has nonempty interior.  Also    $ G' - 1$ is the 
    $ \mathcal{A}$-harmonic  Green's   function for $ \mathbb{R}^{n}\setminus \He$ with pole at infinity.  
    Translating and  dilating  $ \He $ if necessary  we     may assume     that     $ 0 \in \He $  
    and $\mbox{diam}(\He)= 1 $ thanks to    \eqref{5.23}.    
    Applying  \eqref{6.0} $(a)$  to $ G' - 1 $   it follows as in   the  proof  of  \eqref{6.0} $(b),$  
    that there exists  $ R > 0  $  independent  of  $ l $  so  that  
\[   
| G_l' - G'| (x)  \leq   M (R, | G_l' - G'|) \quad \mbox{and} \quad 0  <  \min_{ \ar B (0, R ) }  ( F - G' )   \leq   (F  -  G') (x) 
\]    
when  $ x \in \rn{n} \sem B ( 0, R). $  Using  these  inequalities and the 
definition of   $ \mathcal{C}_{\mathcal{A}} (\cdot ) $ in  Definition  \ref{defn5.4}  we obtain   \eqref{eqn10.0}.  
   \subsection{Proof of  existence in    Theorem \ref{mink} in the  discrete case}  
Let   $  c_1,  c_2,  \dots,  c_m  $  be positive numbers  and  $   \xi_i \in  
\mathbb{S}^{n-1}$  for  $   1 \leq  i   \leq  m$.     
Assume that  $ \xi_i \not = \xi_j$ for $i \not = j$  and  let  $ \de_{\xi_i} $  denote the  measure with  a point mass at  $ \xi_i. $  
Let  $ \mu $  be a measure on   $ \mathbb{S}^{n-1}$  with    
\[   
\mu ( K )  = \sum_{i=1}^m   c_i   \de_{\xi_i} (K) \quad \mbox{whenever} \quad K\subset \mathbb{S}^{n-1} \, \,  \mbox{is a Borel set}.   
\]   
We also  assume that  $ \mu $  satisfies  \eqref{eqn7.1} $ (i)$ and $(ii). $   That is,  
\begin{align} 
\label{eqn10.1}   
\sum_{i=1}^m   c_i  \,  |  \lan \he,  \xi_i \ran |   >   0  \quad \mbox{for all}  \quad \he \in  \mathbb{S}^{n-1}     
\end{align}
and
\begin{align}
\label{eqn10.2} 
\sum_{i=1}^m   c_i  \,  \xi_i  =  0.
\end{align}    
For technical   reasons  we  also first assume that    
\begin{align}
\label{eqn10.3}  
 \mbox{either}\, \, \, \mu  ( \{\xi\} ) \, \, \, \mbox{or} \, \, \, \mu (\{- \xi\} )= 0  \quad \mbox{whenever} \quad \xi \in \mathbb{S}^{n-1}.
 \end{align}   
 This assumption will be removed in the  general  proof  of  existence.   
 For   $  \mu $  as above and fixed  $ p \geq n, $   we    show that   there is  a   
 compact   convex  polyhedron  $E$  with  $0$  in the interior of  $ E $   and  
\[
\mu (K)  = \int_{\mathbf{g}^{-1}(K, E)}  f(\nabla U ) \, d\mathcal{H}^{n-1} \quad \mbox{whenever} \quad K\subset\mathbb{S}^{n-1} \, \, \mbox{is a Borel set}
\]
where   $ \mathbf{g}(\cdot, E)$  is the  Gauss map for  $ \ar E $  and    $U$ is 
the $ \mathcal{A}$-harmonic Green's  function for  $  \rn{n}\sem E$ with pole at infinity.  
 Moreover,  if  $  F_i $  denotes the  face of  $ \ar E $ with outer normal  
 $ \xi_i$ for $1 \leq i  \leq  m$  then  $\mathbf{g}(F_i, E) = \xi_i  $  and   
\begin{align}  
\label{eqn10.4}     
\mu ( \{\xi_i\} )  = c_i   =  \int_{F_i} f(\nabla U ) \, d\mathcal{H}^{n-1}  \quad \mbox{for}\,\,  1 \leq i   \leq m.
\end{align}  
Existence  in   \eqref{eqn10.4}  follows  from  a    variational  type  
argument  apparently due to  Minkowski (see \cite[Section 8.2]{G})).
    Using this  method  one   needs  to  show  that  a  certain  minimum  
    problem  has  a  solution within the  given class of  compact convex  sets with nonempty interior.  
     To  be more specific       let 
$ q=(q_1,\ldots, q_m)\in \rn{m} $ with  $ q_i   \geq 0$ for $1 \leq i  \leq m. $ Let  
\[  
E(q):=  \bigcap_{i=1}^m  \{x:\,  \lan x,  \xi_i \ran  \leq q_i  \}    \quad \mbox{and}\quad  \Ph:=   \{E (q):\,   \mathcal{C}_{\mathcal{A}}  (E(q)) \geq  1 \}.
 \]
We also set     
\[ 
\he (q) =  \sum_{i=1}^m  c_i  \, q_i \quad \mbox{and}\quad     \la =     \inf \{ \he (q):\,\,  E (q)  \in  \Ph \}. 
\]    
We want to show there exists  
\begin{align} 
\label{eqn10.5}    
\breve q  =  (\breve q_1, \dots, \breve q_m),\, \,  \breve  q_i  > 0 \, \, \mbox{for}\, \, 1 \leq i  \leq m, \, \,  \mbox{with}\, \,  \he ( \breve q ) = \la\,\,  \mbox{and} \,\,  \mathcal{C}_{\mathcal{A}} (   E (\breve q  )  ) \geq  1.
\end{align} 
Using    \eqref{eqn10.5}   and   a  variational type argument,    it  follows  easily     
from  Proposition  \ref{proposition9.2}   that   $  \mu $  is  a  constant multiple of  
the  measure  defined   in   \eqref{eqn7.2} $ (c)$   relative  to  the 
$ \mathcal{A}$-harmonic  Green's  function for    $  \mathbb{R}^{n}\setminus E (\breve q)$ with a pole at infinity.    
    
 To begin the proof of \eqref{eqn10.5}   we first note  that  if  $ E (q)\in \Ph$  
 then  $ E(q) $  is   a closed convex set.   Also  we note from   \eqref{eqn10.2} that 
 \[   
 \int_{ \mathbb{S}^{n-1}}   \lan \tau, \xi \ran^+  d \mu ( \xi ) =  \int_{ \mathbb{S}^{n-1}}   \lan \tau, \xi \ran^-  d \mu ( \xi )  \quad \mbox{whenever}\quad    \tau \in \mathbb{S}^{n-1}
  \]   
  where $ a^+ =  \max{(a, 0) } $  and  $ a^-   = \max{(- a, 0)}$.   From this note and  \eqref{eqn10.1} 
 we see that   for some $ \ph >  0, $ 
 \begin{align}  
 \label{eqn10.6}
  \ph  <    \int_{ \mathbb{S}^{n-1}}   \lan \tau, \xi \ran^+  d \mu ( \xi )   \quad \mbox{for all} \quad  \tau  \in \mathbb{S}^{n-1}.
 \end{align}
If    $ h = h  (  \cdot,  E (q) ) $  is the support function for $ E (q) $  and    $ r \tau  \in  E (q ) $ 
for some $r>0$ and $\tau\in\mathbb{S}^{n-1}$  then    it  follows from \eqref{eqn10.6}  and the definition of  $ h $  that  
\begin{align} 
\label{eqn10.7}  
\begin{split}
0  <  \ph  r     \leq  { \ds  \int_{\mathbb{S}^{n-1}} \,  \lan  r \tau,   \xi   \ran^+   \,  d \mu(\xi)  \leq   \int_{\mathbb{S}^{n-1}} h ( \xi )  d \mu ( \xi ) } \leq    \he (q). 
\end{split}
\end{align}
From \eqref{eqn10.7}  we first see that  $E(q)   \subset   \{ x:\,\,  |x| \leq  \he (q) /\ph \}$. 
We then conclude the existence of       $ q^l  = (q_1^l,  \dots, q_m^l ),  q_i^l \geq 0$ for $1 \leq i \leq m$  such that  $  E_l = 
E (q^l)$, $l = 3, 4,  \dots,  $ is  a  sequence of  uniformly bounded compact convex sets in the class   $ \Ph $   with    
\[
\hat  q  =  \lim_{l\to \infty} q^l \quad \mbox{and}\quad  \lim_{l \to \infty}  \he ( q^l )  =  \la =  \he ( \hat q ). 
\]   
From   finiteness of  $ \la $   we also may assume  that       $   E_l \to  E (\hat q )=  E_1$, 
a compact convex set  containing $0$,  where convergence is uniform in the Hausdorff distance sense. From  \eqref{eqn10.0}  we see  that   
\begin{align}
\label{eqn10.8} 
\lim_{l \to \infty}   \mathcal{C}_{\mathcal{A}} (E_l ) =  \mathcal{C}_{\mathcal{A}} (E_1 ). 
\end{align} 
Thus   $ \mathcal{C}_\mathcal{A}(  E_1 )  \geq  1$  and in    fact    
$ \mathcal{C}_{\mathcal{A}} (   E_1 )  = 1$.   Otherwise, we  would  have 
$ \he (\ti q) < \he(\hat q)$  for  $  \ti E  = \ti E (\ti q) \in \Ph $   where  for 
$ j \in \{1,2, \dots, m \}, $     
\begin{align*}
\ti q_j  =
\begin{cases}
 \frac{\, {\ds \hat q_j}}{{\ds\mathcal{C}_{\mathcal{A}} (  E_1 )}}  & \mbox{when}\, \, p=n, \\
  \frac{\, {\ds \hat q_j }}{{\ds \mathcal{C}_{\mathcal{A}} (  E_1 )^{1/(p-n)}}} & \mbox{when}\, \, p   >  n.  
    \end{cases}   
\end{align*}


If $ E_1 $  has nonempty interior,  say    $ z $ is an  interior point of $  E_1,  $   
then  $  \breve E   =    E_1  -  z   \in   \Ph $  and   $  \mathcal{C}_{\mathcal{A}} (\breve  E ) = 1  $  
as we deduce  from translation invariance of  $ \mathcal{C}_{\mathcal{A}} (  \cdot ) .$    
Moreover,   if    $  \breve  E  =  E (\breve  q ) $       then from \eqref{eqn10.2}  and    $  \he (\hat q)  =  \la$  
 we see that
  $ \he (\breve q ) = \la.      $   Thus,      \eqref{eqn10.5}  is  valid if $  E_1 $  has nonempty interior.  
      
If $   E_1 $  has empty interior,  then  from convexity of  $  E_1 $  and \eqref{eqn10.3} we find that 
      $    E_1 $ is  contained in an  $l$-dimensional plane with $l< n - 1 $  and  $  0  < \mathcal{H}^l (E_1) 
<  \infty. $   Also  $E_1$  must  contain  at least two points  since            
  $  \mathcal{C}_{\mathcal{A}} (  E_1 )  = 1 $  so $ l \geq 1$.   
  We assume,  as  we  may,  that  $0$ is an interior point of  $  E_1 $   
  relative to the  $ l$-dimensional plane containing  $   E_1 $  since otherwise we  
  consider $  E_1 - z $ for some $ z $ having this property and  argue as above.      
  From the definition of  $ \Ph $   we see that  there exists  a subset, say  $  \La $  of  $ \{ 1, \dots, m \}$,  with $ \hat q_i = 0 $ when 
$ i  \in \La. $   Since  $    \hat q  $  gives    a minimum for  $ \he $    we observe  that  if  $  s    \not \in  \La, $  then  $ \hat q_s \not = 0$   and  
\[    
\{ x:\,\,  \lan x, \xi_s  \ran = \hat q_s  \}  \cap   E_1   \neq  \es.
\]   
Let     $ a   =   \frac{1}{4}    \min \{  \hat q_i  :   i   \not  \in  \La  \} $  and  for   small $ t > 0 $   let   
\begin{align} 
\label{eqn10.9}  
\begin{split}
\ti E  (t)  =   {\ds \bigcap_{i = 1}^m  \{ x:\,\, \lan x,  \xi_i \ran    \leq  \hat q_i + a t  \} } \quad \mbox{and} \quad  E_2   =  {\ds \bigcap_{i = 1 }^m }  \{ x: \,\,\lan x,  \xi_i \ran  \leq   a   \}.
\end{split}
\end{align}  
 Put  
\begin{align}
\label{eqn10.10}   
E_t     =       \frac{ \ti  E  ( t )  }{ \mathcal{C}_{\mathcal{A}} ( \ti E  (t) )}\, \,  \mbox{when $ p  = n$}\quad \mbox{and}\quad E_t =   \frac{ \ti  E  ( t )  }{ \mathcal{C}_{\mathcal{A}} ( \ti E  (t) )^{ 1/(p-n)}} \, \,
\mbox{when  $ p  > n.$}    
\end{align}
 We note that, in view of \eqref{eqn10.10},  $  E_t  =   E ( q (t) )  $  where $q(t)=(q_1(t), \ldots, q_m(t))$ and for $j=1,2,\ldots, m$ 
\begin{align}    
\label {eqn10.11}     
q_j (t) =   \frac{\hat q_j +  a t}{  \mathcal{C}_{\mathcal{A}} ( \ti  E (t) )} \, \, \mbox{when $p=n$}\quad \mbox{and}\quad  q_j (t) =   \frac{\hat q_j +  a t}{  \mathcal
{C}_{\mathcal{A}} ( \ti  E (t) )^{1/(p - n)}}\, \, \mbox{when $ p > n$.} 
\end{align}   
From \eqref{5.23} we have $  \mathcal{C}_{\mathcal{A}}  (E_t ) = 1 $  so  $ E_t \in \Ph. $ 
To get a  contradiction to our assumption that $   E_1 $ has empty interior  we show that  
\begin{align}  
\label{eqn10.12}  
\he ( q (t) )  <   \la  \quad \mbox{for some  small}\, \,  t > 0. 
\end{align} 
 To prove \eqref{eqn10.12},   we first note that     
 $   E_1   +  t   E_2   \subset    \ti E (t )  $  for  $ t \in (0, 1) $ so    
 \[    
 \mathcal{C}_{\mathcal{A}} ( E_1 + t E_2  ) \leq    \mathcal{C}_{\mathcal{A}} ( \ti  E (t ) ). 
\]      
From this  inequality and    \eqref{eqn10.10} and \eqref{eqn10.11},    we conclude  that  to prove     \eqref{eqn10.12}  it suffices  to show  if     
\begin{align} 
\label{eqn10.13} 
\begin{split}
 \chi  ( t )  =
 \begin{cases}
 [\mathcal{C}_{\mathcal{A}} ( E_1  +  t  E_2  )]^{- 1}  {\ds  \sum_{i = 1 }^m  }   c_i  ( \hat q_i  + at )  & \mbox{when}\, \,   p = n, \\
[\mathcal{C}_{\mathcal{A}} ( E_1  +  t  E_2  )]^{- 1/(p -n)}   { \ds  \sum_{i = 1 }^m  }   c_i  ( \hat q_i  + at ) & \mbox{when} \, \, p > n
 \end{cases}
\end{split}
\end{align} 
then  
\begin{align}   
    \label{eqn10.14}  \chi  (t)  <  \la \quad  \mbox{for}\, \,    t  > 0 \, \, \mbox{near  0}.  
\end{align}  
 \noindent    
To  prove   \eqref{eqn10.14}, we   let, as in section \ref{section10}, $ u ( \cdot, t )  $  
be the  $ \mathcal{A} = \nabla f$-harmonic Green's function for    
$\mathbb{R}^{n}\setminus (E_1  + t E_2)$ with pole at infinity. We also 
let   $\mathbf{g} ( \cdot,  E_1 + t E_2  )     $ be the   Gauss map for  $ \ar (E_1 + t E_2 )  $    
while   $ h_1$ and $h_2  $  are   the support  functions for  $  E_1$ and $E_2$,  respectively.  
Then  from   Remark \ref{remark9.3}  and  Proposition \ref{proposition9.2}  we have    for  $ t  \in   (0, 1), $  
\begin{align}
\label{eqn10.15}   
\begin{split}
& \frac{d}{dt}  \mathcal{C}_{\mathcal{A}} (  E_1 +  t E_2 )\\
 &=\begin{cases} 
n \ga^{-1}   \mathcal{C}_{\mathcal{A}} (  E_1  +  t E_2 )   {\ds  \int_{  \ar  ( E_1 +  t E_2  ) }   h_2  ( \mathbf{g}(x, E_1  +  t E_2)  )   f ( \nabla u (x, t ) ) d\mathcal{H}^{n-1} } & \mbox{when}\, \, p =  n,  \\
 p  (p - 1 )  \mathcal{C}_{\mathcal{A}} (  E_1  +  t E_2 )^{\frac{p-2}{p-1}}   {\ds \int_{  \ar  ( E_1 +  t E_2  ) }   h_2  ( \mathbf{g}(x, E_1  +  t E_2)  )   f ( \nabla u (x, t ) ) d\mathcal{H}^{n-1} }   
&\mbox{when}\,\, p  >  n.  
\end{cases}
\end{split}
\end{align}
We shall prove     
\begin{proposition}  
\label{proposition10.1}  
\begin{align} 
\label{eqn10.16} 
 \lim_{\tau \to 0}   \int_{\ar  (E_1 + \tau E_2 )}   h_2   ( \mathbf{g}( x,  E_1  + \tau E_2 ) )  f ( \nabla u (x, \tau) ) d\mathcal{H}^{n-1}  \, = \, \infty.   
 \end{align}
 \end{proposition} 
Assuming Proposition  \ref{proposition10.1} we get  \eqref{eqn10.14} and so    a contradiction to our assumption that 
$ E_1 $  has empty interior as follows.  First observe from \eqref{eqn10.15}  and  \eqref{eqn10.13}   that   for $ p  = n, $  
\begin{align}
\label{eqn10.17} 
\begin{split}
 \mathcal{C}_{\mathcal{A}} &(  E_1 + t E_2 )  \left.{\ds  \frac{d}{dt} } \chi  (t) \right|_{t=\tau}  \\
&=     \sum_{i   = 1}^m  c_i  a   -  n \ga^{-1}  \,  \,   [  \sum_{i  = 1}^m     c_i ( \hat q_i   +  a \tau  )]  \int_{  \ar ( E_1  + \tau  E_2 ) } 
 h_2   ( \mathbf{g}(x,  E_1 + \tau  E_2  ) )  f ( \nabla u (x, \tau) ) d\mathcal{H}^{n-1} 
 \end{split}
\end{align} 
and if $ p > n, $ 
\begin{align}  
\label{eqn10.18}
\begin{split}
  [\mathcal{C}_{\mathcal{A}} &(  E_1 + t E_2 ) ]^{  \frac{ 2p - (n+1)}{(p-n)(p - 1)} }  \left. {\ds  \frac{d}{dt} } \chi  (t) \right|_{t=\tau}  = 
    \mathcal{C}_{\mathcal{A}} (  E_1 + \tau  E_2 )^{1/(p-1)}{\ds \sum_{i   = 1}^m  c_i  a}\\ 
 &-   \, \frac{p(p-1)}{(p-n)} \,  {\ds  [  \sum_{i  = 1}^m     c_i ( \hat q_i   +  a \tau  )]  \int_{  \ar ( E_1  + \tau  E_2 ) } 
 h_2   ( \mathbf{g}(x,  E_1 + \tau  E_2  ) )  f ( \nabla u (x, \tau) ) d\mathcal{H}^{n-1}}.
 \end{split}
 \end{align}  
Now    $  E_1 + \tau  E_2   \to   E_1 $ as $ \tau  \to 0 $  in  the sense of  
Hausdorff distance so by   \eqref{eqn10.0},   we have      (for all $n\leq p<\infty$) 
\begin{align} 
\label{eqn10.19}  
\lim_{\tau \to 0}   \mathcal{C}_{\mathcal{A}} (  E_1 + \tau  E_2   )  = \mathcal{ C}_{\mathcal{A}} (   E_1  ) =1.  
\end{align}   
Clearly,   \eqref{eqn10.16}-\eqref{eqn10.19}            imply for  some  $ t_0 > 0  $  small   that   
\begin{align}     
\label{eqn10.20}    
\left.\frac{d}{dt} \chi (t)\right|_{t=\tau} < 0 \quad   \mbox{ for } \, \, \tau \in ( 0, t_0].  
\end{align}
On the other hand, from  \eqref{eqn10.13} and  \eqref{eqn10.19} we see that  
\[
\lim_{\tau \to 0} \chi  ( \tau ) = \la.
\]     
From this observation, the mean value theorem from calculus,  
and   \eqref{eqn10.20} we conclude that \eqref{eqn10.14}  holds  so   $  E_1 $  has interior points.     
Now \eqref{eqn10.5}  follows from our earlier remarks.  
\begin{proof}[Proof of   Proposition  \ref{proposition10.1}] 
  Recall  that  $  E_1 $ is  contained in  a  $1\leq  l  <  n - 1 $ dimensional plane.   
  We assume as we may that    
\begin{align}
\label{eqn10.21} 
 E_1    \subset  \{ x = ( x',  x'' ) :  x' = (x_1, \dots, x_l) \, \,    \mbox{and} \,\, x''  =  (x_{l+1}, \dots  x_n )   = (0, \dots, 0)  \}  = \mathbb{R}^l.  
\end{align}   
Indeed, otherwise  we  can   rotate  our coordinate system to get  \eqref{eqn10.21}   
and  corresponding  $ \hat{\mathcal{A}}$-harmonic Green's functions, 
say   $  \bar u ( \cdot, t )$ for $\mathbb{R}^n\setminus E_1$ with a pole at infinity.      
Proving    Proposition \ref{proposition10.1}  for  $ \bar  u ( \cdot, t ) $   
and transferring back we obtain Proposition  \ref{proposition10.1}.     
   
We   also note   that
\begin{align}    
\label{eqn10.22}     
\bar B (0,  4 a )   \cap  \mathbb{R}^l    \subset   E_1    \subset  E_1  +  E_2   \subset  \bar B (0, \rho )    
\end{align}  
which follows from our choice of  $a$ and for some   $\rho  >  0 $  
depending only on the data, since  $ \mathcal{C}_{\mathcal{A}}(E_1)  = 1$.  
Next   we  observe  from   Lemma  5.3  in    \cite{LN4}  that    for fixed  $ p  > n -  l, $  
there exists     an    $\mathcal{A}$-harmonic  function $  \hat  V $  on   
 $  \mathbb{R}^n  \sem  \mathbb{R}^l   $ with continuous boundary value $0$   on  $\mathbb{R}^l$  satisfying  
\begin{align}
\label{eqn10.23}  
\hat V   (x)  \approx  | x'' |^{\psi} \quad \mbox{whenever}\, \,  x = (x', x'')   \in      \mathbb{R}^l\times \mathbb{R}^{n-l}  \, \,  \mbox{where}\,\,   \psi   =   \frac{ p -  n   + l }{p-1}.   
\end{align}
Ratio constants depend only on $ p, n, l $  and the structure 
constants for $ \mathcal{A}$.   Using   \eqref{eqn10.23}  we prove    
   \begin{lemma}  
 \label{lemma10.2}
 Fix  $ p  \geq n.  $   Then  there exists  $   C_1 \geq 10^{10},  $   depending on  $ p, n, l$, and $\rho$ but independent of $ t  \in (0, 1] $,  such   that 
   if   $ \psi = ( p  -  n  + l )/(p-1) $  and $ x=(x', x'') \in B (0, \rho )$  then    
\begin{align}   
\label{eqn10.24}    
|x''|^\psi   \leq  C_1  u ( x, t )  \quad \mbox{whenever}  \quad    C_1  t  \leq  |x''|.
\end{align}  
\end{lemma}    
\begin{proof}[Proof of Lemma \ref{lemma10.2}]  For  fixed $ t \in (0, 1 )$,  
let  $ v =  \max ( \hat V - C_2 t, 0 )$ for $C_2>0$ which will be fixed soon. 
Then  $  v $ is  $  \mathcal{A} $-harmonic  in       $\mathbb{R}^n \sem W  $ 
and  continuous on  $\mathbb{R}^n$   with  $ v \equiv 0 $ on $ W = \{ x :  \hat V (x) \leq  C_2 t \}. $    
    From  the definition of  $ E_1 + t E_2 $ and $ v $  we see  for $ C_2 $ large enough,  
    depending on $ p, n, l, $  the structure constants for  $ \mathcal{A}$,  and  $  \rho,  $  that 
        $ v  = 0 $ on  $ E_1 + t E_2 .$  Also    $  C_3  u  ( \cdot, t  )   \geq  v $  on  $  \ar  B ( 0, 2 \rho ) $  
        for  $ t \in [0, 1 ] $  as we deduce from  Harnack's  inequality and   
         $  E_1  +  E_2  \subset  B (0, \rho). $ 
         Here  $ C_3$  has the same dependence as  $ C_2.$    Using the maximum principle 
         for  $  \mathcal{A}$-harmonic functions it now follows that  
         $   v  \leq  C_4  u ( \cdot, t ) )$  in    $  B ( 0, 2 \rho ). $    From this fact and our 
         knowledge of  $  \hat  V $   we get   Lemma  \ref{lemma10.2}.   
         \end{proof}    
      
To begin the proof of  Proposition \ref{proposition10.1}   we  assume  $ 0  < t \leq \ti t_0,  $  
where $ \ti t_0   < <  a$.  We  also  observe    
 that     $ E_1 +  t  E_2  $   is  a compact  convex set   with  nonempty interior.   
 From   \eqref{eqn8.6} $(b)$ we find that            for  $ \mathcal{H}^{n-1}$-almost every  $ \hat x  \in  \ar  ( E_1 + t E_2) $  
 \[
\nabla u ( y, t ) \to  \nabla u ( \hat x, t )\quad \mbox{as} \quad  y \to \hat x
\]
 non-tangentially  in  $   \mathbb{R}^n \sem (E_1 + t E_2 ). $  
 Moreover,  there exists  $ \ti c $   such that     
 $  B ( \hat x, 4t/\ti c ) \cap  \ar ( E_1 + t  E_2 ) $ is  the graph of  a  Lipschitz  function     whenever 
 \[ 
 \hat x \in     B (0, 2a) \cap  \ar ( E_1 + t  E_2 )\quad   \mbox{and} \quad   0  < t \leq  \ti t_0  
 \]
 with Lipschitz constant independent of $ \hat x, t. $   
 It then    follows  from \eqref{eqn8.5}, \eqref{eqn8.3}, \eqref{eqn8.8}, and  \eqref{eqn8.9} $(a)$ with $ q = p $   that   
 \begin{align}  
 \label{eqn10.25}  
\begin{split}
  c   \int_{B ( \hat x, t/\ti c ) \cap \ar (E_1 + t E_2)   }  f ( \nabla u ( \cdot, t ) ) d\mathcal{H}^{n-1} & \geq (  u ( w, t ) )^p   t^{ n - 1 -p}  \\
   & \geq   c^{-1}   \int_{B ( \hat x, t/\ti c ) \cap  \ar ( E_1 + t  E_2)    }  f ( \nabla u ( \cdot, t ) ) d\mathcal{H}^{n-1} 
   \end{split}
\end{align}
 where  $ c$  has the  same  dependence as in  Lemma  \ref{lem8.2}.    Also     
     $ w = w (\hat x, t ) $ denotes a point in  $ B (\hat  x, t/\ti c ) \cap (\mathbb{R}^{n}\setminus (E_1+t E_2))$ 
     whose distance from  $ \ar (E_1+ t E_2  ) $ is   $ \geq  t/ c^2.$

Using   Harnack's inequality in  a  chain of  balls of radius $ \approx t $  connecting $ w $  to a  point 
 $ x \in  B (0, a )  $  with    $ 2 C_1  t  = |x''|   $   we deduce from  \eqref{eqn10.24} of  Lemma 
 \ref{lemma10.2} that  
\begin{align}   
\label{eqn10.26}  
 u (w, t ) \geq  C^{-1}  t^{\psi}
\end{align}
where $ C $ is independent of $ t \in (0,1). $    
 Using   \eqref{eqn10.26} in  \eqref{eqn10.25} we obtain   for some $ C'   \geq 1, $ 
 independent of  $ t, 0 < t \leq  \ti t_0,   $     that  
\begin{align}
\label{eqn10.27}    
   C' \int_{B ( \hat x, t/\ti c ) \cap  \ar (E_1  + t   E_2) }  f ( \nabla u ( \cdot, t ) )  d\mathcal{H}^{n-1} \,   \geq  t^{p( \psi - 1) + n-1}. 
\end{align}    
Now since  $  \ar (E_1 +  t E_2 )  \cap  B (0, 2a )  $ projects onto a  set  
containing   $  B (0, 2a )  \cap  \rn{l} $   for $ 0 < t \leq  \ti t_0$,    
we see there is  a  disjoint collection of  balls   $  B ( \hat x,  t/\ti c ) $  
for $ \hat  x \in  \ar ( E_1 + t E_2) $   of  cardinality  approximately  $ t^{-l}  $   for which  \eqref{eqn10.27}  holds.    
Since     
\[  
p  ( \psi - 1)  +  ( n - 1) -  l  =  (l + 1 - n ) / (p-1) <  0 
\]   
   we conclude from  
    \eqref{eqn10.27}    that  for some $ C^*  $  independent of  small positive $ t $       
 \begin{align} 
 \label{eqn10.28a}    
 C^*  \int_{\ar (E_1 + t E_2 ) \cap B(0,2a)}  f ( \nabla u ( \cdot, t ) )    d\mathcal{H}^{n-1}   \,  \geq   t^{ (l + 1 - n)/ (p-1) } \to \infty  \quad \mbox{as} \quad t  \to 0.  
\end{align}    
      Finally, note that for  $ 0 < t  \leq \ti t_0, $   
      \[
       \mathbf{g}(x, E_1 + t  E_2   )  \in  \{ \xi_i : i \in \La \}
       \]  
       for   $\mathcal{H}^{n-1}$-almost  every  $ x  \in \ar ( E_1 + t E_2) \cap  B (0, 2a )   $   
       and      $ h_2  ( \xi_i  )  \equiv   a  $  whenever $ \xi_i \in \La. $    
  From this note and  \eqref{eqn10.28a},  we   obtain first 
  the validity  of \eqref{eqn10.14}  in  Proposition  \ref{proposition10.1} 
  and thereupon that  \eqref{eqn10.5} is true.      
   \end{proof}    
  Armed with   \eqref{eqn10.5}, we can use  the  well-known  variational  argument mentioned
   after this display to  complete the  proof  of   existence in  Theorem  \ref{mink}
    in the discrete case.  To  do so we first  observe  that if  $ E_1$ and $E_2 $  are  compact 
    convex  sets  with nonempty interiors and if   $ h_1$ and $h_2$  are the support  
    functions  for these  sets with the corresponding  $ \mathcal{A}$-harmonic  Green's functions $ u_1$ and $u_2$ then   
\begin{align}
 \label{eqn10.28} 
 \begin{split}    
  \frac{d}{ds} & \mathcal{C}_{\mathcal{A}}  ( (1-s) E_1 + s  E_2 )  |_{s=0} \\
 & =
  \begin{cases} 
  \hs{.1in}  n \ga^{-1}  \mathcal{C}_{\mathcal{A}} (E_1)   {\ds \int_{ \ar  E_1}  (h_2 - h_1 ) ( \mathbf{g} ( x,   E_1 )  )     f   ( \nabla u_1 (x ) ) d \mathcal{H}^{n-1}} &
  \mbox{when} \, \,p = n,  \\
 \hs{.1in}      p  (p-1)  \,   \mathcal{C}_{\mathcal{A}}(E_1  )^{\frac{p-2}{p-1}}    \, \,     
      \, {\ds  \int_{  \ar   E_1  } }  (h_2 - h_1)  ( \mathbf{g}(x,  E_1  ) )   f( \nabla u_1  ( x ) ) \, d  \mathcal{H}^{n-1} &  \mbox{when}\,\,  p > n. 
\end{cases}
\end{split}
\end{align}      
 Here      \eqref{eqn10.28} follows from  Proposition \ref{proposition9.2}, \eqref{eqn9.37},  \eqref{eqn9.38},   
 and the  chain rule  using  $ t =   s/ (1-s) $    and    
\[ 
 \mathcal{C}_{\mathcal{A}}  ( (1-s) E_1 + s  E_2 )  =
 \begin{cases}  
 (1-s)  \mathcal{C}_{\mathcal{A}}  (  E_1 + t  E_2 ) & \mbox{when}\, \, p = n, \\
 (1-s)^{p-n}    \mathcal{C}_{\mathcal{A}}  (  E_1 + t  E_2 ) & \mbox{when}\, \, p > n.
\end{cases}
\] 
Let    $   q^*    = ( q^*_1,  \dots, q^*_m ) \in \mathbb{R}^m $   with  $ q^*_i  > 0$ for $1 \leq i  \leq m$ 
and  $  \breve  q  $  as in   \eqref{eqn10.5}.    From the note   after  \eqref{eqn10.12} and  \eqref{eqn10.5}  
we deduce   for $ \bar t_0   > 0 $ sufficiently small,    that     
$  E ( q^* (t) ) \in   \Ph $   when   $  0 < t  \leq  \bar t_0, $     where  
\[ 
q^* (t )   =
\begin{cases}
 {\ds \frac{  (1-t) \breve  q  +  t q^* }{\mbox{C}_{\mathcal{A}}( (1-t) E (\breve   q)  + t E (q^*))} }& \mbox{when} \, \, p  = n, \\   
  {\ds    \frac{  (1-t) \breve  q  +  t q^* }{\mbox{C}_{\mathcal{A}}( (1-t) E (\breve   q)  + t E (q^*))^{1/(p-n)} }} & \mbox{when}\,\,    p  > n.
\end{cases}      
         \]     
Also,    $  \he (q^* (t)) \geq  \la  $  for  $ 0 \leq t  \leq  \bar{t}_0 $ thanks to  \eqref{eqn10.5}.    
Differentiating  $ \he  ( q^* (t) ) $ with respect to $ t $  and evaluating at  0  we  obtain from   
\eqref{eqn10.28} with $ E_1 = E (\breve q ) ,$   $  E_2  = E ( q^* ),  $    and $  u_1 $ the  
$ \mathcal{A}$-harmonic Green's function for $\mathbb{R}^{n}\setminus E (\breve q ) $, 
as well as     \eqref{eqn10.5} and the fact that $\mathcal{C}_{\mathcal{A}}  (  E_1)=1$ by \eqref{eqn10.19} that  when 
$ p = n, $    
\begin{align}
 \label{eqn10.29} 
 \begin{split}
  0   & \leq   \left. \frac{ d \he ( q^* (\tau) )}{d\tau}\right|_{\tau=0}  \\
  &=     { \ds  \sum_{i=1}^m  c_i  ( q^*_i -  \breve  q_i ) } -  n \la    \ga^{-1}   \, {\ds  \int_{\partial E (\breve  q) }  
 ( h_2  -   h_1)  ( \mathbf{g} ( x, E ( \breve q) ) )   \,  f  ( \nabla  u_1  ( x ) )  d \mathcal{H}^{n-1} }  \\
 &=    { \ds   \sum_{i=1}^m  c_i  ( q_i^* - \breve q_i ) -   n \la  \ga^{-1}  \, \sum_{i=1}^m  ( q_i^* - \breve q_i ) \int_{ \mathbf{g}^{-1} ( \xi_i,  E ( \breve q )  ) } f ( \nabla u_1 ( x) ) d \mathcal{H}^{n-1} } 
 \end{split}
 \end{align}
 provided  $ q^*  $  is  near enough $ \hat q. $  Similarly, if  $ p  > n, $ 
\begin{align}
 \label{eqn10.30} 
 \begin{split}
  0   & \leq   \left. \frac{ d \he ( q^* (\tau) )}{d\tau} \right|_{\tau=0}  \\
  &=     { \ds  \sum_{i=1}^m  c_i  ( q^*_i -  \breve  q_i ) } -  p  \la  {\ts  \frac{p-1}{p-n}}   \, {\ds  \int_{\partial E (\breve  q) }  
 ( h_2  -   h_1)  ( \mathbf{g} ( x, E ( \breve q) ) )   \,  f  ( \nabla  u_1  ( x ) )  d \mathcal{H}^{n-1} }  \\
 &=    { \ds   \sum_{i=1}^m  c_i  ( q_i^* - \breve q_i ) -   p  \la   {\ts \frac{p-1}{p-n}} \, \sum_{i=1}^m  ( q_i^* - \breve q_i ) \int_{ \mathbf{g}^{-1} ( \xi_i,  E ( \breve q )  ) } f ( \nabla u_1 ( x ) ) d \mathcal{H}^{n-1} } 
 \end{split}
 \end{align} 
From arbitrariness  of $  \breve  q_i  -  q_i^* $  we conclude  that   
\begin{align}
 \label{eqn10.31} 
 c_i  =  b   \,   \int_{ \mathbf{g}^{-1} ( \xi_i,  E ( \breve  q )  ) } f ( \nabla u_1(x))   \, d \mathcal{H}^{n-1} \quad  \mbox{for}\, \,   1 \leq i \leq m. 
 \end{align}
  where  
  \[
   b  = 
   \begin{cases}
   n \la/\ga  &\mbox{when}\, \,  p =n \\
     p  \la   {\ts \frac{p-1}{p-n}}  & \mbox{when}\, \, p > n.
\end{cases}
\]       
         Observe   
that  $ \la  > 0  $   since   otherwise  $  E ( \breve q )  =  \{0\}  $  a contradiction  to    
$ \mathcal{C}_{\mathcal{A}} ( E (\breve q )  )  =  1. $    
From  \eqref{eqn10.31}  and  $ p$-homogeneity of  $ f $   we  find that if  $ p \geq n  $  and 
$  E  =  \de   E ( \breve  q ) $   where $ \de^{-1}    =  n \la/\ga  $ when $  p  = n $  while   
$ \de ^{n - p -  1} = p { \ts (\frac{p-1}{p-n}) } \,  \la    $  for  $ p > n ,$   then  \eqref{eqn10.4} holds.   
  This completes the proof of existence in the discrete case when     \eqref{eqn10.1}-\eqref{eqn10.3} are valid.  
  \begin{remark} 
  \label{remark10.3}
For later use we note from     \eqref{eqn9.37} and \eqref{eqn9.38} with $ E_0 = E,   $   
    that if   $ h $ is the support function for $ E $  as in  \eqref{eqn7.2}  when $ p  \geq  n, $    then  
\begin{align}
 \label{eqn10.32} 
\begin{split}
p {\ds \int_{ \mathbb{S}^{n-1} }} &    \,   h ( \xi   ) d \mu  ( \xi  ) =
 \begin{cases}  
 \ga  &  \mbox{when}\, \, p  = n, \\
    \, \frac{p-n}{p-1}  \, \mathcal{C}_{\mathcal{A}}  (E)^{1/(p-1)}  
   \leq  c      (\mbox{diam}(E))^{\frac{p   - n }{p - 1}}  & \mbox{when}\, \, p  > n.
   \end{cases}
 \end{split}
 \end{align}
\end{remark} 
 \subsection{Existence in Theorem \ref{mink} in   the continuous case} 
Armed with existence in Theorem  \ref{mink} in the discrete case when $ p  \geq n$, we now  
consider existence when  $ \mu  $  is  a  finite positive Borel measure on $ \mathbb{S}^{n - 1}$  satisfying  \eqref{eqn7.1}.  We  choose  a  sequence of  discrete measures     
$ \{\mu_{j}\}_{j\geq 1}  $  satisfying  \eqref{eqn10.1}-\eqref{eqn10.3} when $ p \geq n  $ is fixed  with  
\[
\mu_j  \rightharpoonup \mu\quad   \mbox{weakly as}\quad   j  \to  \infty.  
\]   
Let  $  E_j, j = 1, 2, \dots,  $  be a  corresponding sequence of  compact convex sets with 
$ 0 $ in the interior of  $ E_j $  and  corresponding $ \mathcal{A}$-harmonic Green's  
functions $U_j$ for $\mathbb{R}^{n}\setminus E_j$ with pole at infinity  for which    \eqref{eqn10.4}  holds  at support points of  $ \mu_j. $   
From the  definition of  weak convergence we  may  assume  for  some  $ C \geq 1 $ that  
\begin{align} 
\label{eqn11.1}   
C^{-1}  \leq   \mu_j  (  \mathbb{S}^{n-1} )  \leq   C  \quad  \mbox{for}\, \, j = 1, 2, \dots 
\end{align}  
We  claim that we may also assume  
\begin{align} 
\label{eqn11.2}   
  E_j  \subset  \bar B ( 0,  \rho ) \quad  \mbox{for $ j = 1, 2, \dots, $ and some}\, \,   \rho  <  \infty.
\end{align} 
 To prove   \eqref{eqn11.2}   we  observe  from   weak convergence of  
 $ ( \mu_j) $  to $ \mu $  and  \eqref{eqn10.1},   \eqref{eqn10.2},  that  
 for some  $  \hat C  \geq 1, $  independent of  $ j =1, 2, \ldots, $  
\begin{align}  
\label{eqn11.3} 
{\hat C }^{-1}    \leq   \int_{ \mathbb{S}^{n-1}}   \lan \tau, \xi \ran^+  d \mu_j ( \xi )  \quad \mbox{for all}\, \, \tau  \in \mathbb{S}^{n-1}.
\end{align}  
Using   \eqref{eqn11.3}, \eqref{eqn10.32},   and arguing   as in   \eqref{eqn10.7}  
we deduce    that  if $  \hat \tau_j \in  \mathbb{S}^{n-1}$ so that $ \hat s_j  \hat \tau_j  \in  E_j  $  
with  $  \hat s_j   \geq \frac{1}{10}  $   $\mbox{diam}(E_j)$  and $ h_j $ is the  support function   for   $ E_j$   then  
\begin{align} 
\label{eqn11.4}   
   \hat  C^{-1} \,  \hat  s_j  \leq     \int_{\mathbb{S}^{n-1}}       \lan    \hat s_j  \hat \tau_j,  \xi   \ran^+  \,  d \mu_j  (\xi)    \leq     \int_{\mathbb{S}^{n-1}}     h_j ( \xi )  d\mu_j ( \xi )    \leq   
 \ti C s_j^{\frac{p - n}{p-1}}
 \end{align}    
  where all  constants are positive  and  independent of $ j. $  
   Thus, claim  \eqref{eqn11.2}  is true.  
       
From  \eqref{eqn11.2}  we   see   that  a  subsequence of  
$ \{E_j\}_{j\geq 1}$  (also denoted $\{E_j\}$)  converges to a  compact convex set  $  E   \subset  \bar B ( 0, \rho ) $
 in the sense  of  Hausdorff distance.  Choosing   another  subsequence if  necessary we may assume 
 from the same argument used in proving \eqref{eqn10.0}  that  either  $ E = \{0\}$ or  $ U_j  \to U $ uniformly in $ \rn{n} $ where   $ U  $ is
the   $ \mathcal{A}$-harmonic Green's function  for  $ \mathbb{R}^{n}\setminus E$ with pole at infinity.    
If  $ E $ has  nonempty  interior then  from  Proposition  \ref{proposition9.1}  we conclude that   \eqref{eqn7.2} $(b)$ and $(c)$  hold for $ U$,  $E$, and $\mu. $  
 
 If  $ E $ has empty interior we  consider the following three  cases :\\     

\noindent{\bf Case A:} $E $ has dimension $ l,   1 \leq l   <  n - 1. $
 In this case  we once again use the argument in \cite[section 13.2, Case B1]{AGHLV} to get a contradiction. 
Translating $ E $ if  necessary we may assume that $0$  is an interior point of the $l$ 
dimensional plane $ P $  containing $ E  $.   So as in the discrete case  we  assume that 
\begin{align} 
\label{eqn11.5}  
B ( 0, 4a ) \cap P  \subset E   \subset    B ( 0, \rho) \cap  E 
\end{align}
 and   
 \[  
  t_j  =     d_{\mathcal{H}} (  E_j,  E ) \quad \mbox{for}\, \, j  = 1, 2, \dots.    
\]  
Then for $  j $  large  enough we can argue as in  Lemma  \ref{lemma10.2}  with 
$  t $  replaced by $ t_j,     $     and  $  u(\cdot, t ), E_1 + t E_2$   by  $ U_j, E_j$.
We  obtain from the analogue of  \eqref{eqn10.24}  for $ j \geq j_0,  $  that  
\begin{align}  
\label{eqn11.6}   
 U_j  ( x  )    \geq  C_1^{-1}   \, |x''| ^{\psi }  \quad    \mbox{for}\, \,    x = ( x', x'' ) \in  B (0,  \rho ) \, \,   \mbox{and}\, \,     C_1 \,  t_j  \leq  |x''|
\end{align}  
where $ x' \in \mathbb{R}^l $  and  $ \psi   = (p-n+l)/(p-1). $   
Fix  $ j  \geq  j_0, $   and given  $  y   \in \ar E_j  \cap  B ( 0,  a)  $,  
let $ T_j (y),  $     be  a  supporting hyperplane  to  $  \ar E_j $  at  $ y. $  
 Let  $\hat H_j $  be the open half space  with  
 $\hat H_j \cap E_j = \es $  and  $ \ar \hat H_j  =  T_j( y ). $  
  Let   $  y^* $  denote the point  in  $ \hat H_j $  which lies on the normal line through $ y  $  
  with  $  | y - y^*  |  =  2  C_1  t_j $ where $  C_1  $ is as in  \eqref{eqn11.6}.   
  Note that  for $ j_0 $ sufficiently large and $ j \geq j_0 $   that  $ d ( y^*, P )  >  C_1 t_j $,  
  since otherwise  it  would follow from the triangle inequality that there exists  $ z  \in B  (0, 2 a ) \cap E $ 
  with  $ d ( z, E_j ) >  t_j$.  Thus  \eqref{eqn11.6} holds with $ x = y^*$.     
  Let   $ \ph $  be the  $ \mathcal{A}$-harmonic function in    $ \hat  H_j \cap  B ( y,  8    C_1  t_j )  \sem 
  \bar B ( y^*,  C_1 t_j)   $  with continuous boundary values
  \[
\phi\equiv 
\left\{
\begin{array}{ll}
 U_j & \mbox{on} \,\,\, \ar B ( y^*,   C_1   \, t_j), \\
0 & \mbox{on}\, \,\, \ar ( \hat H_j  \cap B ( y,  8 C_1  t_j )).
\end{array}
\right.  
  \]
  Then  from the maximum principle for 
     $ {\mathcal{A}}$-harmonic functions we have  $ \ph \leq    U_j$ on 
     $ \hat  H_j \cap  B ( y,  8    C_1  t_j )  \setminus B ( y^*,   C_1   \, t_j)$.    
     Comparing $ \ph $ to a linear function and using a boundary Harnack inequality from  \cite{LLN} in  
  $\hat  H_j  \cap  B ( y,  8 C_1   t_j) $  we deduce for some $ c^* $ depending only on the data   and  $\rho $ that     
  \[ 
   U_j (  y^{*} ) ) /t_j  \, \leq \,  c^*    U_j  (  \hat  z ) ) / d (  \hat z,  T_j (y)) 
  \]
when $  j \geq j_0 $  and $  \hat z  \in      \hat H_j \cap B ( y,   C_1   \,  t_j )$.  
Letting  $ \hat z  \to y $ non-tangentially,   we conclude from      this  inequality and  \eqref{eqn11.6} with $ x  = y^* $  that         
\begin{align}
\label{eqn11.7}   
\, t_j^{\psi  - 1}\,    \leq  C^{**}  |   \nabla U_j ( y ) |  \quad   \mbox{for}\, \,  \mathcal{H}^{n-1}\mbox{-almost  every}\, \,  y \in \ar E_j \cap B (0, a) 
  \end{align}
  and  $ j  \geq j_0.  $ Here  $ C^{**} $  has the same dependence as $ C_1. $  

Let $ \nu_j $ be the positive Borel  measure corresponding to  $ U_j $ with support contained in $E_j$ as in  Lemmas \ref{lem5.2}.  \ref{lem5.3}.    
Then  from  \eqref{eqn8.8}  we deduce  for $ j = 1, 2, \dots $ that  
\begin{align} 
\label{eqn11.8}   
   \frac{d \nu_j }{ d \mathcal{H}^{n-1}} (y)   =   p   \frac{ f  ( \nabla U_j (y) )}{| \nabla U_j ( y ) |} \quad \mbox{for}\quad  \mathcal{ H}^{n-1}\mbox{-almost every}\, \, y  \in  \ar  E_j.
   \end{align} 
From  \eqref{eqn11.8},  the  definition of  $ \mu_j,  $   the  structure assumptions on  $f$,  
and  \eqref{eqn11.7}  we   conclude for some $ \breve C \geq  1,  $ independent of  $ j  \geq j_0, $  that 
for fixed $ p  \geq n $   
\begin{align}
\label{eqn11.9}  
t_j^{\psi  - 1}\,  \nu_j ( \ar E_j \cap B (0,a) )   \leq   C'  \mu_j  ( \mathbf{g}( \ar  E_j  \cap  B (0, a ), E_j))  .
  \end{align}
     Here $\mathbf{g}(\cdot, E_j)$ is the Gauss map for $\partial E_j$.  Now     $ \psi  -  1  =    ( 1 - n + k )/(p  -1) < 0$   
     and  from  weak  convergence of  $ (\nu_j) $  to  $ \nu $   we  
  have    
     \[  
     \liminf_{j  \to  \infty}  \nu_j (\ar E_j  \cap B (0, a )  )  \geq  \nu  ( \ar E  \cap B (0, a/2)) > 0 
     \]  where the  right-hand inequality follows from the fact that    $ E \cap B ( 0, a/2) $   is 
     uniformly  $ (a/2, p)$-fat.   
     Putting this inequality in  \eqref{eqn11.9} we see that    $  \mu_j  ( \mathbb{S}^{n-1})  \to \infty $  in  
     contradiction  to   \eqref{eqn11.1}.  Thus  $ E $  does not have dimension $ l, 1 \leq l < n - 1. $ \\ 

\noindent  {\bf Case B:}  $E = \{0\}$.   In  this case  we put   $ t_j   =\mbox{diam}(E_j)$  and  let  $ \nu_j $  be  the positive Borel measure, as  above,     relative to 
$ U_j$ with support contained in $E_j$ as in Lemma \ref{lemma2.4}.     Then from  $(d)$ of Lemmas  \ref{lem5.2} and \ref{lem5.3}  we have  
$  \nu_j ( E_j )  = 1$ for $j  = 1, 2,  \dots  $   From  this  fact  and   \eqref{eqn2.6} $(ii)$      we   see that   
\begin{align}  
\label{eqn11.10}    
U_j   \approx  t_j^{\frac{p-n}{p-1}}  \quad \mbox{on}\,\,  \ar B ( 0, 2 t_j ) \, \, \mbox{for}\, \, j  = 1, 2,  \dots 
\end{align}
  where ratio constants depend only on the data.  Using 
\eqref{eqn11.10}  and   the same argument as in  the derivation of   \eqref{eqn11.7}  it follows that for some $ c \geq 1 $  depending only on the data,    
\[   
| \nabla  U_j ( y ) |  \geq    c^{-1}  t_j^{\frac{1-n}{p-1}} \quad \mbox{for} \, \,  \mathcal{H}^{n-1}\mbox{-almost every}\, \,   y \in \ar E_j.  
\]   
Since $ \nu_j  (E_j) = 1$ for $j = 1, 2, \dots, $  we  again  obtain  as in  {\bf Case A} that  
 \[  
 c \mu_j  (  \ar E_j )  \geq    t_j^{\frac{1-n}{p-1}}  \to  \infty  \quad \mbox{as}\, \,  j \to \infty  
 \]  
   in contradiction to  
\eqref{eqn11.1}.    So  $ E  \not =  \{0\}. $    \\

\noi {\bf Case C:} $ E $ has dimension  $ n - 1. $ In this case, we essentially  copy the  proof  from  \cite{AGHLV}  through the statement of  
Proposition \ref{proposition11.2}.  However  to  prove   Proposition  \ref{proposition11.2}    we state  and  prove   Lemma 
\ref{lemma11.3}  which  in   \cite{AGHLV}   was  only  available  when   
$ f (\eta)  =  p^{-1} |\eta|^p, $  i.e, for the  $p$-Laplace equation.  Using    Lemma  \ref{lemma11.3}   we  can  then copy the  
so called simple proof given in    \cite{AGHLV}   of  \eqref{eqn11.18}   for the  $p$-Laplace   equation when   $ 2 <  p <  n. $   
  
To  begin the proof  we assume, as we may,   that   $ P  = \{ x: x_n  = 0  \}  $   
and 
\begin{align}  
\label{eqn11.11}   
B  ( 0,  4a)  \cap  P  \subset     E   \subset  B ( 0, \rho ) \cap P.
\end{align}
   Also  translating  $  E_j $ slightly upward  if  necessary  we  may assume  that   
\[   
\lim_{j \to  \infty}    d_{\mathcal{H}} ( E_j,  E ) = 0   \quad \mbox{and}\quad   E_j  \subset \{  x: x_n > 0  \}.   
\]  
 Let  $ \nabla U_{+} (x) $  denote  the  limit (whenever it exists)  as   $y \to  x \in  E $     
 non-tangentially  through values with $  y_n  > 0 .$    We  prove   
 \begin{proposition}  
 \label{proposition11.1}    
  There exists $ C \geq 1 $  such that    
\begin{align}
\label{eqn11.12}   
C  \liminf_{j \to \infty }   \int_{\ar  E_j }  f ( \nabla U_j )  d\mathcal{H}^{n-1}  \geq    
    \int_{E}  f ( \nabla U_+)  d\mathcal{H}^{n-1}   
     -  C^2    \mathcal{H}^{n-1} ( E ).  
\end{align}
 \end{proposition}   
\begin{proof} 
 Given  
    $ \ep  >  0  $  choose  $ j_1 $ so large that  
   $  d_{\mathcal{H}} ( E_j,  E)  \leq  \ep $ for 
    $ j \geq j_1. $  Using  uniform convergence of  $(U_j)$  to $ U $ on $ \rn{n} $                       
    and   comparing   boundary values of   $ U$ and $U_j  $ in  
$ B (0,  2 \rho )  \sem  E_j,   $   we deduce from    Lemma  \ref{lemma2.3} that there exist   $ 0 < \ti  \si  \leq  1/2$ and $\hat C \geq  1$ 
  such that      
\begin{align}
\label{eqn11.13}  
  U  \leq  \hat  C   (    U_j     +  \ep^{\ti \si}  ) 
    \end{align}
    for  $ j  \geq j_1$.     Next  we  divide  the interior of   $ E $ into   
    $(n  - 1)$-dimensional closed  Whitney cubes   $ \{ Q_k  \}$.   Let    $ \ar'   E  $  denote  
the  boundary of  $  E  $  considered  as  a  set in  $ P$ and let $d'(\cdot,\cdot)$ denote the distance between sets considered in $P$. 
Then  the cubes in  $  \{ Q_k \} $  have disjoint interiors with  
side length  $ s ( Q_k ) $  and the property that considered as  sets in  $  
    P,   $  the distance say  $ d' ( Q_k,  \ar'  E  ) $  from $ Q_k $ to  the boundary of  $ E $    satisfies  
\begin{align}
\label{eqn11.14} 
10^{-n} s(Q_k)  \leq   d' ( Q_k, \ar'  E  )  \leq  10^n  s ( Q_k). 
\end{align}  
Let         
  $  Q  \in  \{ Q_k \}  $  with  $ s ( Q )  \geq   \ep^{\ti  \si }    $  and put    
  $   Q_+  =  Q \times (0, s(Q))$.
Suppose  $y_Q= (y_1, \dots, y_n )$  is   a  point in    $  Q_+ \sem E_j $  with    
 $ d( y_Q,  \ar'  E )   \geq y_n/ 2  \geq    s(Q)/4$.  

 We  consider two possibilities.  If   $   U  ( y_Q )  \geq 2 \hat C \ep^{\ti \si } $ ($\hat C $  as in 
\eqref{eqn11.13}),  then from \eqref{eqn11.13} we have    $       U(y_Q)     \leq 2 \hat C  U_j ( y_Q)   $  
and using    \eqref{eqn8.5} $(b)$,    \eqref{eqn8.8},  \eqref{eqn8.9} $(a)$,       for   $ U_j,  U  $   we  get 
\begin{align}  
\label{eqn11.15}          
\begin{split}    
    \bar C^3 \int_{\ar  E_j  \cap  Q_+ }  f ( \nabla U_j )  d\mathcal{H}^{n-1} &\geq  \bar C^2  (U_j )^p (y_Q) \,   s ( Q )^{ n-1- p}  \\
    & \geq  \bar C  (U )^p (y_Q)  \, s ( Q )^{ n-1 - p}  \\
    &\geq { \ds \int_{ Q  }  f ( \nabla U_+ )  d\mathcal{H}^{n-1} }.
  \end{split}
  \end{align}
      If   $   U ( y_Q )  <   2 \hat C   \ep^{\ti \si}, $ then since  $ s ( Q )  \geq \ep^{\ti \si}  ,$  an argument similar to the above gives  
\begin{align}
\label{eqn11.16}   
\int_{ Q  }  f ( \nabla U_+ )  d\mathcal{H}^{n-1}   \leq  C_+ \,  s ( Q )^{n-1} 
\end{align}
 where  $ C_+  $ is independent of $ j \geq j_2 \geq j_1 $  provided $ j_2 $  is  large enough. 
  Combining   \eqref{eqn11.15},  \eqref{eqn11.16},  and using    \eqref{eqn11.14}    we find  after summing over  
  $  Q  \in  \{ Q_k \} $  that  for some  $ \breve C  \geq 1, $ independent of $ j  \geq j_2, $     
\begin{align} 
\label{eqn11.17}     
\breve C \int_{\ar E_j }  f ( \nabla U_j )  d\mathcal{H}^{n-1}  \geq  \int_{ \{ x \in  E :\,  d' (x, \ar' E) \geq \breve C  \ep^{\ti \si } \} }   f  ( \nabla  U_+ )  d\mathcal{H}^{n-1}  -  
    \breve C^2  \mathcal{H}^{n-1}  ( E ).
    \end{align}

Letting first $ j \to \infty $  and after that  $ \ep \to 0 $   we obtain from \eqref{eqn11.17}   
and  the monotone convergence theorem or  Fatou's Lemma that  \eqref{eqn11.12} is true. 
This finishes the proof of   Proposition \ref{proposition11.1}.    
\end{proof}

Next we prove   

  \begin{proposition}  
  \label{proposition11.2} 
\begin{align}   
\label{eqn11.18}     
\int_{ E }  f ( \nabla U_+)  d\mathcal{H}^{n-1} =  \infty.    
    \end{align}
     \end{proposition}   
   We  note that   Propositions  \ref{proposition11.1} and \ref{proposition11.2}  give a 
   contradiction to      \eqref{eqn11.1}.    From  this  contradiction we  conclude  
   first  that  $ E $  does not have dimension  $ n-1$.  From  our  previous work  
   it now follows that $ E $  is  a  compact convex set with nonempty interior  and   
   \eqref{eqn7.2}  $(a)-(c) $  are valid with  $  u = U.$ 
   This finishes the proof of existence in Theorem \ref{mink} up  to  proving  Proposition  
   \ref{proposition11.2}.    
To  prove    Proposition \ref{proposition11.2} we shall  need
\begin{lemma}  
\label{lemma11.3}  
Given  $  \hat \eta =  ( \hat \eta_1,  \hat \eta_2 )   \in  \rn{2} \sem \{0\} $, put   $  \ti  f ( \hat \eta )  =  f ( \hat \eta_1,  \hat \eta_2,  0, \dots, 0 ). $      
 Then  there  exists   a  continuous function  $ v  $  on $ \rn{2}$  that is    $  \mathcal{ \ti A}:= \nabla \ti f $-harmonic in    
 $ \rn{2} \sem \{ x = (x_1, 0) \in \rn{2}   :  x_1 \leq 0  \} $     
 with  $ v  \equiv  0 $  on $ \{(x_1, 0) \in \rn{2}   :  x_1 \leq 0  \} $.   Moreover,   $ v (1, 0)  = 1 $  and  
\begin{align}  
\label{eqn11.19}   
v (t x  )  =  t^{ 1 - 1/p} v ( x )  \quad \mbox{whenever}\, \, t > 0 \, \, \mbox{and} \, \,  x \in \rn{2}.  
\end{align}
\end{lemma}   
\begin{proof}
We point out that   Krol in  \cite{Kr}   proved    Lemma   \ref{lemma11.3} 
 when  $ \ti f (\hat \eta) =  p^{-1}  |\hat \eta |^p,  $  so that  $ v $ in this case is  
 a  solution to the  $p$-Laplace equation.     
 We introduce polar coordinates, $ r = |x|,   x_1 = r  \cos \he,   x_2 = r \sin \he,   $   and put   
\[   
D (\tau )  =   \{ x :\, \,  r > 0, \, \,    |\he |   <  \tau\}  \quad  \mbox{for}\, \,      \pi/2  \leq   \tau  \leq  \pi. 
\]   
We  claim  for  $ \pi/2 \leq   \tau   <  \pi  $ that   there exists  a  unique  
   positive  $ \mathcal{\ti A}=  \nabla  \ti f$-harmonic function $  w  = w ( \cdot, \tau )  $  in  $ D ( \tau ) $ 
   which is continuous  in  $  \rn{2} $  with    $ w  \equiv  0  $  on   $ \rn{2} \sem D ( \tau)$  
 and  $  w (1, 0 ) =  1 $.  Moreover,   
\begin{align}
\label{eqn11.20} 
w (  t x )    =  t^{\la} \, w ( x ) \quad \mbox{whenever}\, \, x \in \rn{2} 
\end{align}
 for some  $ \la = \la  (\tau) > 0$.  To construct  $ w $ for fixed  $ \tau$ with $\pi/2  \leq \tau < \pi$, 
 let   $  w_l   $ be the  continuous   function   in  $  \bar B (0, 2l) $    
 with  $ w_l $   an   $  \mathcal{ \ti A} =  \nabla  \ti f$-harmonic   function in 
   $  B ( 0,  2 l)  \sem  [( \rn{2} \sem D ( \tau )  ) \cap  \bar B ( 0, l ) ] $ and      
 $ w_l  \equiv  0  $  on  $    ( \rn{2} \sem D ( \tau )  ) \cap  \bar B ( 0, l )  $ 
 while $ w_l  =  M_l $ on  $ \ar B ( 0,  2l)  $  where $ M_l $  is chosen so that  $ w_l ( 1, 0 ) = 1.   $   
 Using Lemmas \ref{lemma2.1}-\ref{lemma2.3} and taking limits of  a  certain 
 subsequence of  $ \{ w_l\}_{l\geq 1} $, we  see there exists $ w  \geq 0,$   
 a H\"{o}lder  continuous function on $  \mathbb{R}^2 $  which is                  
     $  \mathcal{ \ti A}  =  \nabla  \ti  f$-harmonic in   $ D ( \tau )  $
 for fixed  $ p  \geq 2 $      with  $ w \equiv 0  $  on  $  \mathbb{R}^2  \sem D ( \tau )  $  
     and     $ w ( 1, 0  )  = 1. $       Uniqueness  of   $ w  $  follows  from \eqref{eqn8.10}
      applied 
     in $ D (\tau ) \cap    B (0, 2r^+)  $ to  $ w$ and $w'$  where $ w' $ has the same 
     properties  as $ w$ (and $r^+$ is as in Lemma \ref{lem8.3}). 
     Letting $ r^+ \rar  \infty$, we conclude first that $ w/w' $  is bounded in $ D ( \tau ) $ 
     and after that from H\"{o}lder    continuity of the ratio that  $ w' = w. $       
 To prove  \eqref{eqn11.20} observe  that  $ w ( t x )$ for $x  \in \mathbb{R}^2$ is 
 also  $ \mathcal{\ti A} =  \nabla  \ti  f$-harmonic  in   $ D(\tau) $  and  $ \equiv 0 $  on 
 $  \rn{2} \sem D ( \tau )  $ as follows from $p$-homogeneity of  $ \ti  f. $ 
 From  uniqueness   of  $ w $  we   conclude  that   
 $ w ( t x  )  =  w (t, 0)  w (x).$  Differentiating   this  expression  with  respect to $ t $ 
 and evaluating at $ t =  1 $  we  arrive at   
\[         
\lan x,  \nabla w ( x ) \ran =    \frac{\ar w }{\ar x_1}  (1,0)  \,  w (x)   \quad \mbox{whenever} \quad x  \in D ( \tau ). 
\] 
  Thus in polar coordinates,    
\[
r  w_r  ( r,  \he ) =    w_ r  (1,0) \,  w ( r, \he  ).  
\]
Dividing this equality by  $ r  w   ( r,  \he )  $   and integrating with respect to $ r, $  
we find  after exponentiating that   \eqref{eqn11.20} holds  with   $ \la  =   w_r  (1,0) .$   To avoid  confusion we  now write 
        $ \la ( \tau ) $  for $ \la $   in  \eqref{eqn11.20}.  
        Next we show  that       
\begin{align} 
   \label{eqn11.21}   
   \mbox{$ \la $  is  non-increasing  on $[\pi/2, \pi ) $  with   $  \la ( \pi/2) = 1 $  and  $  \la ( \pi )  = {\ds  \lim_{ \tau \rar \pi} } \la ( \tau )   \geq  1 -   1/p. $   } 
 \end{align} 
  To  prove  \eqref{eqn11.21}  let  $  K ( \tau ) $ be the compact convex set  
  $   ( \rn{2}  \sem  D ( \tau ) )  \cap  \{(x_1, x_2)\in\mathbb{R}^2:\, \,  x_1 \geq - 1 \}  $  and  let  $  \ze  = \ze  ( \cdot, \tau ) $  be the   $ \mathcal{\ti A}  =  \nabla  \ti  f$-harmonic  
Green's function  for     $ \rn{2}  \sem K ( \tau) $  with pole at $ \infty. $

 Let   $ z  $  be  a point  on  
  $  \ar  K ( \tau )  $  with polar coordinates  $  r = 1/2,  \he =  \tau.  $   
  Define    $ w^+$ and $\ze^+ $ in  $ B ( z, 1/4 )  $  by  $ w^+  = w   $ and $\ze^+  =  \ze   $ at points $ x $  in  
  $ D ( \tau )    \cap  B (  z, 1/4 )      $ with  $ x_2 > 0 $    while  $ w^+ = \ze^+  \equiv 0  $   in  the  rest of  
 $  B (  z, 1/4 )$.    Using  Lemma  \ref{lem8.3}  in $ D( \tau ) \cap B (z, 1/4 )  $    to 
    compare  $ w^+/\ze^+ $   and then  Harnack's  inequality  we  see that 
  $   w  ( 1/2,  \he  )  \approx  \ze ( 1/2,  \he ) $  for  $ 0 \leq  \he  < \tau   $  
  where ratio  constants depend only on the data so are independent   of  $ \tau  \in [\pi/2,  \pi).$       
  Applying the same argument  below the $ x_2 $  axis  we  get the  above  ratio  
  for  $ r= 1/2$ and $0 \leq |\he| < \tau.$   So by the maximum principle for 
  $\mathcal{\ti A}$-harmonic  functions we have   
\begin{align}
  \label{eqn11.22}   
  w ( x  )  \approx  \ze  ( x) \quad \mbox{for}\, \, x  \in D (  \tau )  \cap \bar  B ( 0, 1/2)
\end{align}  
where ratio constants depend only on the data.    
From  \eqref{eqn11.22}  and the maximum principle for 
$ \mathcal{\ti A}$-harmonic functions, we  have  with $ x_1 = r$, $x_2 = 0$,  $ 0 \leq r \leq 1/2, $  and
   $ \pi/2 \leq \tau_1 < \tau_2 < \pi, $    
\begin{align} 
\label{eqn11.23} w ( r, 0, \tau_1 )  \approx   \ze ( r, 0, \tau_1 )  <  \ze ( r, 0, \tau_2 ) 
\approx   w ( r, 0, \tau_2 ). 
\end{align}
Letting $ r \to 0 $  it follows that necessarily $ \la ( \tau_2 )  
\leq \la ( \tau_1 ).$   Also  $ w ( \cdot, \pi/2 )  = x_1$ is 
$\mathcal{\ti A}$-harmonic in $\{(x_1, x_2)\in\mathbb{R}^2;\, \, x_1>0\}$ 
with continuous zero boundary value on $\{x_1=0\}$ and has value $1$ at $(1,0)$. 
From the uniqueness of $w$, we conclude that $\la(\pi/2)=1$
To prove  the right hand inequality in   \eqref{eqn11.21}  
we first observe from comparison of     $ w^+ $ to  a linear 
function  that  vanishes at points with polar coordinate  $ \he  =  \tau $   and also to  
one  that  vanishes  at  points  with  coordinate  $  \he  =  - \tau, $    as well as  
 the   argument used   in  the derivation of  \eqref{eqn11.7}  
that  
\begin{align}
\label{eqn11.24} 
c^{-1}  \leq   | \nabla  w ( 1/2, \pm\tau ) | \leq  c 
\end{align}
 where $ c $  depends only on the structure assumptions for $ f $ and $p$  so  is independent of $ \tau.  $   
 Using this fact, \eqref{eqn8.9}$(a)$, and the homogeneity of $ w $  we see that 
\begin{align} 
\label{eqn11.25} 
\infty  >   \int_{ \ar D ( \tau )  \cap  B ( 0, 1/2)}  \ti f (  \nabla  w  ( \cdot, \tau) )  d \mathcal{H}^{1} 
\approx   \int_0^{1}   r^{(\la (\tau)  - 1 ) p}  dr 
\end{align}
 where ratio constants are independent of 
$ \tau \in (3\pi/4,  \pi) .$  Clearly,   \eqref{eqn11.25}  implies   the exponent in 
the integral is larger than $ -  1 $  so  $ \la (\tau) > 1 - 1/p.$ 
Letting  $ \tau \to \pi $  we get the right hand inequality in \eqref{eqn11.21}.   
Next  using   \eqref{eqn2.1}-\eqref{eqn2.4},  \eqref{eqn11.20}, and  Ascoli's  theorem 
 we see that  a  subsequence of  $ \{ w ( \cdot, \pi - 1/m) \} $ 
converges to  $v$ which is  continuous on $ \rn{2} $  with $v\equiv 0$ on  $\{(x_1,0)\in \mathbb{R}^2:\, \,  x_1\leq 0\}$, $ v ( 1, 0) = 1$, and  
$\mathcal{\ti A}$-harmonic in $\mathbb{R}^2\setminus\{(x_1, 0)\in \mathbb{R}^2:\, \,  x_1\leq 0\}$    satisfying   
\begin{align}
\label{eqn11.26}   
v ( tx ) =  t^{ \la ( \pi )}  v ( x )  \quad  \mbox{where}\, \,  \la ( \pi )   \geq 1 - 1/p. 
\end{align}
To  show   $ \la ( \pi ) = 1 - 1/p  $  we let  $  0 < \de    < 10^{-10}$  be  a small  
but fixed  positive number.   Also  $ \ep > 0,  $   $ 0 <  \ep  < <  \de^{100} $ is   allowed  to vary. 
 We put   $ \tau  =  \pi  - \ep   $  and   write $ \ze$ and $w$  for  
$   \ze  (  \cdot,  \tau )$ and $w ( \cdot, \tau ). $     
If     $ p =2 $  we note from    \eqref{eqn9.37}  and  existence  for discrete measures in Theorem \ref{mink} as well as \eqref{eqn10.2} that       
\begin{align}
\label{eqn11.27} 
\begin{split}
 \ga   &=  {  \ds 2 \int_{ \ar K(\tau)}  
\lan y, \nabla \ze (y) \ran  \ti f ( \nabla \ze ( y))   |  \nabla \ze ( y ) |^{-1}  d \mathcal{ H }^{1}}  \\
&  = 2  {\ds \int_{ \ar K(\tau)} \lan y + (1, 0), \nabla \ze (y) \ran  \ti f ( \nabla \ze ( y))   |  \nabla \ze ( y ) |^{-1}  d \mathcal{ H }^{1}}.
\end{split}
\end{align}
While if  $ p  >  2 $ we get  from the same reasoning and   \eqref{eqn9.38}  that  
\begin{align} 
  \label{eqn11.28}  
\begin{split}  
   1  & \approx  \frac{p - 2}{ p - 1} \mathcal{C}_{\mathcal{\ti A}}  (  K  ( \tau ) )^{1/(p-1)}  \\
   & = p \int_{ \ar K ( \tau ) }   \lan y + (1, 0 ), \nabla \ze (y) \ran \ti  f ( \nabla \ze ( y ))   |  \nabla \ze ( y ) |^{-1}  d \mathcal{ H }^{1} 
\end{split}
\end{align}   
Let  $  J ( \tau )  =   K ( \tau ) \cap   \{ y  :  y_1   \geq   - 1  +  \de  \}  $   and let   
$  q ( \cdot ) $ be the  $ \mathcal{\ti A}$-harmonic Green's  function for  the  complement of 
$ I  ( \tau)  =  K ( \tau )   \cap  \{ x :  x_1   \leq     -  1 + 3 \de \}$ with a pole at infinity.  We note that  
$ q ( \cdot ) \geq  \ze ( \cdot ) $ in  $ \rn{2}. $ If $ p >  2 $   then from   the  Hopf  boundary maximum  principle we deduce  
\begin{align}
\label{eqn11.29} 
\begin{split}
   { \ds \int_{ \ar K(\tau) \sem J  (\tau)}} & {\ds 
\lan y + (1, 0),\nabla \ze (y) \ran  \ti f ( \nabla \ze ( y ))   |  \nabla \ze ( y ) |^{-1}  d \mathcal{ H }^{1}}  \\  
& \leq  
   { \ds \int_{ \ar  I(\tau) }  
\lan y + (1, 0), \nabla q  (y) \ran  \ti f ( \nabla q  ( y ))   |  \nabla q ( y ) |^{-1}  d \mathcal{ H }^{1}} \\
&  \leq c  \de^{\frac{p-2}{p-1}}  
\end{split}
\end{align}
  where  $ c $  depends only on the data. Here we once again used the existence  for discrete measures in Theorem \ref{mink}, \eqref{eqn10.2}, and Remark \ref{remark10.3}.
If  $ p  =  2 $  let  $ z_0 $  be the point  with   polar coordinate  $ r   =  1 + \de$ and $\he =  \pi.  $     
Applying  the  same  argument  as in the derivation of  \eqref{eqn11.22}       we obtain   
\begin{align}  
\label{eqn11.30}  
\ze ( x )/q (x )  \approx   \ze( z_0 )/ q ( z_0 )   \quad
\mbox{when $ x \in  D (\tau)$  and}\, \,  |  x   +  ( 1, 0 )  |    <   2 \de.  
\end{align}
From  \eqref{eqn2.6} $(ii)$    we observe that    $ q   ( z_0 )  \approx 1.  $   Also  since  
 $  \ar K ( \tau ) $ is uniformly  $ (1, 2 )$-fat  we  conclude from   \eqref{eqn2.1} 
 $(iii)$  that there exists $  \ti \si $  with  
 $ \ze ( z_0)  \leq     \de^{\ti \si}. $  Using these inequalities,   \eqref{eqn11.30},    \eqref{eqn11.27} 
 with $ \ze $  replaced by  $q,$    and  the  Hopf  boundary maximum   principle  (as in   \eqref{eqn11.29}) 
 it follows that  if $ p = 2, $  
\begin{align}  
 \label{eqn11.31}
     { \ds \int_{ \ar K(\tau) \sem J  (\tau)}  
\lan y + (1, 0), \nabla \ze (y) \ran \ti f ( \nabla \ze ( y ))   |  \nabla \ze ( y ) |^{-1}  d \mathcal{ H }^{1}}  \leq  c   \de^{\ti \si}  
\end{align}
 where $ c$ and  
$  \ti  \si  $  depend only on the data.   To finish the proof   of  Lemma  \ref{lemma11.3} 
we use \eqref{eqn11.29},  \eqref{eqn11.31},   \eqref{eqn11.27},  \eqref{eqn11.28},     and the  fact  that  $  \lan x + (1,0), \nabla \ze \ran      
=  \sin(\ep)  |  \nabla  \ze  |  $   on  $  \ar J ( \tau ) $  to first  get  for  $  \de $  sufficiently small  that 
\begin{align} 
\label{eqn11.32}  
 c^{-1}  \leq   \ep  \int_{ \ar J ( \tau ) \cap \ar K ( \tau)  }  \ti   f ( \nabla \ze ( y ))     d \mathcal{ H }^{1} .  
\end{align} 
  Second as in \eqref{eqn11.22}   we see that  
    $  \ze    \leq c (  \de )   w  $  in  $   D ( \tau )\cap  B  ( 0, 1 - \de/2). $   Using this  inequality  and 
    the  Hopf boundary maximum principle a  final time we find in view  of  \eqref{eqn11.24} and \eqref{eqn11.32}  that
\begin{align}
\label{eqn11.33}  
\begin{split}
 c^{-1}  &\leq  c ( \de )   \ep  
    {  \ds \int_{ \ar J ( \tau ) \cap \ar K ( \tau) }  \ti  f ( \nabla w  ( y ))    d \mathcal{ H }^{1}   } \\ 
    & \leq   c \,  c ( \de ) \ep { \ds\int_0^{1} r^{(\la (\tau)  - 1 ) p}  dr}   \\
    &= {\ds  \frac{ c\, c ( \de )  \ep }{ ( \la ( \tau )  - 1) p + 1}    } 
\end{split}
\end{align}    
Letting    $ \ep \to 0 $  we   conclude    that    $ (\la ( \tau )  - 1) p + 1   \rar 0 $  as $ \tau  \to \pi    $ so $  \la ( \pi ) = 1  -  1/p.$  
\end{proof} 
  
 Armed with  Lemma  \ref{lemma11.3} we  now 
  can prove \eqref{eqn11.18} and so finish the  proof of Proposition \ref{proposition11.2}. 
         We use the same  notation as in  Proposition \ref{proposition11.1} except we  
  assume  $  E  \subset P =   \{ x :  x_2 = 0  \}. $  Let    $ \{ Q_k \}  $  be   a  
  Whitney decomposition of  the interior of $ E $  considered as  a  subset of  $ P. $   
  Let  $ Q  \in  \{  Q_k \},   $   and  let $ z = (z_1, 0, z_3, \dots, z_n) $ be a  point in  $ \ar'  E $ 
   with  $ d'  ( z,    Q )  \approx  s ( Q). $   From convexity of  $  E  $  we see that  
   there is  a  $(n - 2)$-dimensional  plane, say  $ P_1  $  containing $ z $ 
   with the property that  $  E $  is contained in the closure of  
one of the components  of    $ P \sem  P_1$.    Rotating $  P_1 $ if necessary we may assume that   
\[    
P_1 = \{ x\in\mathbb{R}^n:\, \,  x_1 = z_1,  x_2 = 0 \} \quad \mbox{and}  \quad E  \subset  \Om = \{   x \in \mathbb{R}^n  : \,\,  x_1  - z_1 \leq  0  \, \, \mbox{ and  } \, \,  x_2 = 0   \}.   
\]   
Let  $ v $  be as in Lemma \ref{lemma11.3}.    We   extend $v $  continuously to  
$ \mathbb{R}^n $ (also denoted $v$)  by  defining  this  function to   
be constant in the other  $(n - 2)$ coordinate directions.  
Then  $ \hat v (x) = v (x-z)$ for   $x  \in \rn{n}$ is 
$ \mathcal{A}$-harmonic in $\mathbb{R}^n  \sem \Om. $       
Comparing boundary values  and using the maximum  principle, 
as  well as   Lemmas \ref{lem5.2} and \ref{lem5.3}   we  deduce that 
\begin{align}
\label{eqn11.34}     
 C  U ( x )  \geq  \hat v ( x )   \quad \mbox{whenever}\, \,   x  \in    B (0,  2 \rho  ),  
\end{align}   
 where  $ C $  depends only on  the data and    $ \mathcal{C}_{\mathcal{A}} ( E ). $    
 As  in  Proposition \ref{proposition11.1}    we see   that    
\begin{align} 
\label{eqn11.35} 
c'   \int_Q    f(\nabla  U_+)   \, d\mathcal{H}^{n-1}  \geq    U (y_Q) )^p  s (Q)^{n-1-p}.
\end{align} 
Now from Lemma \ref{lemma11.3}  we also deduce that 
\begin{align} 
\label{eqn11.36}  
c'' \hat v  (y_Q) \geq  s (Q)^{1-1/p}.
\end{align}  
Combining  \eqref{eqn11.34}-\eqref{eqn11.36}  we conclude that  for some  $ \ti c  $  with the same dependence as the above constants, 
\begin{align} 
\label{eqn11.37}    \ti c   \int_Q   f (\nabla  U_+ )   \, d\mathcal{H}^{n-1}  \geq   
  s (Q)^{n-2}.
  \end{align}
Now since  $  B (0, 4a)  \cap P \subset  E  $  we see for  $ l  $ large from  Lipschitz  starlikeness of  $ E$  that there are  at least $ \approx  2^{ l (n -2)}  $  members of  $ \{Q_k \} $ whose side length lies between  $  2^{- l-1} a $ and  $ 2^{-l} a. $     Using this fact  and  summing \eqref{eqn11.37}   we get   Proposition \ref{proposition11.2}.   

\subsection{Uniqueness in Theorem \ref{mink}}
\label{uniq}
Uniqueness in Theorem \ref{mink}  follows  from  Theorem \ref{theoremA}  as in \cite{CNSXYZ}  (see also \cite{AGHLV}).    To give the reader the gist of  the argument we consider  the simplest  $ p = n$  case.  Suppose  $  \mu $  
is a positive  finite   Borel measure on  $  \mathbb{S}^{n-1}$  satisfying   \eqref{eqn7.1}  and  let  $ E_0$ and  $E_1$ be two compact convex sets  with nonempty interiors for which the  corresponding  $ \mathcal{A}$-harmonic  Green's functions   satisfy\eqref{eqn7.2} $ 
(a)-(c) $  in Theorem \ref{mink}  relative to $ \mu. $       Let $h_0$ and $h_1$ be the support functions of $E_0$ and $E_1$ respectively. For $t\in[0,1]$ we let $ E_t=(1-t) E_0+t  E_1$  and put     $ \mathbf{m}(t) =     \mathcal{C}_{\mathcal{A}}(E_t)$ for $t \in [0,1]$.      
     Using     \eqref{eqn10.28} and \eqref{eqn9.37},  we deduce   that   
\begin{align}
\label{eqn11.39}
\begin{split}
\left.\frac{d}{dt}  \mathbf{m}(t)   \right|_{ t = 0 } \,  &= n\gamma^{-1} \mathcal{C}_{\mathcal{A}}(E_0) \int_{\mathbb{S}^{n-1}} (h_1(\xi)-h_0(\xi)) d\mu(\xi) = 0. 
\end{split}
\end{align}
From Theorem  \ref{theoremA} we see that $ m $ is concave on  $[0,1]. $  Using this fact  \eqref{eqn11.39} we find that 
 \begin{align}  
\label{eqn11.40} 
\mathbf{m}'(0) = 0  \geq   \mathbf{m}(1)  -  \mathbf{m}(0).
\end{align}  
with   equality  only  if   $  m $ is  constant on  [0,1].     From \eqref{eqn11.40} we first get    
 $\mathcal{C}_{\mathcal{A}}(E_0) = \mathbf{m}(0)\geq \mathbf{m}(1)=\mathcal{C}_{\mathcal{A}}(E_1)$ and second  by reversing the roles of  $ E_0$ and $E_1,  $  we  get  
  $ \mathcal{C}_{\mathcal{A}}(E_1) = \mathbf{m} ( 1)     =   \mathbf{m} ( 0) =\mathcal{C}_{\mathcal{A}}(E_0)$.   We conclude that $\mathbf{m}$ is constant and therefore equality holds in \eqref{BMn} of  Theorem \ref{theoremA} and so   $ E_0$ is a translation and dilation of $ E_1$.   Using Remark \ref{rmk5.5} and the fact that $\mathcal{C}_{\mathcal{A}}(E_0)=\mathcal{C}_{\mathcal{A}}(E_1)$ we see that honest dilations  are not possible    when $p= n$. For $p>n$, by considering $\mathbf{m}(t)=C_\mathcal{A}(E_t)^{1/(p-n)}$, a similar argument can be used to prove uniqueness of $E$ up to translation. This finishes our outline of the proof of uniqueness in  Theorem \ref{mink} for $n\geq p$. 
\section*{Acknowledgment}
This material is based upon work supported by National Science Foundation under Grant No. DMS-1440140 while the first and the third authors were in residence at the MSRI in Berkeley, California, during the Spring 2017 semester. The second author was partially supported by NSF DMS-1265996. The third author was partially supported by the Hausdorff Center for Mathematics as well as DFG-SFB 1060. Research for this article was carried out while the first and fourth author were visiting the department of mathematics at the University of Kentucky, the authors would like to thank  the department for its hospitality.
\newcommand{\etalchar}[1]{$^{#1}$}

      \end{document}